\documentclass[11pt]{article}

\usepackage[final]{graphicx}
\usepackage{color}
\usepackage{amsmath}
\usepackage{amssymb}
\usepackage{mathtools}
\usepackage{amsthm}
\usepackage{mathrsfs}
\usepackage{stmaryrd}
\usepackage[symbol,perpage]{footmisc}
\usepackage[sort]{natbib}
\usepackage{enumerate}
\usepackage{nth}
\usepackage[np,autolanguage]{numprint}
\usepackage{ifthen}
\usepackage{afterpage}
\usepackage[paper = letterpaper, centering,
            lmargin = 1in, rmargin = 1in,
            tmargin = 1in, bmargin = 1in]{geometry}
\usepackage{booktabs}
\usepackage{multirow, array}
\usepackage{xspace}
\usepackage{url}
\usepackage{acronym}
\acrodef{FSI}[FSI]{fluid structure interaction}
\acrodef{CSF}[CSF]{cerebrospinal fluid}
\acrodef{ALE}[ALE]{arbitrary Lagrangian-Eulerian}
\acrodef{FEM}[FEM]{finite element method}
\acrodef{FE}[FE]{finite element}
\acrodef{DAE}[DAE]{differential-algebraic equations}
\acrodef{BDF}[BDF]{backward differentiation formulas}

\theoremstyle{plain}
\newtheorem{theorem}{Theorem}
\newtheorem{lemma}{Lemma}

\theoremstyle{definition}

\newtheorem{problem}{Problem}
\theoremstyle{remark}
\newtheorem{remark}{Remark}

\DeclareMathOperator{\ldiv}{\nabla_{\!\bv{x}}\,\cdot}
\DeclareMathOperator{\udiv}{\nabla_{\!\bv{X}_{\s}}\,\cdot}
\DeclareMathOperator{\grad}{\nabla_{\!\bv{x}}}
\DeclareMathOperator{\Grad}{\nabla_{\!\bv{X}_{\s}}}

\DeclareMathOperator{\PSym}{PSym}

\DeclareMathOperator{\composition}{\text{\scriptsize$\circ$}}
\DeclareMathOperator{\Span}{span}

\newcommand{\lnth}[1]{$#1$th}
\newcommand{\ui}{\ensuremath{\hat{\imath}}}
\newcommand{\uj}{\ensuremath{\hat{\jmath}}}
\newcommand{\nsd}[1]{\ensuremath{\mathrm{d}{#1}}}

\newcommand{\bv}[1]{\boldsymbol{#1}}

\newcommand{\tensor}[1]{\mathsf{#1}}
\newcommand{\transpose}[1]{#1^{\mathrm{T}}}
\newcommand{\inversetranspose}[1]{\ensuremath{#1^{-\mathrm{T}}}}

\newcommand{\f}{\mathrm{f}}
\newcommand{\s}{\mathrm{s}}
\newcommand{\colondot}{\mathbin{:}}
\newcommand{\test}[1]{\ensuremath{\tilde{#1}}}
\newcommand{\ucdot}{\ensuremath{\!\cdot\!}}
\newcommand{\sap}{\ensuremath{1}}       
\newcommand{\sbp}{\ensuremath{2}}       
\newcommand{\scp}{\ensuremath{3}}       
\newcommand{\sdp}{\ensuremath{4}}       
\newcommand{\sad}{\ensuremath{\bar{1}}} 
\newcommand{\sbd}{\ensuremath{\bar{2}}} 
\newcommand{\scd}{\ensuremath{\bar{3}}} 
\newcommand{\sdd}{\ensuremath{\bar{4}}} 

\newcommand{\us}{\ensuremath{\bv{u}_{\s}}}
\newcommand{\vs}{\ensuremath{\bv{v}_{\s}}}
\newcommand{\vf}{\ensuremath{\bv{v}_{\f}}}
\newcommand{\fv}{\ensuremath{\bv{v}_{\text{flt}}}}
\newcommand{\ush}{\ensuremath{\bv{u}_{\s_{h}}}}
\newcommand{\vsh}{\ensuremath{\bv{v}_{\s_{h}}}}
\newcommand{\vfh}{\ensuremath{\bv{v}_{\f_{h}}}}
\newcommand{\fvh}{\ensuremath{\bv{v}_{\text{flt}_{h}}}}

\newcommand{\comsol}{COMSOL Multiphysics\textsuperscript{\textregistered}\xspace}

\title{%
Finite Element Formulation for a Poroelasticity Problem Stemming from Mixture Theory\footnote{Submitted for publication in Computer Methods in Applied Mechanics and Engineering.}
}

\author{%
Francesco Costanzo\footnote{Corresponding Author: \texttt{<costanzo@engr.psu.edu>}; Tel.: +1 814 863-2030; Fax: +1 814 865-9974}\\
Center for Neural Engineering\\
The Pennsylvania State University\\
W-315 Millennium Science Complex\\
University Park,
PA 16802
USA
\and
Scott T.~Miller\footnote{\texttt{<stmille@sandia.gov>};Tel.: +1 505 845-0487.}\\
Computational Simulation Group\\
Sandia National Laboratories\\
Albuquerque, NM 87185\\
USA}

\begin{document}

\maketitle
               
\begin{abstract}
A finite element formulation is developed for a poroelastic medium consisting of an incompressible hyperelastic skeleton saturated by an incompressible fluid.  The governing equations stem from mixture theory and the application is motivated by the study of interstitial fluid flow in brain tissue.  The formulation is based on the adoption of an \ac{ALE} perspective. We focus on a flow regime in which inertia forces are negligible. The stability and convergence of the formulation is discussed, and numerical results demonstrate agreement with the theory.
\end{abstract}


\allowdisplaybreaks{                             %


\section{Introduction}
A central problem in brain physiology is the transport of metabolites produced by cell functions in brain tissue from their production site to the main \ac{CSF} compartment \citep{Iliff2012A-Paravascular-0,Iliff2013Brain-wide-Path0,Iliff2013Cerebral-Arteri-0,Iliff2013Is-There-a-Cere-0}.  The modeling of these transport phenomena has traditionally focused on Fickian diffusion within the extracellular space \citep{Sykova2004Diffusion-Prope0,Sykova2008Diffusion-in-Br0,Vargova2011Glia-and-Extrac0,Vargova2011The-Diffusion-P0} (see also \citealp{Gevertz2008A-Novel-Three-P0}).  More recently, the studies by Iliff and co-workers \citep{Iliff2012A-Paravascular-0,Iliff2013Brain-wide-Path0,Iliff2013Cerebral-Arteri-0,Iliff2013Is-There-a-Cere-0} point to the existence of pathways for metabolite exchange with significant convective transport.  Furthermore, evidence indicates that such convective component is driven by the pulsatile motion of arterial walls along the various elements of the brain vascular tree.

The coupling between transport and mechanical properties is a fundamental aspect of the design of tissue engineered scaffolds, especially for application in the regeneration of peripheral nerves \citep{Saracino2013Nanomaterials-D0,Dey2008Development-of-0,Dey2010Crosslinked-Ure0,Nguyen2015Tissue-Engineer-0}.  In these applications, it is essential to coordinate the evolution of the transport and mechanical properties with the rate of degradation of the material.  Modeling of these systems requires a framework for the system's poroelastic behavior along with the reaction-diffusion physics of the degradation process.

Motivated by these these problems, we considered transport models that could simultaneously include both diffusive and convective components.  The models must also have the ability to characterize the flow field of specific molecular constituents and their interaction with large deformations.  Finally, we needed models that can be expanded to include chemical reactions.. For these reasons, we focused on a model of transport in poroelastic media based on mixture theory \citep{Bowen1976Theory-of-Mixtu0,Bowen1980Incompressible-0,Bowen1982Compressible-Po0,Rajagopal1995Mechanics-of-Mixtures-0}.

There is a significant body of literature on numerical methods for flow in porous media and poroelasticity (see, e.g., \citealp{SelvaduraiMechPoro1996,Lewis1998The-Finite-Element-0,Armero1999Formulation-and-Finite-0,Coussy2010Mechanics-and-P0}), but we found limited work when such phenomena are studied from a mixture theory perspective, i.e., by considering the simultaneous existence of multiple independent velocity fields, as they appear in mixture theory.  In this paper we present our experience in formulating a mixed \ac{FEM} for a simple mixture consisting of an incompressible hyperelastic solid skeleton saturated by an incompressible fluid.  Our focus is to determine the fluid flow in addition to the filtration velocity.  In future developments, we plan to combine the porous flow problem with a convection-reaction-diffusion component for the species present in the fluid.  Hence, our quantities of interest are the pore pressure, the velocity fields of the fluid and solid phases, the filtration velocity, as well as the displacement field of the solid phase. The incompressibility of the constituents yields a constraint equation, which, when expressed in the body's current configuration, requires the volume-fraction-average of the solid and fluid velocity fields to be divergence-free.  Perhaps the most delicate aspect of the enforcement of this condition is the fact that the action of the divergence operator affects the porosity field in addition to the velocities.  From an experimental viewpoint, this field might have significant uncertainty in its determination.  To avoid dealing with gradients of the porosity we have considered two strategies: (\emph{i}) selecting four fields as primary unknowns, namely the solid displacement, solid velocity, fluid velocity, and pore pressure, along with weakening the constraint equation using the divergence theorem; and (\emph{ii}) modifying this formulation by replacing the fluid velocity with the filtration velocity as a primary field.  In the first approach the filtration velocity can be determined during post-processing as a simple $L^{2}$-projection.  Fluid velocity is recovered in an analogous manner
in the second approach.  We will discuss the advantages and disadvantages of both strategies.  As it turns out, the second strategy offers a coercivity property that is stronger than the first and can possibly justify such an approach in practice.  The first strategy does not appear to work well unless properly stabilized.  We have implemented  stabilization in both strategies for the quasi-static motion of a nonlinear poroelastic body.  In particular, we have adapted to the present context  the stabilization strategy demonstrated by \cite{Masud2002A-Stabilized-Mixed-0} (see also \citealp{Masud2007A-Stabilized-Mixed-0}).

We point out that there are various flow regimes of interest.  Specifically, there are conditions in which it is reasonable to assume that inertia effects are negligible and that therefore the evolution of the system is quasi-static.  On the other hand, there are some flow regimes in which it is desirable to account for inertia effects.  With this in mind, the work in the paper is limited to the analysis of the quasi-static approximation of the problem.  The fully dynamic case will be considered in future publications.

\section{Basic definitions and governing equations}
\label{sec: PP Governing equations}

\subsection{Configurations, motions, volume fractions, and incompressibility}
\label{subsec: configurations volume fractions}
With reference to Fig.~\ref{fig: configurations},
\begin{figure}[htb]
    \centering
    \includegraphics{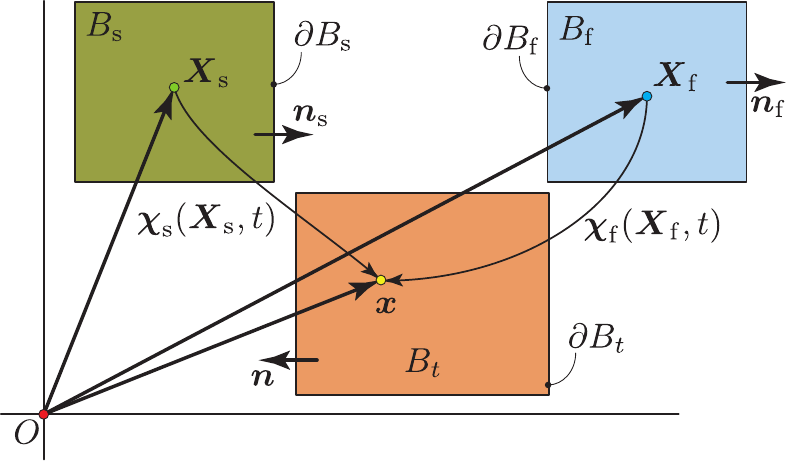}
    \caption{Current configuration $B_{t}$ of a binary mixture consisting of a solid skeleton and a fluid.  $B_{t}$ is the common image of the two motions $\bv{\chi}_{\s}$ and $\bv{\chi}_{\f}$ of the solid skeleton and the fluid, respectively.}
    \label{fig: configurations}
\end{figure}
we consider the deformed configuration of a porous body $B_{t}$ saturated by an incompressible fluid.  The solid phase is assumed to be hyperelastic and incompressible.  To describe the motion of the system, we adopt the theoretical setting in \cite{Bowen1980Incompressible-0} (see also \citealp{Bowen1976Theory-of-Mixtu0}).  Hence, $B_{t}$ is the common image of two diffeomorphisms: $\bv{\chi}_{\s}: B_{\s} \times [0,T] \to B_{t} \subset \mathscr{E}^{d}$ and $\bv{\chi}_{\f}: B_{\f} \times [0,T] \to B_{t} \subset \mathscr{E}^{d}$, where $B_{\s}$ and $B_{\f}$ are the reference configurations of the solid and fluid phases, respectively, $t \in [0,T]$, with $[0,T]$ a chosen time interval of interest.  By $\mathscr{E}^{d}$, with $n = 2, 3$, we denote the $d$-dimensional Euclidean point space and by $\mathscr{T}^{d}$ its companion translation vector space. The subscripts `$\s$' and `$\f$' stand for `solid' and `fluid', respectively.  Points in $B_{\s}$ and $B_{\f}$ will be denoted by $\bv{X}_{\s}$ and $\bv{X}_{\f}$, respectively, whereas $\bv{x}$ denotes position in $B_{t}$. $\partial B_{\s}$, $\partial B_{\f}$, $\partial B_{t}$ are the boundaries of $B_{\s}$, $B_{\f}$ and $B_{t}$, respectively, oriented by corresponding outward unit normal fields $\bv{n}_{\s}$, $\bv{n}_{\f}$, and $\bv{n}$.

For simplicity, the initial configurations for the solid and fluid are made to coincide with the initial configuration of the mixture, so that
\begin{equation}
\label{eq: initial configurations}
B_{\s} = B_{\f} = B_{t}\big|_{t = 0}.
\end{equation}

For each constituent $a = \s,\f$, the displacement, deformation gradient, and Jacobian determinant of the motions are, respectively,
\begin{equation}
\label{eq: us Fs and Js}
\bv{u}_{a}(\bv{X}_{a},t) \coloneqq \bv{\chi}_{a}(\bv{X}_{a},t) - \bv{X}_{a},
\quad
\tensor{F}_{a}(\bv{X}_{a},t) \coloneqq \frac{\partial \bv{\chi}_{a}(\bv{X}_{a},t)}{\partial \bv{X}_{a}},
\quad
J_{a}(\bv{X}_{a},t) \coloneqq \det{\tensor{F}_{a}(\bv{X}_{a},t)},
\end{equation}
where, $\forall \bv{X}_{a} \in B_{a}$ and $\forall t \in [0,t]$, $J_{a}(\bv{X}_{a},t) > 0$.

The spatial (or Eulerian) representation of the material velocity of constituent $a$ ($a = \s,\f$) is
\begin{equation}
\label{eq: material velocities}
\bv{v}_{a}(\bv{x},t) \coloneqq \frac{\partial \bv{\chi}_{a}(\bv{X}_{a},t)}{\partial t}\bigg|_{\bv{X}_{a} = \bv{\chi}_{a}^{-1}(\bv{x},t)}
=
\frac{\partial \bv{u}_{a}(\bv{X}_{a},t)}{\partial t}\bigg|_{\bv{X}_{a} = \bv{\chi}_{a}^{-1}(\bv{x},t)}
,\quad\forall\bv{x} \in B_{t}.
\end{equation}
We note that, in principle, it is possible to express the displacement field in Eulerian form, that is by writing
\begin{equation}
\label{eq: displacement Eulerian}
\bv{u}_{a}(\bv{x},t)\coloneqq \bv{x} - \bv{\chi}_{a}^{-1}(\bv{x},t),
\end{equation}
where, with a slight abuse of notation, we are using the symbol $\bv{u}_{a}$ to refer to the displacement field whether described in Lagrangian or Eulerian form and relying on the context to resolve the ambiguity.  With this in mind, for future reference we note that
\begin{equation}
\label{eq: Eulerian disp J rel}
\tensor{F}_{\s}(\bv{X},t)\big|_{\bv{X} = \bv{\chi}_{\s}^{-1}(\bv{x},t)} = (\tensor{I} - \grad \bv{u}_{\s})^{-1},
\end{equation}
where $\grad$ is the gradient with respect to $\bv{x}$ over $B_{t}$.  Denoting the material time derivative following the motion of phase $a$ by $D_{t_{a}}$, recalling that $\bv{v}_{\s}$ is the material time derivative of $\bv{u}_{\s}$, in an Eulerian context, for all $\bv{x} \in B_{t}$, we must have
\begin{equation}
\label{eq: us vs relation Eulerian}
D_{t_{\s}} \bv{u}_{\s} = \bv{v}_{\s}
\quad\Rightarrow\quad
\partial_{t}\bv{u}_{\s} + (\grad\bv{u}_{\s}) \bv{v}_{\s} = \bv{v}_{\s}
\quad\Rightarrow\quad
\bv{v}_{\s} = (\tensor{I} - \grad \bv{u}_{\s})^{-1} \partial_{t}\bv{u}_{\s},
\end{equation}
which, again, requires the invertibility of $\tensor{I} - \grad \bv{u}_{\s}$.  For future reference, we note that the Lagrangian expression of the above relation is: for all $\bv{X}_{\s} \in B_{\s}$ and $t \in (0,T)$, we must have
\begin{equation}
\label{eq: us vs relation Lagrangian}
D_{t_{\s}} \bv{u}_{\s} = \bv{v}_{\s}
\quad\Rightarrow\quad
\partial_{t}\bv{u}_{\s}(\bv{X}_{\s},t) = \bv{v}_{\s}(\bv{x},t)\big|_{\bv{x} = \bv{\chi}_{\s}(\bv{X}_{\s},t)}.
\end{equation}

The spatial representation of the volume fraction of constituent $a$ ($a = \s,\f$) is denoted by $\xi_{a}(\bv{x},t)$ and such that
\begin{equation}
\label{eq: volume fraction assumption}
0 < \xi_{a} < 1,\quad a = \s,\f.
\end{equation}
After the theory is formulated, we will consider the limit cases for $\xi_{\f} \to 0$ and $\xi_{\f} \to 1$.  As the solid has been assumed to be fluid-saturated, we have
\begin{equation}
\label{eq: volume fraction constraint}
\xi_{\s}(\bv{x},t) + \xi_{\f}(\bv{x},t) = 1,
\quad\forall\bv{x} \in B_{t}.
\end{equation}
The spatial representation of the mass density of the \lnth{a} constituent ($a = \s,\f$) is
\begin{equation}
\label{eq: ath mass density}
\rho_{a}(\bv{x},t) = \xi_{a}(\bv{x},t) \rho_{a}^{*}(\bv{x},t),
\quad\forall\bv{x} \in B_{t}.
\end{equation}
where $\rho_{a}^{*}(\bv{x},t)$ is the \emph{intrinsic mass density} of the \lnth{a} constituent (mass of constituent $a$ per unit volume actually occupied by $a$). The material time derivative of $\rho_{a}^{*}$ is equal to zero due to incompressibiity \citep{Bowen1980Incompressible-0,Gurtin-CMBook-1981-1,GurtinFried_2010_The-Mechanics_0}, that is,
\begin{equation}
\label{eq: incompressibility}
\partial_{t}\rho_{a}^{*} + \grad \rho_{a}^{*} \cdot \bv{v}_{a} = 0,
\end{equation}
where $\partial_{t}$ is the partial derivative with respect to $t$. Recalling that, in the absence of chemical reactions, the balance of mass demands that $\partial_{t}\rho_{a} + \ldiv (\rho_{a} \bv{v}_{a}) = 0$, which, in turn, implies that 
\begin{equation}
\label{eq: volume fraction balance}
\partial_{t}\xi_{a} + \ldiv (\xi_{a} \bv{v}_{a}) = 0
\quad\Rightarrow\quad
\xi_{a}(\bv{x},t)\big|_{\bv{x} = \chi_{a}(\bv{X},t)}  J_{a}(\bv{X}_{a},t) = \xi_{R_{a}}(\bv{X}_{a}),~a = \s, \f,
\end{equation}
where $\xi_{R_{a}}$ is the referential volume fraction distribution of constituent $a$. Thus, we have
\begin{equation}
\label{eq: densities}
\rho_{\s} = \frac{\xi_{R_{\s}}}{J_{\s}} \rho_{\s}^{*},
\quad \text{and} \quad
\rho_{\f} = \biggl(1 - \frac{\xi_{R_{\s}}}{J_{\s}}\biggr) \rho_{\f}^{*}.
\end{equation}
\begin{remark}[Incompressibility constraint \,---\, Filtration velocity]
\label{remark: vflt functional space}
It is well known that the motion of a single incompressible constituent must satisfy the constraint that the determinant of the deformation gradient is constant and equal to one.  The motion of a saturated mixture of incompressible constituents is also subject to a constraint arising from the incompressibility of the constituents.  However, this constraint is not as easily expressed.  Specifically, we have assumed that $0 < \xi_{R_{\s}} < 1$ along with $0 < \xi_{a} < 1$ ($a = \s,\f$).  Furthermore, the saturation condition requires that $\xi_{\s} + \xi_{\f} = 1$, and the solid motion must have $J_{\s} > 0$ as well as $\xi_{\s} = \xi_{R_{\s}}/J_{\s}$.  Therefore, keeping in mind that $\xi_{R_{\s}}$ is a defining property of the reference configuration,  $J_{\s}$ need not be a constant, but it must satisfy the following inequality:
\begin{equation}
\label{eq: Js requirement}
0 < \xi_{R_{\s}} < J_{\s}.
\end{equation}
Clearly, something similar can be said about $J_{\f}$, however such inequalities do not provide enforceable constraint equations.  As it turns out, one can express the needed constraint by considering the combined effects of the balance of mass along with the saturation condition.  Specifically, taking the partial derivative with respect to time of Eq.~\eqref{eq: volume fraction constraint} and then rewriting the result using the first of Eqs.~\eqref{eq: volume fraction balance} we have
\begin{equation}
\label{eq: vavg constraint}
\ldiv \bigl(\xi_{\s} \bv{v}_{\s} + \xi_{\f} \bv{v}_{\f} \bigr) = 0\quad \text{in $B_{t}$},
\end{equation}
which implies that the volume fraction average of the velocity fields is divergence-free.  However, in general, neither $\bv{v}_{\s}$ nor $\bv{v}_{\f}$ are divergence free.  From an analytical viewpoint, Eq.~\eqref{eq: vavg constraint} presents some difficulties in that, in its current form, it would require the computation of gradients of $\xi_{\s}$ and $\xi_{\f}$, whose regularity depends by the smoothness of both the solid's motion and the datum $\xi_{R_{\s}}$.  The latter quantity must be estimated via experimental methods and it should be expected to lack sufficient smoothness.  Hence, it is useful to re-express the constraint equation in alternative ways that, from a numerical viewpoint, might offer a means not to compute gradients of $\xi_{R_{\s}}$.  Here, we reformulate our constraint equation using the \emph{filtration velocity}, which is a concept that arises naturally in the study of flow through a porous medium (see, e.g., \citealp{Bowen1982Compressible-Po0}) and that is useful in both practical and computational applications.  The filtration velocity, denoted by $\fv$, is defined as the velocity of the fluid relative to the solid scaled by the fluid's volume fraction: 
\begin{equation}
\label{eq: filtration def}
\fv \coloneqq \xi_{\f} (\bv{v}_{\f} - \bv{v}_{\s}).
\end{equation}
Using the filtration velocity, Eq.~\eqref{eq: vavg constraint} can be rewritten as
\begin{equation}
\label{eq: div free vels}
\ldiv \bigl(\vs + \fv) = 0\quad \text{in $B_{t}$}.
\end{equation}
\end{remark}
\begin{remark}[Fluid velocity vs.\ filtration velocity]
\label{remark: filtration}
When modeling the flow of an incompressible fluid within a rigid porous medium, the porosity is a constant and the solid is typically viewed as stationary.  In this context, the filtration velocity is a mere (possibly local) rescaling of the fluid velocity and, under common assumptions on the porosity, the two fields can be viewed as having the same analytical properties.  This is often the case in most applications of linear poroelasticity.  However, in a context of large deformations, the properties of the filtration and fluid velocities can be significantly different due to the fact that the porosity is also a function of the deformation of the solid skeleton.  In fact, even for our elementary flow model, the difference in question affects the coercivity of certain operators.
\end{remark}

\subsection{Constitutive assumptions and momentum balance laws}
Following \cite{Bowen1980Incompressible-0}, we assume that $\rho^{*}_{a}$ ($a = \s,\f$) is constant. Furthermore, neglecting surface tension effects, the interaction between fluid and solid phases will be described by a drag force proportional to the fluid velocity relative to the solid.  The solid phase is assumed to be hyperelastic with a strain energy per unit volume of the pure species given by 
\begin{equation}
\label{eq: strain energy}
W = W(\tensor{C}_{\s}) \quad\text{with}\quad \tensor{C}_{\s} = \transpose{\tensor{F}}_{\s} \tensor{F}_{\s}.
\end{equation}
Therefore, denoting by $\Psi(\bv{x},t)$ the strain energy density of the mixture per unit volume of the deformed configuration, we have
\begin{equation}
\label{eq: internal strain energy}
\Psi(\bv{x},t)\big|_{\bv{x} = \bv{\chi}_{\s}(\bv{X}_{\s},t)} = \frac{\xi_{R_{\s}}(\bv{X}_{\s})}{J_{\s}(\bv{X}_{\s},t)} \, W_{\s}(\tensor{C}_{\s}(\bv{X}_{\s},t)),
\end{equation}
and the elastic contribution to the overall Cauchy stress in the mixture is
\begin{equation}
\label{eq: Cauchy stress}
\tensor{T}^{e} = 
2 \frac{\xi_{R_{\s}}}{J_{\s}} \, \tensor{F}_{\s} \frac{\partial W_{\s}}{\partial \tensor{C}_{\s}} \transpose{\tensor{F}}_{\s},
\end{equation}
where the total Cauchy stress of the mixtures is
\begin{equation}
\label{eq: Cauchy stress def}
\tensor{T} = -p\tensor{I} + \tensor{T}^{e}.
\end{equation}
The quantity $p$ in the above equation is called the pore pressure and, in the weak formulation that follows, it can be viewed as a multiplier for the enforcement of the constraint in Eq.~\eqref{eq: vavg constraint}.

With these assumptions, following \cite{Bowen1980Incompressible-0}, the spatial (Eulerian) expression of the balance of momentum laws for the solid and the fluid phases are, respectively, 
\begin{align}
\label{eq: BOM solid}
\bv{0} &=
\rho_{\s} (\partial_{t} \bv{v}_{\s} + \bv{v}_{\s} \cdot \grad \bv{v}_{\s} - \bv{b}_{\s})
+ \xi_{\s} \grad p - \xi_{\f}^{2}\frac{\mu_{\f}}{k_{\s}} (\bv{v}_{\f} - \bv{v}_{\s}) -
\ldiv\tensor{T}^{e} \quad \text{in $B_{t}$},
\\
\label{eq: BOM fluid}
\bv{0} &=
\rho_{\f} \bigl( \partial_{t} \bv{v}_{\f} + \bv{v}_{\f} \cdot \grad \bv{v}_{\f} - \bv{b}_{\f}\bigr) +\xi_{\f} \grad p + \xi_{\f}^{2}\frac{\mu_{\f}}{k_{\s}} (\bv{v}_{\f} - \bv{v}_{\s})\quad \text{in $B_{t}$},
\end{align}
where $\bv{b}_{a}$ ($a = \s,\f$) is the body force per unit mass acting on phase $a$, $\mu_{\f}>0$ is the dynamic viscosity of the fluid, and $k_{\s}>0$ is the permeability of the solid.

\begin{remark}[Quasi-static Approximation]
\label{remark: quasi static equations}
Various flow regimes are of interest in applications.  In this paper we focus on quasi-static processes, which we define to be a motion with negligible inertia effects, that is, $\rho_{\s}D_{t_{\s}}\vs \approx \bv{0}$ and $\rho_{\f}D_{t_{\f}}\vf \approx \bv{0}$.  In this case, Eqs.~\eqref{eq: BOM solid} and~\eqref{eq: BOM fluid} reduce to
\begin{align}
\label{eq: BOM solid quasi static}
\bv{0} &=
- \rho_{\s} \bv{b}_{\s}
+ \xi_{\s} \grad p - \xi_{\f}^{2}\frac{\mu_{\f}}{k_{\s}} (\bv{v}_{\f} - \bv{v}_{\s}) -
\ldiv\tensor{T}^{e} \quad \text{in $B_{t}$},
\\
\label{eq: BOM fluid quasi static}
\bv{0} &=
- \rho_{\f} \bv{b}_{\f}+\xi_{\f} \grad p + \xi_{\f}^{2}\frac{\mu_{\f}}{k_{\s}} (\bv{v}_{\f} - \bv{v}_{\s})\quad \text{in $B_{t}$},
\end{align}
which, using the notion of filtration velocity in  Eq.~\eqref{eq: filtration def}, and recalling that we have assumed  $0 < \xi_{\f} < 1$,  can also be given the form
\begin{align}
\label{eq: BOM solid flt}
\bv{0} &=
-\rho_{\s} \bv{b}_{\s} -\rho_{\f} \bv{b}_{\f}
+ \grad p - \ldiv\tensor{T}^{e} \quad \text{in $B_{t}$},
\\
\label{eq: BOM fluid flt}
\bv{0} &=
-\rho_{\f}^{*} \bv{b}_{\f} + \grad p + \frac{\mu_{\f}}{k_{\s}} \fv \quad \text{in $B_{t}$}.
\end{align}
When $\bv{b}_{\f} = \bv{0}$, Eq.~\eqref{eq: BOM fluid flt} is referred to as Darcy's law, and it is important to note that the velocity field in this case is not the velocity of fluid particles in the strictest sense, but the filtration velocity.
\end{remark}

\subsection{Boundary conditions and governing equations}
As is traditional, we partition $\partial B_{t}$ into subsets $\Gamma_{t_{\s}}^{D}$ and $\Gamma_{t_{\s}}^{N}$ such that
\begin{equation}
\label{eq: boundary partition}
\partial B_{t} = \Gamma_{t_{\s}}^{D} \cup \Gamma_{t_{\s}}^{N}
\quad\text{and}\quad
\Gamma_{t_{\s}}^{D} \cap \Gamma_{t_{\s}}^{N} = \emptyset.
\end{equation}
We assume that
\begin{gather}
\label{eq: solid BCs}
\bv{u}_{\s}(\bv{x},t) = \bar{\bv{u}}_{\s}(\bv{x},t)~\text{for $\bv{x} \in \Gamma_{t_{\s}}^{D}$},
\quad\text{and}\quad
\bv{v}_{\s}(\bv{x},t) = \bar{\bv{v}}_{\s}(\bv{x},t) = D_{t_{\s}}\bar{\bv{u}}_{\s}(\bv{x},t) ~\text{for $\bv{x} \in \Gamma_{t_{\s}}^{D}$},
\shortintertext{as well as}
\tensor{T}(\bv{x},t)\bv{n}(\bv{x},t) = \bar{\bv{s}}(\bv{x},t)~\text{for $\bv{x} \in \Gamma_{t_{\s}}^{N}$},
\end{gather}
so that $\Gamma_{t_{\s}}^{D}$ is where Dirichlet data are prescribed for the solid skeleton, whereas $\Gamma_{t_{\s}}^{N}$ is where Neumann data are prescribed in the form of a traction distribution $\bar{\bv{s}}$. For the fluid, we assume that the boundary is impermeable:
\begin{equation}
\label{eq: boundary impermeability}
(\bv{v}_{\f}-\bv{v}_{\s}) \cdot \bv{n} = 0
~\text{on $\partial B_{t}$}.
\end{equation}
Clearly, Eq.~\eqref{eq: boundary impermeability} can be replaced by a number of conditions, depending on the specific physics at hand (cf.\ \citealp{dellIsola2009Boundary-Condit-0}).

\begin{remark}[Impermeability and filtration velocity]
\label{eq: impermeability}
For $0 < \xi_{\f} < 1$, Eq.~\eqref{eq: boundary impermeability} is equivalent to
\begin{equation}
\label{eq: boundary impermeability vflt}
\fv \cdot \bv{n} = 0~\text{on $\partial B_{t}$}.
\end{equation}
From a physical viewpoint, one can argue that Eq.~\eqref{eq: boundary impermeability vflt} is a better representation of the impermeability condition, in that its effect is \emph{weighed} by the volume fraction of fluid present.  That is, for a given difference $\vf-\vs$, the physical effect of this condition needs to take into account the amount of fluid that is actually flowing.
\end{remark}

\subsection{Problem's strong form}
In an Eulerian framework, the problem we consider is
\begin{problem}[Strong Form \,---\, Eulerian Framework]
\label{pb strong eulerian}
Given the
\begin{itemize}
\item
body force fields $\bv{b}_{\s}:B_{t} \to \mathscr{T}^{d}$ and $\bv{b}_{\f}:B_{t} \to \mathscr{T}^{d}$,
\item
prescribed boundary displacement $\bar{\bv{u}}_{\s}: \Gamma_{t_{\s}}^{D}\to\mathscr{T}^{d}$,
\item
applied boundary tractions $\bar{s}:\Gamma_{t_{\s}}^{N}\to\mathscr{T}^{d}$,
\item
intrinsic mass density distributions $\rho_{\s}^{*}>0$ and $\rho_{\f}^{*}>0$, 
\item
referential volume fraction distribution $\xi_{R_{\s}}:B_{\s}\to(0,1)$,
\item
constitutive relation in Eq.~\eqref{eq: Cauchy stress} along with constitutive properties $\mu_{\f}>0$, $k_{\s}>0$, and the strain energy $W(\tensor{C}_{\s}): \PSym(\mathscr{T}^{d})\to \mathbb{R}$ convex in $\tensor{C}_{\s}$, with $\PSym(\mathscr{T}^{d})$ the set of symmetric positive definite second order tensors on $\mathscr{T}^{d}$,
\item
initial conditions $\bv{v}_{\s}^{0}: B_{0}\to \mathscr{T}^{d}$, $\bv{v}_{\f}^{0}: B_{0}\to \mathscr{T}^{d}$, and  $\bv{u}_{\s}^{0}: B_{0}\to \mathscr{T}^{d}$, with $B_{0} = B_{t}\big|_{t = 0}$,
\end{itemize}
for $t \in (0,T)$, find $\bv{v}_{\s}: B_{t}\to \mathscr{T}^{d}$, $\bv{v}_{\f}: B_{t}\to \mathscr{T}^{d}$, $p: B_{t}\to \mathbb{R}$, and  $\bv{u}_{\s}: B_{t}\to \mathscr{T}^{d}$ with $\det\bigl[(\tensor{I} - \grad\bv{u}_{\s})^{-1}\bigr]>0$, such that Eqs.~\eqref{eq: vavg constraint}, \eqref{eq: BOM solid}, \eqref{eq: BOM fluid}, and the last of Eqs.~\eqref{eq: us vs relation Eulerian} are satisfied along with the boundary conditions in Eqs.~\eqref{eq: solid BCs}--\eqref{eq: boundary impermeability}.
\end{problem}

For the purpose of predicting the deformation of the solid skeleton from its reference configuration, we adopt an \ac{ALE} approach by which the governing equations are reformulated on the domain $B_{\s}$, taken to serve the double duty of both reference and initial configuration of the solid skeleton.  As mentioned earlier, to avoid a proliferation of symbols, functions describing physical quantities defined on $B_{t}$ and on $B_{\s}$ will be denoted in the same way and ambiguity is resolved by context.  With this in mind, using standard techniques from continuum mechanics (cf.\ \citealp{GurtinFried_2010_The-Mechanics_0}) we now restate Problem~\ref{pb strong eulerian} as follows:
\begin{problem}[Strong Form \,---\, \ac{ALE} Framework]
\label{problem: ALE strong}
Referring to the first of Eqs.~\eqref{eq: us Fs and Js} and~\eqref{eq: displacement Eulerian}, we consider the map $\bv{\chi}_{\s}^{-1}: B_{t} \to B_{\s}$ such that $\bv{x} = \bv{\chi}_{\s}(\bv{X}_{\s},t) = \bv{X}_{\s} + \bv{u}_{\s}(\bv{X}_{\s},t)$.  Under this map, given a generic field $\zeta(\bv{x},t)$ on $B_{t}$, we define a corresponding field $\zeta(\bv{X}_{\s},t)$ such that
\begin{equation}
\label{eq: image under map}
\zeta(\bv{X}_{\s},t) \coloneqq \zeta(\bv{x},t) \composition \bv{\chi}_{\s}(\bv{X}_{\s},t).
\end{equation}
Then, given the
\begin{itemize}
\item
body force fields $\bv{b}_{\s}:B_{\s} \to \mathscr{T}^{d}$ and $\bv{b}_{\f}:B_{\s} \to \mathscr{T}^{d}$,
\item
prescribed boundary displacement $\bar{\bv{u}}_{\s}: \Gamma_{\s}^{D}\to\mathscr{T}^{d}$,
\item
applied boundary tractions $\bar{s}:\Gamma_{\s}^{N}\to\mathscr{T}^{d}$,
\item
intrinsic mass density distributions $\rho_{\s}^{*}>0$ and $\rho_{\f}^{*}>0$, 
\item
referential volume fraction distribution $\xi_{R_{\s}}:B_{\s}\to (0,1)$,
\item
constitutive relation in Eq.~\eqref{eq: Cauchy stress} along with constitutive properties $\mu_{\f}>0$, $k_{\s}>0$, and the strain energy $W(\tensor{C}_{\s}): \PSym(\mathscr{T}^{d})\to \mathbb{R}$ convex in $\tensor{C}_{\s}$, with $\PSym(\mathscr{T}^{d})$ the set of symmetric positive definite second order tensors on $\mathscr{T}^{d}$,
\item
initial conditions $\bv{v}_{\s}^{0}: B_{0}\equiv B_{\s}\to \mathscr{T}^{d}$, $\bv{v}_{\f}^{0}: B_{0}\to \mathscr{T}^{d}$, and  $\bv{u}_{\s}^{0}: B_{0}\equiv B_{\s} \to \mathscr{T}^{d}$,
\end{itemize}
for $t \in (0,T)$, find $\bv{v}_{\s}: B_{\s}\to \mathscr{T}^{d}$, $\bv{v}_{\f}: B_{\s}\to \mathscr{T}^{d}$, $p: B_{\s}\to \mathbb{R}$, and  $\bv{u}_{\s}: B_{\s}\to \mathscr{T}^{d}$ with $J_{\s} = \det\bigl(\tensor{I} + \Grad\bv{u}_{\s}\bigr)>0$, such that, for all $\bv{X}_{\s} \in B_{\s}$,
\begin{equation}
\label{eq: ALE initial conditions}
\bv{u}_{\s}\big|_{t=0} = \bv{u}_{\s}^{0}, \quad
\bv{v}_{\s}\big|_{t=0} = \bv{v}_{\s}^{0}, \quad\text{and}\quad
\bv{v}_{\f}\big|_{t=0} = \bv{v}_{\f}^{0}~\text{on $B_{\s}$},
\end{equation}
\begin{align}
\label{eq: ALE us vs constraint}
\bv{0} &= \partial_{t}\bv{u}_{\s} - \bv{v}_{\s},
\\
\label{eq: ALE BOM solid}
\bv{0} &=
\xi_{R_{\s}} \rho_{\s}^{*} (\partial_{t} \bv{v}_{\s} - \bv{b}_{\s})
+ \xi_{R_{\s}} \inversetranspose{\tensor{F}}_{\s} \Grad p - \bigl(J_{\s} - \xi_{R_{\s}}\bigr)^{\!2}\frac{\mu_{\f}}{J_{\s}k_{\s}} (\bv{v}_{\f} - \bv{v}_{\s}) -
\udiv\tensor{P}^{e},
\\
\label{eq: ALE BOM fluid}
\bv{0} &=
\begin{multlined}[t]
(J_{\s} -  \xi_{R_{\s}}) \rho_{\f}^{*} \bigl[ \partial_{t} \bv{v}_{\f} +
\Grad \bv{v}_{\f}\tensor{F}_{\s}^{-1}( \bv{v}_{\f} -  \bv{v}_{\s}) - \bv{b}_{\f}\bigr]
\\
\qquad\qquad+ 
(J_{\s} -  \xi_{R_{\s}}) \inversetranspose{\tensor{F}}_{\s} \Grad p + \bigl(J_{\s} - \xi_{R_{\s}}\bigr)^{\!2}\frac{\mu_{\f}}{J_{\s}k_{\s}} (\bv{v}_{\f} - \bv{v}_{\s}),
\end{multlined}
\\
\label{eq: ALEvavg constraint}
0 &=\inversetranspose{\tensor{F}}_{\s} \colondot \Grad \biggl[\frac{\xi_{R_{\s}}}{J_{\s}} \bv{v}_{\s} + \biggl(1 - \frac{\xi_{R_{\s}}}{J_{\s}} \biggr) \bv{v}_{\f} \biggr],
\end{align}
and such that
\begin{alignat}{4}
\label{eq: ALE boundary conditions 1}
\bv{u}_{\s} - \bar{\bv{u}}_{\s} &= \bv{0} \quad& &\text{on $\Gamma_{\s}^{D}$},
&\qquad
\bv{v}_{\s} - \partial_{t}\bar{\bv{u}}_{\s} &= \bv{0} \quad& &\text{on $\Gamma_{\s}^{D}$},
\\
\label{eq: ALE boundary conditions 2}
\frac{-p J_{\s}\inversetranspose{\tensor{F}}_{\s} + \tensor{P}^{e}}{J_{\s}\|\inversetranspose{\tensor{F}}_{\s}\bv{n}_{\s}\|}\bv{n}_{\s} - \bar{\bv{s}} &= \bv{0} \quad& &\text{on $\Gamma_{\s}^{N}$},
&\qquad
\tensor{F}^{-1}_{\s}(\bv{v}_{\f} - \bv{v}_{\s}) \cdot \bv{n}_{\s} &= 0 \quad& &\text{on $\Gamma_{\s}$},
\end{alignat}
where $\Gamma_{\s}^{D} = \bv{\chi}_{\s}^{-1}\bigl(\Gamma_{t_{s}}^{D}\bigr)$, $\Gamma_{\s}^{N} = \bv{\chi}_{\s}^{-1}\bigl(\Gamma_{t_{s}}^{N}\bigr)$, and $\tensor{P}^{e}$ is the Piola-Kirchhoff stress tensor corresponding to $\tensor{T}^{e}$, i.e.,
\begin{equation}
\label{eq: Piola transform of Te}
\tensor{P}^{e} = J_{\s} \tensor{T}^{e} \inversetranspose{\tensor{F}}_{\s}
\quad\Rightarrow\quad
\tensor{P}^{e} = 2 \xi_{R_{\s}} \, \tensor{F}_{\s} \frac{\partial W_{\s}}{\partial \tensor{C}_{\s}},
\end{equation}
where this last expression results from Eq.~\eqref{eq: Cauchy stress}.
\end{problem}

\section{Weak formulations}
\label{sec: abstract EWF}

\subsection{Functional setting for principal unknowns}
We choose the function spaces for the solid's displacement and velocity, for the fluid's velocity, and the pore pressure to be, respectively:
\begin{alignat}{4}
\label{eq: functional spaces eulerian us}
\mathcal{V}^{\bv{u}_{\s}}
&\coloneqq
\bigl\{&
\bv{u}_{\s} &\in L^{2}(B_{\s})^{d} \,& &\big|\,& \grad \bv{u}_{\s} &\in L^{\infty}(B_{\s})^{d\times d}, \bv{u}_{\s} = \bar{\bv{u}}_{\s}~\text{on $\Gamma_{\s}^{D}$}
\bigr\},
\\
\label{eq: functional spaces eulerian vs}
\mathcal{V}^{\bv{v}_{\s}}
&\coloneqq
\bigl\{&
\bv{v}_{\s} &\in L^{2}(B_{\s})^{d} \,& &\big|\,& \grad \bv{v}_{\s} &\in L^{2}(B_{\s})^{d\times d}, \bv{v}_{\s} =  \partial_{t}\bar{\bv{u}}_{\s}~\text{on $\Gamma_{\s}^{D}$}
\bigr\},
\\
\label{eq: functional spaces eulerian vf}
\mathcal{V}^{\bv{v}_{\f}}
&\coloneqq
\bigl\{&
\bv{v}_{\f} &\in L^{2}(B_{\s})^{d}\,& &\big|\,& \grad \bv{v}_{\f} &\in L^{2}(B_{\s})^{d\times d}
\bigr\},
\\
\label{eq: functional spaces eulerian p}
\mathcal{V}^{p}
&\coloneqq
\bigl\{&
p &\in L^{2}(B_{\s})^{d}\,& &\big|\,& \grad p &\in L^{2}(B_{\s})^{d\times d}
\bigr\}.
\end{alignat}
We also introduce the spaces
\begin{alignat}{5}
\label{eq: functional spaces eulerian dus}
\mathcal{V}^{\bv{u}_{\s}}_{0}
&\coloneqq
\bigl\{&
\bv{u}_{\s} &\in L^{2}(B_{\s})^{d}\,& &\big|\,& \grad \bv{u}_{\s} &\in L^{\infty}(B_{\s})^{d\times d},\,& \bv{u}_{\s} &= \bv{0}~\text{on $\Gamma_{\s}^{D}$}
\bigr\},
\\
\label{eq: functional spaces eulerian dvs}
\mathcal{V}^{\bv{v}_{\s}}_{0}
&\coloneqq
\bigl\{&
\bv{v}_{\s} &\in L^{2}(B_{\s})^{d}\,& &\big|\,& \grad \bv{v}_{\s} &\in L^{2}(B_{\s})^{d\times d},\,& \bv{v}_{\s} &= \bv{0}~\text{on $\Gamma_{\s}^{D}$}
\bigr\}.
\end{alignat}

\begin{remark}[Functional space for the filtration velocity]
Referring to Eq.~\eqref{eq: filtration def}, identifying a space for the filtration velocity requires a declaration of our expectations on the smoothness of the datum $\xi_{R_{\s}}$.  This will also have consequences on what we can expect for the mass density fields $\rho_{\s}$ and $\rho_{\f}$.  In applications $\xi_{R_{\s}}$ is measured experimentally and for applications in brain mechanics we expect a relatively high level of uncertainty in such measurements.  Here we simply assume that
\begin{equation}
\label{eq: xiRs smoothness}
\xi_{R_{\s}} \in L^{\infty}(B_{\s}),
\end{equation}
which, along with previous assumptions, implies that
\begin{equation}
\label{eq: vflt functional space}
\fv \in \mathcal{V}^{\vf}.
\end{equation}
\end{remark}

\subsection{Functional setting for the time derivatives}
\label{subsec: time drivatives function spaces}
The function spaces for the time derivatives $\partial_{t}\bv{u}_{\s}$, $\partial_{t}\bv{v}_{\s}$, and $\partial_{t}\bv{v}_{\f}$ are not normally taken as elements of $\mathcal{V}^{\bv{u}_{\s}}$, $\mathcal{V}^{\bv{v}_{\s}}$, and $\mathcal{V}^{\bv{v}_{\f}}$, respectively.  Rather, they are typically assumed to be of the same class as the prescribed fields $\bv{b}_{\s}$ and $\bv{b}_{\f}$, the latter taken in $H^{-1}(B_{\s})$.  However, due to the remapping of the governing equation to $B_{\s}$, the regularity of the fields in question cannot be chosen independently of that of the map from $B_{t}$ to $B_{\s}$.  The crucial element of this map is its gradient $\tensor{F}_{\s}^{-1} = (\tensor{I} + \nabla_{\bv{X}_{\s}} \bv{u}_{\s})^{-1}$, which is the justification for the choice of $\mathcal{V}^{\bv{u}_{\s}}$.  Therefore, using the argument presented in \cite{Heltai2012Variational-Imp0}, we select $\partial_{t}\bv{u}_{\s}$, $\partial_{t}\bv{v}_{\s}$, and $\partial_{t}\bv{v}_{\f}$ in appropriate pivot spaces $\mathcal{H}^{\bv{u}_{\s}}$, $\mathcal{H}^{\bv{v}_{\s}}$, $\mathcal{H}^{\bv{v}_{\f}}$ such that
\begin{gather}
\label{eq: pivot us}
\mathcal{V}^{\bv{u}_{\s}}  \subseteq \mathcal{H}^{\bv{u}_{\s}} \subseteq \bigl(\mathcal{H}^{\bv{u}_{\s}}\bigr)^{*} \subseteq \bigl(\mathcal{V}^{\bv{u}_{\s}}\bigr)^{*},
\\
\label{eq: pivot vs}
\mathcal{V}^{\bv{v}_{\s}}  \subseteq \mathcal{H}^{\bv{v}_{\s}} \subseteq \bigl(\mathcal{H}^{\bv{v}_{\s}}\bigr)^{*} \subseteq \bigl(\mathcal{V}^{\bv{v}_{\s}}\bigr)^{*},
\\
\label{eq: pivot vf}
\mathcal{V}^{\bv{v}_{\f}}  \subseteq \mathcal{H}^{\bv{v}_{\f}} \subseteq \bigl(\mathcal{H}^{\bv{v}_{\f}}\bigr)^{*} \subseteq \bigl(\mathcal{V}^{\bv{v}_{\f}}\bigr)^{*},
\end{gather}
where the notation $(\square)^{*}$ is meant to indicate the dual of $\square$.  Again, as discussed in \cite{Heltai2012Variational-Imp0},  when pulled back to $B_{\s}$, the satisfaction of Eq.~\eqref{eq: us vs relation Eulerian} and the use of standard Sobolev inequalities (cf., e.g., \citealp{Evans2010Partial-Differential-0}), allow one to deduce that 
\begin{equation}
\mathcal{V}^{\bv{u}_{\s}}  \subseteq \mathcal{H}^{\bv{u}_{\s}} \subseteq H^{1}(B_{\s}),
\end{equation}
so that the pivot space for $\partial_{t}\bv{u}_{\s}$ can be taken to be $H^{1}(B_{\s})$.

\subsection{Functional setting for the data}
\label{subsec: function spaces for data}
Referring to Eqs.~\eqref{eq: functional spaces eulerian us}--\eqref{eq: functional spaces eulerian vf}, we take the initial conditions as follows:
\begin{equation}
\label{eq: ICs function spaces}
\bv{u}_{\s}^{0} \in \mathcal{V}^{\bv{u}_{\s}},\quad
\bv{v}_{\s}^{0} \in \mathcal{V}^{\bv{v}_{\s}},\quad\text{and}\quad
\bv{v}_{\f}^{0} \in \mathcal{V}^{\bv{v}_{\f}}.
\end{equation}
The body for terms are chosen as follows:
\begin{equation}
\label{eq: bs bf function spaces}
\bv{b}_{\s} \in H^{-1}(B_{\s})
\quad\text{and}\quad
\bv{b}_{\f} \in H^{-1}(B_{\s}).
\end{equation}
Finally, the boundary traction field is chosen as follows:
\begin{equation}
\label{eq: s function spaces}
\bar{\bv{s}} \in H^{-\frac{1}{2}}\bigl(\Gamma_{\s}^{N}\bigr).
\end{equation}

\subsection{Functional setting for fields in Eulerian form}
The function spaces introduced so far have been defined relative to the reference configuration $B_{\s}$, which we also take as the initial configuration.  Combined with the map $\bv{\chi}_{\s}^{-1}$, these spaces can be used for the definition of corresponding spaces of functions with domain $B_{t}$. We will present abstract weak formulations for both the Eulerian and \ac{ALE} frameworks.  This is because the energy estimates developed in the Eulerian framework are more readily related to classic results from continuum mechanics.  Again in the interest of limiting the proliferation of symbols, we will use the same notation for the function spaces supported over $B_{t}$ as for those defined above and supported over $B_{\s}$.  As noted earlier, the ambiguity can be resolved in context.

\subsection{Formulations}
\label{subsec: Stabilization}
\begin{problem}[Abstract Weak Formulation\,---\,Eulerian Framework]
\label{pb abstract weak eulerian}
Given the same data as in Problem~\ref{pb strong eulerian}, find $\bv{u}_{\s} \in \mathcal{V}^{\bv{u}_{\s}}$, $\bv{v}_{\s} \in \mathcal{V}^{\bv{v}_{\s}}$, $\bv{v}_{\f} \in \mathcal{V}^{\bv{v}_{\f}}$, and $p \in \mathcal{V}^{p}$ such that, for all $\test{\bv{u}}_{\s} \in \mathcal{V}^{\bv{v}_{\s}}_{0}$, $\test{\bv{v}}_{\s} \in \mathcal{V}^{\bv{v}_{\s}}_{0}$, $\test{\bv{v}}_{\f} \in \mathcal{V}^{\bv{u}_{\s}}$, and $\test{p} \in \mathcal{V}^{p}$,
\begin{gather}
\label{eq: weak form kinematic consistency}
\int_{B_{t}} \test{\bv{u}}_{\s} \cdot \bigl[(\tensor{I} - \grad \bv{u}_{\s})^{-1} \partial_{t}\bv{u}_{\s} - \bv{v}_{\s} \bigr] = 0
\shortintertext{and}
\label{eq: weak form balance laws}
\begin{aligned}
&
-\int_{B_{t}} \grad \test{p} \cdot \bigl(\xi_{\s} \bv{v}_{\s} + \xi_{\f} \bv{v}_{\f} \bigr) + \int_{\partial B_{t}} \test{p} \, \bv{v}_{\s} \cdot \bv{n}
\\
&
\qquad+\int_{B_{t}} \test{\bv{v}}_{\f} \cdot
\bigl[
\rho_{\f} \bigl( \partial_{t} \bv{v}_{\f} + \bv{v}_{\f} \cdot \grad \bv{v}_{\f} - \bv{b}_{\f}\bigr) +\xi_{\f} \grad p + \xi_{\f}^{2}\frac{\mu_{\f}}{k_{\s}} (\bv{v}_{\f} - \bv{v}_{\s}) \bigr]
\\
&
\qquad+\int_{B_{t}} \test{\bv{v}}_{\s} \cdot
\bigl[
\rho_{\s} (\partial_{t} \bv{v}_{\s} + \bv{v}_{\s} \cdot \grad \bv{v}_{\s} - \bv{b}_{\s})
+ \xi_{\s} \grad p - \xi_{\f}^{2}\frac{\mu_{\f}}{k_{\s}} (\bv{v}_{\f} - \bv{v}_{\s})
\bigr]
\\
&
\qquad+ \int_{B_{t}} \grad \test{\bv{v}}_{\s} \colondot \tensor{T}^{e} -\int_{\Gamma_{t_{\s}}^{N}} \test{\bv{v}}_{\s} \cdot (\bar{\bv{s}} + p \bv{n}) = 0,
\end{aligned}
\end{gather}
where time derivatives are taken in the spaces discussed in Section~\ref{subsec: time drivatives function spaces}.
\end{problem}

\begin{problem}[Abstract Weak Formulation\,---\,\ac{ALE} Framework]
\label{pb abstract weak ALE}
Given the same data as in Problem~\ref{problem: ALE strong}, find $\bv{u}_{\s} \in \mathcal{V}^{\bv{u}_{\s}}$, $\bv{v}_{\s} \in \mathcal{V}^{\bv{v}_{\s}}$, $\bv{v}_{\f} \in \mathcal{V}^{\bv{v}_{\f}}$, and $p \in \mathcal{V}^{p}$ such that, for all $\test{\bv{u}}_{\s} \in \mathcal{V}^{\bv{v}_{\s}}_{0}$, $\test{\bv{v}}_{\s} \in \mathcal{V}^{\bv{v}_{\s}}_{0}$, $\test{\bv{v}}_{\f} \in \mathcal{V}^{\bv{u}_{\s}}$, and $\test{p} \in \mathcal{V}^{p}$,
\begin{gather}
\label{eq: ALE weak form kinematic consistency}
\int_{B_{\s}} \test{\bv{u}}_{\s} \cdot (\partial_{t}\bv{u}_{\s}-\bv{v}_{\s}) = 0,
\shortintertext{and}
\label{eq: ALE weak form balance laws}
\begin{aligned}
&
-\int_{B_{\s}} \Grad \test{p} \cdot \tensor{F}_{\s}^{-1}\bigl[\xi_{R_{\s}} \bv{v}_{\s} + (J_{\s} - \xi_{R_{\s}})\bv{v}_{\f} \bigr]
+\int_{\Gamma_{\s}^{N}} \test{p} \, J_{\s} \tensor{F}^{-1}_{\s} \bv{v}_{\s} \cdot \bv{n}_{\s}
+\int_{\Gamma_{\s}^{D}} \test{p} \, J_{\s} \tensor{F}^{-1}_{\s} \bar{\bv{v}}_{\s} \cdot \bv{n}_{\s}
\\
&
\qquad+\int_{B_{\s}} \test{\bv{v}}_{\f} \cdot
(J_{\s} - \xi_{R_{\s}}) \rho_{\f}^{*} \bigl[ \partial_{t} \bv{v}_{\f} + (\Grad\bv{v}_{\f}) \tensor{F}_{\s}^{-1}(\bv{v}_{\f} - \bv{v}_{\s}) - \bv{b}_{\f}\bigr] 
\\
&
\qquad+\int_{B_{\s}} \test{\bv{v}}_{\f} \cdot
\bigl[(J_{\s} - \xi_{R_{\s}}) \inversetranspose{\tensor{F}}_{\s}\Grad p + (J_{\s} - \xi_{R_{\s}})^{2}\frac{\mu_{\f}}{J_{\s} k_{\s}} (\bv{v}_{\f} - \bv{v}_{\s}) \bigr]
\\
&
\qquad+\int_{B_{\s}} \test{\bv{v}}_{\s} \cdot
\bigl[
\xi_{R_{\s}}\rho_{\s}^{*} (\partial_{t} \bv{v}_{\s} - \bv{b}_{\s})
+ \xi_{R_{\s}} \inversetranspose{\tensor{F}}_{\s}\Grad p - (J_{\s} - \xi_{R_{\s}})^{2}\frac{\mu_{\f}}{J_{\s} k_{\s}} (\bv{v}_{\f} - \bv{v}_{\s})
\bigr]
\\
&
\qquad+ \int_{B_{\s}} \Grad \test{\bv{v}}_{\s} \colondot \tensor{P}^{e} - \int_{\Gamma_{\s}^{N}} \test{\bv{v}}_{\s} \cdot \bigl(\bar{\bv{s}} + J_{\s} p \, \inversetranspose{\tensor{F}}_{\s}\bv{n}_{\s}\bigr) = 0,
\end{aligned}
\end{gather}
where, for $\tensor{A}\colondot\tensor{B}$ denotes the inner product of tensors $\tensor{A}$ and $\tensor{B}$, time derivatives are taken in the spaces discussed in Section~\ref{subsec: time drivatives function spaces}, and where $\bar{\bv{s}}$ is understood to be the prescribed traction field on $\Gamma_{\s}^{N}$.  Clearly, if the \emph{given} $\bar{\bv{s}}$ were to be prescribed on $\bv{u}_{\s}(\Gamma_{\s}^{N})$, i.e., the image of $\Gamma_{\s}^{N}$ under the motion of the solid skeleton, then the field $\bar{\bv{s}}$ in the above equation should be replaced by $J_{\s} \|\inversetranspose{\tensor{F}}_{\s} \bv{n}_{\s}\| \bar{\bv{s}}$.
\end{problem}

\begin{remark}
Problems~\ref{pb abstract weak eulerian} and~\ref{pb abstract weak ALE} can be shown to be equivalent to their respective strong counterparts using standards approaches (see, e.g., \citealp{Brenner2002The-Mathematical-Theory-0} or \citealp{Hughes2000The-Finite-Element-0}).  Furthermore, under the assumption concerning the space $\mathcal{V}^{\bv{u}_{\s}}$ stated in Eq.~\eqref{eq: functional spaces eulerian us}, we have that Problems~\ref{pb abstract weak eulerian} and~\ref{pb abstract weak ALE} are equivalent to each other.
\end{remark}

The formulation in Problem~\ref{pb abstract weak eulerian} can be easily related to energy estimates that match corresponding relations in continuum mechanics.  Specifically, we have the following
\begin{lemma}(Abstract Energy Estimates)
\label{lemma: theorem of power expended}
Given the formulation in Problem~\ref{pb abstract weak eulerian} (or its equivalent form in Problem~\ref{pb abstract weak ALE}), letting $\bv{v}_{\s} \in \mathcal{V}^{\bv{v}_{\s}}_{0}$, and for $\test{\bv{v}}_{\s} = \bv{v}_{\s}$, $\test{\bv{v}}_{\f} = \bv{v}_{\f}$, and $\test{p} = p$,
\begin{equation}
\label{eq: theorem of power expended}
\frac{\nsd{}}{\nsd{t}} (\mathscr{K} + \mathscr{W}) + \mathscr{D}
= \int_{B_{t}} (\bv{v}_{\s} \cdot \rho_{\s} \bv{b}_{\s} + \bv{v}_{\f} \cdot \rho_{\f} \bv{b}_{\f}) 
+ \int_{\Gamma_{t_{\s}}^{N}} \bv{v}_{\s} \cdot \bar{\bv{s}}, 
\end{equation}
where
\begin{equation}
\label{eq: energies}
\mathscr{K} \coloneqq \int_{B_{t}} \tfrac{1}{2} \bigl(\rho_{\s} \|\bv{v}_{\s}\|^{2} + \rho_{\f} \|\bv{v}_{\f}\|^{2}\bigr),
\quad
\mathscr{W} \coloneqq \int_{B_{t}} \Psi(\bv{x},t),
\quad\text{and}\quad
\mathscr{D} \coloneqq \int_{B_{t}} \frac{\mu_{\f}}{k_{\s}} \|\xi_{\f}(\bv{v}_{\f} - \bv{v}_{\s})\|^{2}.
\end{equation}
\end{lemma}
\begin{proof}
Setting $\test{\bv{v}}_{\s} = \bv{v}_{\s}$, $\test{\bv{v}}_{\f} = \bv{v}_{\f}$, and $\test{p} = p$, Eq.~\eqref{eq: weak form balance laws} becomes
\begin{multline}
\label{eq: weak form balance laws energy estimate 1}
\int_{B_{t}} 
\bigl[
\bv{v}_{\f} \cdot
\rho_{\f} \bigl( \partial_{t} \bv{v}_{\f} + \bv{v}_{\f} \cdot \grad \bv{v}_{\f}\bigr) + (\bv{v}_{\f} - \bv{v}_{\s}) \cdot \xi_{\f}^{2}\frac{\mu_{\f}}{k_{\s}} (\bv{v}_{\f} - \bv{v}_{\s})
\bigl]
\\
+\int_{B_{t}} \bv{v}_{\s} \cdot
\rho_{\s} (\partial_{t} \bv{v}_{\s} + \bv{v}_{\s} \cdot \grad \bv{v}_{\s})
+ 
\int_{B_{t}} \grad \bv{v}_{\s} \colondot \tensor{T}^{e}
\\
= 
\int_{B_{t}} (\bv{v}_{\s} \cdot \rho_{\s} \bv{b}_{\s} + \bv{v}_{\f} \cdot \rho_{\f} \bv{b}_{\f}) + \int_{\Gamma_{t_{\s}}^{N}} \bv{v}_{\s} \cdot \bar{\bv{s}},
\end{multline}
We now recall that $\partial_{t} \bv{v}_{\f} + \bv{v}_{\f} \cdot \grad \bv{v}_{\f} = D_{t_{\f}} \bv{v}_{\f}$ and $\partial_{t} \bv{v}_{\s} + \bv{v}_{\s} \cdot \grad \bv{v}_{\s} = D_{t_{\s}} \bv{v}_{\s}$, where $D_{t_{\f}} \bv{v}_{\f}$ and $D_{t_{\s}} \bv{v}_{\s}$ are the material time derivatives of $\bv{v}_{\f}$ and $\bv{v}_{\s}$, respectively.  In turn this implies that Eq.~\eqref{eq: weak form balance laws energy estimate 1} can be rewritten as
\begin{multline}
\label{eq: weak form balance laws energy estimate 2}
\int_{B_{t}} 
\bigl[ 
\tfrac{1}{2} \rho_{\s} D_{t_{\s}} (\bv{v}_{\s} \cdot \bv{v}_{\s})
+
\tfrac{1}{2} \rho_{\f} D_{t_{\s}} (\bv{v}_{\f} \cdot \bv{v}_{\f})
\bigr]
+
\int_{B_{t}} \grad \bv{v}_{\s} \colondot \tensor{T}^{e}
+
\int_{B_{t}} \frac{\mu_{\f}}{k_{\s}} \|\xi_{\f}(\bv{v}_{\f} - \bv{v}_{\s})\|^{2}
\\
= 
\int_{B_{t}} (\bv{v}_{\s} \cdot \rho_{\s} \bv{b}_{\s} + \bv{v}_{\f} \cdot \rho_{\f} \bv{b}_{\f}) + \int_{\Gamma_{t_{\s}}^{N}} \bv{v}_{\s} \cdot \bar{\bv{s}},
\end{multline}
By the combined application of the balance of mass and the transport theorem \citep{GurtinFried_2010_The-Mechanics_0}, and using the definition in the first of Eqs.~\eqref{eq: energies}, we have that the first integral in the above equation can be written as follows:
\begin{equation}
\label{eq: weak form balance laws energy estimate 3}
\int_{B_{t}} 
\bigl[ 
\tfrac{1}{2} \rho_{\s} D_{t_{\s}} (\bv{v}_{\s} \cdot \bv{v}_{\s})
+
\tfrac{1}{2} \rho_{\f} D_{t_{\s}} (\bv{v}_{\f} \cdot \bv{v}_{\f})
\bigr]
=
\frac{\nsd{\mathscr{K}}}{\nsd{t}}.
\end{equation}
Next, using Eq.~\eqref{eq: Cauchy stress}, we observe that the second integral in Eq.~\eqref{eq: weak form balance laws energy estimate 2} can be rewritten via a change of variables of integration as
\begin{multline}
\label{eq: weak form balance laws energy estimate 4}
\int_{B_{t}} \grad \bv{v}_{\s} \colondot \tensor{T}^{e} = \int_{B_{\s}}
(\Grad\bv{v}_{\s}) \tensor{F}_{\s}^{-1} \colondot
2 \xi_{R_{\s}} \, \tensor{F}_{\s} \frac{\partial W_{\s}}{\partial \tensor{C}_{\s}} \transpose{\tensor{F}}_{\s}
\\
\Rightarrow\quad
\int_{B_{t}} \grad \bv{v}_{\s} \colondot \tensor{T}^{e}
=
\int_{B_{\s}}
\Grad\bv{v}_{\s} \colondot 2 \xi_{R_{\s}} \, \tensor{F}_{\s} \frac{\partial W_{\s}}{\partial \tensor{C}_{\s}}
\end{multline}
Then, in view of Eq.~\eqref{eq: ALE weak form kinematic consistency},
\begin{equation}
\label{eq: weak form balance laws energy estimate 5}
\Grad \bv{v}_{\s} = \partial_{t} \tensor{F}_{\s}
\quad\text{and}\quad
2 \tensor{F}_{\s} \frac{\partial W_{\s}}{\partial \tensor{C}_{\s}} = \frac{\partial W_{\s}}{\partial \tensor{F}_{\s}},
\end{equation}
we have
\begin{equation}
\label{eq: weak form balance laws energy estimate 6}
\int_{B_{t}} \grad \bv{v}_{\s} \colondot \tensor{T}^{e}
=
\frac{\nsd{}}{\nsd{t}}
\int_{B_{\s}} \xi_{R_{\s}}W_{\s}
= 
\frac{\nsd{}}{\nsd{t}}
\int_{B_{t}} \Psi
\end{equation}
Substituting the results of Eqs.~\eqref{eq: weak form balance laws energy estimate 3} and~\eqref{eq: weak form balance laws energy estimate 6} into Eq.~\eqref{eq: weak form balance laws energy estimate 2}, the claim follows.
\end{proof}

The result in Lemma~\ref{lemma: theorem of power expended} is directly related to a fundamental result in continuum mechanics often referred to as the \emph{theorem of power expended} \citep{Gurtin-CMBook-1981-1}. With this result in hand, we can then prove the stability of the abstract formulation:
\begin{theorem}[Stability of the Abstract Weak Formulation]
\label{theorem: abstract stability}
The (equivalent) weak formulations in Problems~\ref{pb abstract weak eulerian} and~\ref{pb abstract weak ALE} are stable.
\end{theorem}
\begin{proof}
Assuming pure Dirichlet boundary conditions, i.e., $\Gamma_{\s}^{D} = \partial B_{\s}$,
letting $\bv{v}_{\s} \in \mathcal{V}^{\bv{v}_{\s}}_{0}$, and for $\test{\bv{v}}_{\s} = \bv{v}_{\s}$, $\test{\bv{v}}_{\f} = \bv{v}_{\f}$, and $\test{p} = p$, and suppressing the external force fields $\bv{b}_{\s}$ and $\bv{b}_{\f}$, Lemma~\ref{lemma: theorem of power expended} implies that
\begin{equation}
\label{eq: stability proof}
\frac{\nsd{}}{\nsd{t}} (\mathscr{K} + \mathscr{W}) + \mathscr{D}
= 0.
\end{equation}
Hence, observing that the $\mathscr{D} \geq 0$, we have that the time rate of change of the total energy of the system is never positive.
\end{proof}

To make the notation more compact, we proceed to reformulate the problem in a (block)-matrix-like notation. For this purpose we adopt the notation in \cite{Heltai2012Variational-Imp0} and defined in the Appendix.
\begin{problem}[\ac{ALE} Framework\,---\,Fluid Velocity Dual Formulation]
\label{dual pb abstract weak ALE}
With the operators in Eqs.~\eqref{aeq: M11 def}--\eqref{eq: M33 check def}, Problem~\ref{pb abstract weak ALE} takes the following form:
\begin{align}
\label{eq: dual ALE weak form us}
\mathcal{M}_{\sad\sap}\partial_{t}\bv{u}_{\s}
-
\mathcal{M}_{\sad\sbp}\bv{v}_{\s}
&= 
\bv{0},
\\
\label{eq: dual ALE weak form vs}
\begin{aligned}[b]
&\mathcal{M}_{\sbd\sbp}(\xi_{R_{\s}})\partial_{t}\bv{v}_{\s}
-
\transpose{\mathcal{B}}_{\sdd\sbp}(\bv{u}_{\s};\xi_{R_{\s}}) p
-
\transpose{\mathcal{D}}_{\scd\sbp}(\bv{u}_{\s}; \xi_{R_{\s}})\bv{v}_{\f}
\\
&\qquad\qquad\;
+
\mathcal{D}_{\sbd\sbp}(\bv{u}_{\s}; \xi_{R_{\s}})\bv{v}_{\s}
+
\mathcal{A}_{\sbd}(\bv{u}_{\s}; \xi_{R_{\s}})
-
\transpose{\mathcal{S}}_{\sdd\sbp}(\bv{u}_{\s})p
\end{aligned}
&= \mathcal{F}_{\sbd}(\xi_{R_{\s}}),
\\
\label{eq: dual ALE weak form vf}
\begin{aligned}[b]
&\mathcal{M}_{\scd\scp}(\bv{u}_{\s};\xi_{R_{\s}})\partial_{t}\bv{v}_{\f}
+
\mathcal{N}_{\scd\scp}(\bv{u}_{\s},\bv{v}_{\f}; \xi_{R_{\s}})\bv{v}_{\f}
-
\mathcal{N}_{\scd\sbp}(\bv{u}_{\s},\bv{v}_{\f}; \xi_{R_{\s}})\bv{v}_{\s}
\\
&\qquad\qquad\;
-
\transpose{\mathcal{B}}_{\sdd\scp}(\bv{u}_{\s};\xi_{R_{\s}}) p
+
\mathcal{D}_{\scd\scp}(\bv{u}_{\s};\xi_{R_{\s}})\bv{v}_{\f}
-
\mathcal{D}_{\scd\sbp}(\bv{u}_{\s};\xi_{R_{\s}})\bv{v}_{\s}
\end{aligned}
&= \mathcal{F}_{\scd}(\bv{u}_{\s}; \xi_{R_{\s}}),
\\
\label{eq: dual ALE weak form p}
\mathcal{B}_{\sdd\sbp}(\bv{u}_{\s};\xi_{R_{\s}}) \bv{v}_{\s}
+
\mathcal{B}_{\sdd\scp}(\bv{u}_{\s};\xi_{R_{\s}}) \bv{v}_{\f}
+
\mathcal{S}_{\sdd\sbp}(\bv{u}_{\s}) \bv{v}_{\s}
&= \mathcal{F}_{\sdd}(\bv{u}_{\s}).
\end{align}
\end{problem}
\begin{remark}[Coercivity of $\mathcal{D}_{\scd\scp}(\bv{u}_{\s};\xi_{R_{\s}})$]
\label{remark: need for fv formulation}
Given its definition in Eq.~\eqref{eq: D33 def}, we have that the operator $\mathcal{D}_{\scd\scp}(\bv{u}_{\s};\xi_{R_{\s}})$ is coercive.  However, its coercivity is ``at the mercy'' of the volume fraction of the fluid.  Specifically, from a practical viewpoint, one should expect that the coercivity of the operator is weaker the lower the fluid volume fraction.  An alternative formulation, intending to circumvent this problem is presented later in the paper.
\end{remark}

\begin{remark}[Some operators coincide in their discrete forms]
\label{remark: matrix operators ids}
Conventional choices of the finite element spaces for the implementation of Problem~\ref{dual pb abstract weak ALE} can guarantee that, in their discrete form, some of the above operators coincide with one another.  For example, using a traditional Galerkin approach and choosing the same interpolation for $\bv{u}_{\s}$ and $\bv{v}_{\s}$, we have that $\mathcal{M}_{\sad\sap}^{h}=\mathcal{M}_{\sad\sbp}^{h}$, where, given an operator $\Xi$, the notation $\Xi^{h}$ denotes the discrete form of $\Xi$.  Similarly, choosing the same interpolation for $\bv{v}_{\f}$ and $\bv{v}_{\s}$, we that $\mathcal{D}_{\scd\scp}^{h} = \mathcal{D}_{\scd\sbp}^{h}$.
\end{remark}

\begin{remark}[Filtration velocity]
\label{remark: filtration velocity calculation}
If the filtration velocity is required in the solution of Problem~\ref{dual pb abstract weak ALE}, it can be recovered in a post-processing step as an $L^{2}$-projection. That is, given the solution of Problem~\ref{dual pb abstract weak ALE}, and referring to Eq.~\eqref{eq: filtration def}, $\fv$  takes on the form: 
\begin{equation}
\label{eq: fv as L2 projection}
\fv = \bigl[\check{\mathcal{M}}_{\scd\scp}(\us; 0)\bigr]^{-1}\bigl(\check{\mathcal{M}}_{\scd\scp}(\us; \xi_{R_{\s}}) \vf - \check{\mathcal{M}}_{\scd\sbp}(\us; \xi_{R_{\s}}) \vs\bigr),
\end{equation}
where $\check{\mathcal{M}}_{\scd\scp}(\us; \xi_{R_{\s}})$ and $\check{\mathcal{M}}_{\scd\sbp}(\us; \xi_{R_{\s}})$ are invertible operators defined in Eq.~\eqref{eq: M33 check def} and~\eqref{eq: M32 check def}, respectively.
\end{remark}

\section{Stabilization for the Quasi-Static Case}
\label{sec: Stabilization}
While the formulation presented so far is stable in the sense illustrated in Theorem~\ref{theorem: abstract stability}, its finite element implementation is expected to suffer from well known pathologies revolving around the Brezzi-Babu\v{s}ka (or $\inf\sup$) condition (cf.\ \citealp{Brezzi1991Mixed-and-Hybrid-0,Brenner2002The-Mathematical-Theory-0}). 
Coercivity loss \citep{Ern2013} can occur when the volume fraction $\xi_{R_{\s}}=0$ or $\xi_{R_{\s}}=1$.  In these 
extreme cases, the well-posedness of the discrete problem is lost.  Our numerical experiments have also indicated that there are numerical issues when $\xi_{R_{\s}}$ approaches said extremes, such as linear solver stagnation, etc.
These issues have been discussed by several authors and stabilization strategies have been suggested, for example, by \cite{Masud2002A-Stabilized-Mixed-0} for the classical Darcy flow problem, and by \cite{Masud2007A-Stabilized-Mixed-0} for the Darcy-Stokes problem.  In both the cited works, the problems considered are linear and do not include inertia effects.  In this paper, we propose an adaptation of the stabilization strategy proposed in \cite{Masud2007A-Stabilized-Mixed-0} to our nonlinear problem with multiple velocity fields, but still for the case when inertia terms are negligible.

\subsection{Adaptation of the strategy in \cite{Masud2002A-Stabilized-Mixed-0}}
\label{subsec: stab Masud and Hughes}
With reference to the momentum balance relations in Eqs.~\eqref{eq: BOM solid} and~\eqref{eq: BOM fluid}, we begin by considering the case in which the inertia of the fluid is negligible, i.e., we let $\rho_{\s} D_{t_{\s}}\bv{v}_{\s} = \bv{0}$ and $\rho_{\f} D_{t_{\f}}\bv{v}_{\f} = \bv{0}$. Then, Eq.~\eqref{eq: BOM fluid} becomes
\begin{equation}
\label{eq: Darcy proper}
- \rho_{\f} \bv{b}_{\f} +\xi_{\f} \grad p + \xi_{\f}^{2}\frac{\mu_{\f}}{k_{\s}} (\bv{v}_{\f} - \bv{v}_{\s}) = \bv{0} \quad \text{in $B_{t}$}.
\end{equation}
Adapting the strategy in \cite{Masud2002A-Stabilized-Mixed-0}, we obtain
a consistent stabilized formulation by adding to our weak form the following term:
\begin{equation}
\label{eq: stabilization term 1}
\int_{B_{t}}
\tfrac{1}{2}
\biggl(
\frac{k_{\s}}{\xi_{\f}\mu_{\f}} \grad\tilde{p}
-
(\tilde{\bv{v}}_{\f} - \tilde{\bv{v}}_{\s})
\biggr)
\cdot \biggl[
- \rho_{\f} \bv{b}_{\f} +\xi_{\f} \grad p + \xi_{\f}^{2}\frac{\mu_{\f}}{k_{\s}} (\bv{v}_{\f} - \bv{v}_{\s})\biggr].
\end{equation}
After pulling the above expression to the computational domain $B_{\s}$ we obtain the following problem:
\begin{problem}[\ac{ALE} Framework\,---\,Fluid Velocity Quasi-Static Stabilized Formulation]
\label{dual pb abstract weak ALE stabilized static}
Given the same data as in Problem~\ref{problem: ALE strong}, and for quasi-static motions, i.e., motions for which the material accelerations are negligible, find $\bv{u}_{\s} \in \mathcal{V}^{\bv{u}_{\s}}$, $\bv{v}_{\s} \in \mathcal{V}^{\bv{v}_{\s}}$, $\bv{v}_{\f} \in \mathcal{V}^{\bv{v}_{\f}}$, and $p \in \mathcal{V}^{p}$ such that
\begin{align}
\label{eq: dual ALE weak form stabilized us}
\mathcal{M}_{\sad\sap}\partial_{t}\bv{u}_{\s}
-
\mathcal{M}_{\sad\sbp}\bv{v}_{\s}
&= \bv{0},
\\
\label{eq: dual ALE weak form stabilized vs}
-
\tfrac{1}{2}
\transpose{\mathcal{B}}_{\sdd\sbp}(\bv{u}_{\s};\xi_{R_{\s}}) p
-
\tfrac{1}{2}
\transpose{\mathcal{D}}_{\scd\sbp}(\bv{u}_{\s};\xi_{R_{\s}})\bv{v}_{\f}
+
\tfrac{1}{2}
\mathcal{D}_{\sbd\sbp}(\bv{u}_{\s};\xi_{R_{\s}})\bv{v}_{\s}
+
\mathcal{A}_{\sbd}(\bv{u}_{\s};\xi_{R_{\s}})
-
\transpose{\mathcal{S}}_{\sdd\sbp}(\bv{u}_{\s})p
&= \tilde{\mathcal{F}}_{\sbd}(\xi_{R_{\s}}),
\\
\label{eq: dual ALE weak form stabilized vf}
\tfrac{1}{2}
\transpose{\mathcal{B}}_{\sdd\scp}(\bv{u}_{\s};\xi_{R_{\s}}) p
+
\tfrac{1}{2}
\mathcal{D}_{\scd\scp}(\bv{u}_{\s};\xi_{R_{\s}})\bv{v}_{\f}
-
\tfrac{1}{2}
\mathcal{D}_{\scd\sbp}(\bv{u}_{\s};\xi_{R_{\s}})\bv{v}_{\s}
&= \tfrac{1}{2} \mathcal{F}_{\scd}(\bv{u}_{\s}),
\\
\label{eq: dual ALE weak form stabilized p}
\tfrac{1}{2}\mathcal{B}_{\sdd\sbp}(\bv{u}_{\s}) \bv{v}_{\s}
+
\tfrac{1}{2}
\mathcal{B}_{\sdd\scp}(\bv{u}_{\s}) \bv{v}_{\f}
+
\mathcal{S}_{\sdd\sbp}(\bv{u}_{\s}) \bv{v}_{\s}
+ \mathcal{K}_{\sdd\sdp}(\bv{u}_{\s}) p
&= \tilde{\mathcal{F}}_{\sdd}(\bv{u}_{\s}),
\end{align}
where the operators $\tilde{\mathcal{F}}_{\sbd}(\bv{u}_{\s};\xi_{R_{\s}})$, $\mathcal{K}_{\sdd\sdp}(\bv{u}_{\s})$, and $\tilde{\mathcal{F}}_{\sdd}(\bv{u}_{\s})$ are defined in Eqs.~\eqref{eq: F4 tilde def}, \eqref{eq: K44 def}, and~\eqref{eq: F2 tilde def}, respectively.  Furthermore, if $\fv \in \mathcal{V}^{\vf}$ is required as part of the solution, this is computed as described in Remark~\ref{remark: filtration velocity calculation}.
\end{problem}
\begin{remark}[Coercivity of the operator $\mathcal{K}_{\sdd\sdp}(\bv{u}_{\s})$]
In view of its definition in Eq.~\eqref{eq: K44 def}, the operator $\mathcal{K}_{\sdd\sdp}(\bv{u}_{\s})$ is, strictly speaking, positive semi-definite, and therefore \emph{not} coercive.
\end{remark}

Now that Problem~\ref{dual pb abstract weak ALE stabilized static} has been stated, we present the following theorem:
\begin{theorem}[Stability of the Formulation in Problem~\ref{dual pb abstract weak ALE stabilized static}]
\label{theorem: stability of stabilized 1}
The quasi-static formulation in Problem~\ref{dual pb abstract weak ALE stabilized static} is stable.
\end{theorem}
\begin{proof}
Assuming that $\Gamma_{\s}^{D} = \partial B_{\s}$, letting $\bv{v}_{\s} \in \mathcal{V}^{\bv{v}_{\s}}_{0}$, $\test{\bv{v}}_{\s} = \bv{v}_{\s}$, $\test{\bv{v}}_{\f} = \bv{v}_{\f}$, and $\test{p} = p$, and suppressing the external force fields $\bv{b}_{\s}$ and $\bv{b}_{\f}$, Lemma~\ref{lemma: theorem of power expended} implies that
\begin{equation}
\label{eq: stabilized 1 stability proof}
\frac{\nsd{}}{\nsd{t}} \mathscr{W} + |p|_{d}^{2} =  -\tfrac{1}{2} \mathscr{D} \leq 0,
\end{equation}
where $|p|_{d}$ is a semi-norm over $\mathcal{V}^{p}$, equivalent to the $H^{1}$-seminorm, defined as follows:
\begin{equation}
\label{eq: pd semi norm def}
|p|_{d}^{2} \coloneqq \prescript{}{(\mathcal{V}^{p})^{*}}{\bigl\langle}
\mathcal{K}_{\sdd\sdp}(\bv{u}_{\s};\xi_{R_{\s}}) p, p
\big\rangle_{\mathcal{V}^{p}} 
\quad
\forall p \in \mathcal{V}^{p}
.
\end{equation}
It is important to note that, similarly to the traditional Navier-Stokes problem, Problem~\ref{lemma: theorem of power expended} does not admit a unique solution under pure Dirichlet boundary conditions for $\us$ (and therefore $\vs$).  As is often done for the Navier-Stokes problem, we restore uniqueness by removing the kernel of the gradient, e.g., adding a scalar constraint on the field $p$.  Specifically, we demand that $p$ satisfy a zero mean constraint:
\begin{equation}
\label{eq: zero mean constraint}
\frac{1}{|B_{\s}|}\int_{B_{\s}} p = 0.
\end{equation}
In this case, $|p|_{d}$ can be taken as a norm and the claim follows.
\end{proof}
\begin{remark}[The above stabilization strategy does not benefit the dynamic problem]
\label{remark: stabilization does not work for dynamics}
If one were to replace the term $-\rho_{\f}\bv{b}_{\f}$ in Eq.~\eqref{eq: stabilization term 1} by the term $\rho_{\f}(\partial_{t}\vf + \vf \cdot \grad\vf - \bv{b}_{\f})$, one obtains a consistent formulation for the fully dynamic problem.  However, one cannot prove the stability of this formulation because the inertia force terms are intrinsically dependent on the system's motion and as such not controllable in the same way that the term $\bv{b}_{\f}$ was in the proof of Theorem~\ref{theorem: stability of stabilized 1}.
\end{remark}

As shown in the result section, the quasi-static formulation was implemented and provides satisfactory results.  Also, the authors verified that, as observed in Remark~\ref{remark: stabilization does not work for dynamics}, there is no concrete beneficial effect from extending the use of the above stabilization technique in the dynamic case.  At the same time, they have observed that an un-stabilized implementation of Problem~\ref{dual pb abstract weak ALE}  in its fully dynamic case does not lead to a useful formulation due to strong difficulties in determining finite element spaces that might satisfy the $\inf\sup$ condition.  A mixed formulation that allows for an easier identification of such spaces is presented in the next Section.

\section{A formulation based on filtration velocity}
\label{section: filtration velocity}
The problem at hand is characterized by five fields of physical interest: $\us$, $\vs$, $\vf$, $\fv$, and $p$.  Hence, depending on the data and on the flow regime, the problem can be formulated in several other ways.  Here we consider one such formulations that includes the filtration velocity as a primary unknown as opposed to being determined in a post-processing step.  In fact, for the quasi-static case of this formulation, it is the fluid velocity that can be considered a ``postprocessing quantity.'' After presenting the formulation in question,  we discuss its stabilization in the quasi-static case and its behavior in the fully dynamic (un-stabilized case).

We begin by revisiting the equations expressing the balance of momentum.  Referring to Eqs.~\eqref{eq: BOM solid} and~\eqref{eq: BOM fluid}, we choose to rewrite the expression of the balance of momentum as follows:
\begin{align}
\label{eq: BOM solid fv}
\bv{0} &=
\rho_{\s} (D_{t_{\s}} \vs - \bv{b}_{\s})
+
\rho_{\f} (D_{t_{\f}} \vf - \bv{b}_{\f})
+ \grad p - \ldiv\tensor{T}^{e} \quad \text{in $B_{t}$},
\\
\label{eq: BOM fluid fv}
\bv{0} &=
\rho_{\f}^{*} (D_{t_{\f}} \vf - \bv{b}_{\f}) + \grad p + \frac{\mu_{\f}}{k_{\s}} \fv \quad \text{in $B_{t}$},
\end{align}
where Eq.~\eqref{eq: BOM solid fv} is obtained by summing Eqs.~\eqref{eq: BOM solid} and~\eqref{eq: BOM fluid} and using the saturation condition, and Eq.~\eqref{eq: BOM fluid fv} is obtained by eliminating the factor $\xi_{R_{\s}}$ from Eq.~\eqref{eq: BOM fluid} and by then using the definition of filtration velocity in Eq.~\eqref{eq: filtration def}.

Using these equations as governing equations, and imitating the derivation of Problem~\ref{dual pb abstract weak ALE}, we then have the following
\begin{problem}[\ac{ALE} Framework\,---\,Filtration Velocity Formulation]
\label{dual pb abstract weak ALE fv}
Given the same data as in Problem~\ref{problem: ALE strong}, find $\bv{u}_{\s} \in \mathcal{V}^{\bv{u}_{\s}}$, $\bv{v}_{\s} \in \mathcal{V}^{\bv{v}_{\s}}$, $\fv \in \mathcal{V}^{\bv{v}_{\f}}$, $\bv{v}_{\f} \in \mathcal{V}^{\bv{v}_{\f}}$, and $p \in \mathcal{V}^{p}$ such that
\begin{align}
\label{eq: fv dual ALE weak form us}
\mathcal{M}_{\sad\sap}\partial_{t}\bv{u}_{\s}
-
\mathcal{M}_{\sad\sbp}\bv{v}_{\s}
&= \bv{0},
\\
\label{eq: fv dual ALE weak form vs}
\begin{aligned}[b]
&\mathcal{M}_{\sbd\sbp}(\xi_{R_{\s}})\partial_{t}\bv{v}_{\s}
+ \mathcal{M}_{\sbd\scp}(\bv{u}_{\s}; \xi_{R_{\s}}) \partial_{t}\vf
\\
&\qquad\qquad\;
+ \mathcal{N}_{\sbd\scp}(\bv{u}_{\s}, \vf; \xi_{R_{\s}}) \vf
- \mathcal{N}_{\sbd\sbp}(\bv{u}_{\s}, \vf; \xi_{R_{\s}}) \vs
\\
&\qquad\qquad\;
-\transpose{\mathcal{B}}_{\sdd\sbp}(\bv{u}_{\s}; 1) p
+
\mathcal{A}_{\sbd}(\bv{u}_{\s};\xi_{R_{\s}})
-
\transpose{\mathcal{S}}_{\sdd\sbp}(\bv{u}_{\s})p
\end{aligned}
&=  \mathcal{F}_{\sbd}(\bv{u}_{\s};\xi_{R_{\s}}) + \tilde{\mathcal{F}}_{\sbd}(\bv{u}_{\s};\xi_{R_{\s}}),
\\
\label{eq: fv dual ALE weak form vf}
\begin{aligned}[b]
  &\mathcal{M}_{\scd\scp}(\us;0) \partial_{t}\vf
+ \mathcal{N}_{\scd\scp}(\us, \vf;0) \vf
- \mathcal{N}_{\scd\sbp}(\us, \vf;0) \vs
\\
&\qquad\qquad\;
+ \transpose{\mathcal{B}}_{\sdd\scp}(\bv{u}_{\s}; 0) p
+ \mathcal{D}_{\scd\scp}(\bv{u}_{\s}; 0) \fv
\end{aligned}
&= \mathcal{F}_{\scd}(\bv{u}_{\s};0),
\\
\label{eq: fv dual ALE weak form p}
\mathcal{B}_{\sdd\sbp}(\bv{u}_{\s};1) \bv{v}_{\s}
+
\mathcal{B}_{\sdd\scp}(\bv{u}_{\s};0) \fv
+
\mathcal{S}_{\sdd\sbp}(\bv{u}_{\s}) \bv{v}_{\s}
&= \mathcal{F}_{\sdd}(\bv{u}_{\s}),
\\
\label{eq: fv dual ALE weak form vf}
\check{\mathcal{M}}_{\scd\scp}(\us; \xi_{R_{\s}}) \vf - \check{\mathcal{M}}_{\scd\sbp}(\us; \xi_{R_{\s}}) \vs - \check{\mathcal{M}}_{\scd\scp}(\us; 0) \fv &= \bv{0},
\end{align}
where the operators $\mathcal{M}_{\sbd\scp}(\bv{u}_{\s}; \xi_{R_{\s}})$, $\mathcal{N}_{\sbd\scp}(\bv{u}_{\s}, \vf; \xi_{R_{\s}})$, and $\mathcal{N}_{\sbd\sbp}(\bv{u}_{\s}, \vf; \xi_{R_{\s}})$ are defined in Eqs.~\eqref{eq: M23 def}, \eqref{eq: N23 def}, and~\eqref{eq: N22 def}, respectively.
\end{problem}
The corresponding quasi-static problem can be written directly in terms of $\us$, $\vs$, $\fv$, and $p$, with $\vf$ computed via an $L^{2}$-projection.  Specifically, we have
\begin{problem}[\ac{ALE} Framework\,---\,Filtration Velocity Quasi-Static Formulation]
\label{dual pb abstract weak ALE fv quasi static}
Given the same data as in Problem~\ref{problem: ALE strong}, find $\bv{u}_{\s} \in \mathcal{V}^{\bv{u}_{\s}}$, $\bv{v}_{\s} \in \mathcal{V}^{\bv{v}_{\s}}$, $\fv \in \mathcal{V}^{\bv{v}_{\f}}$, and $p \in \mathcal{V}^{p}$ such that
%
\begin{align}
\label{eq: fv dual ALE weak form us quasi static}
\mathcal{M}_{\sad\sap}\partial_{t}\bv{u}_{\s}
-
\mathcal{M}_{\sad\sbp}\bv{v}_{\s}
&= \bv{0},
\\
\label{eq: fv dual ALE weak form vs quasi static}
-\transpose{\mathcal{B}}_{\sdd\sbp}(\bv{u}_{\s}; 1) p
+
\mathcal{A}_{\sbd}(\bv{u}_{\s};\xi_{R_{\s}})
-
\transpose{\mathcal{S}}_{\sdd\sbp}(\bv{u}_{\s})p
&=  \mathcal{F}_{\sbd}(\bv{u}_{\s};\xi_{R_{\s}}) + \tilde{\mathcal{F}}_{\sbd}(\bv{u}_{\s};\xi_{R_{\s}}),
\\
\label{eq: fv dual ALE weak form vf  quasi static}
\transpose{\mathcal{B}}_{\sdd\scp}(\bv{u}_{\s}; 0) p
+ \mathcal{D}_{\scd\scp}(\bv{u}_{\s}; 0) \fv
&= \mathcal{F}_{\scd}(\bv{u}_{\s};0),
\\
\label{eq: fv dual ALE weak form p  quasi static}
\mathcal{B}_{\sdd\sbp}(\bv{u}_{\s};1) \bv{v}_{\s}
+
\mathcal{B}_{\sdd\scp}(\bv{u}_{\s};0) \fv
+
\mathcal{S}_{\sdd\sbp}(\bv{u}_{\s}) \bv{v}_{\s}
&= \mathcal{F}_{\sdd}(\bv{u}_{\s}),
\end{align}
where, once the solution to the above problem is available, the field $\vf \in \mathcal{V}^{\bv{v}_{\f}}$ can be recovered as
\begin{equation}
\label{eq: fv dual ALE weak form vf  quasi static}
\vf = \bigl[\check{\mathcal{M}}_{\scd\scp}(\us; \xi_{R_{\s}})\bigr]^{-1}
\bigl(
\check{\mathcal{M}}_{\scd\sbp}(\us; \xi_{R_{\s}}) \vs + \check{\mathcal{M}}_{\scd\scp}(\us; 0) \fv \bigl).
\end{equation}
\end{problem}

Problem~\ref{dual pb abstract weak ALE fv quasi static} can be stabilized using the same technique followed to obtain Problem~\ref{dual pb abstract weak ALE stabilized static}.  That is, referring to Eq.~\eqref{eq: BOM fluid fv}, we add to the formulation of Problem~\ref{dual pb abstract weak ALE fv quasi static} the terms resulting from the pull-back to $B_{\s}$ of the following expression:
\begin{equation}
\label{eq: stabilization term 2}
\int_{B_{t}}
\tfrac{1}{2}
\biggl(
\frac{k_{\s}}{\mu_{\f}} \grad\tilde{p} - \tilde{\bv{v}}_{\text{flt}}
\biggr)
\cdot \biggl[
- \rho_{\f}^{*} \bv{b}_{\f} + \grad p + \frac{\mu_{\f}}{k_{\s}} \fv \biggr].
\end{equation}
Doing so, yields:
\begin{problem}[\ac{ALE} Framework\,---\,Filtration Velocity Quasi-Static Stabilized Formulation]
\label{dual pb abstract weak ALE fv quasi static stabilized}
Given the same data as in Problem~\ref{problem: ALE strong}, find $\bv{u}_{\s} \in \mathcal{V}^{\bv{u}_{\s}}$, $\bv{v}_{\s} \in \mathcal{V}^{\bv{v}_{\s}}$, $\fv \in \mathcal{V}^{\bv{v}_{\f}}$, and $p \in \mathcal{V}^{p}$ such that
%
\begin{align}
\label{eq: fv dual ALE weak form us quasi static stabilized}
\mathcal{M}_{\sad\sap}\partial_{t}\bv{u}_{\s}
-
\mathcal{M}_{\sad\sbp}\bv{v}_{\s}
&= \bv{0},
\\
\label{eq: fv dual ALE weak form vs quasi static stabilized}
-\transpose{\mathcal{B}}_{\sdd\sbp}(\bv{u}_{\s}; 1) p
+
\mathcal{A}_{\sbd}(\bv{u}_{\s};\xi_{R_{\s}})
-
\transpose{\mathcal{S}}_{\sdd\sbp}(\bv{u}_{\s})p
&=  \mathcal{F}_{\sbd}(\bv{u}_{\s};\xi_{R_{\s}}) + \tilde{\mathcal{F}}_{\sbd}(\bv{u}_{\s};\xi_{R_{\s}}),
\\
\label{eq: fv dual ALE weak form vf  quasi static stabilized}
-\tfrac{1}{2}\transpose{\mathcal{B}}_{\sdd\scp}(\bv{u}_{\s}; 0) p
+ \tfrac{1}{2} \mathcal{D}_{\scd\scp}(\bv{u}_{\s}; 0) \fv
&= \tfrac{1}{2} \mathcal{F}_{\scd}(\bv{u}_{\s};0),
\\
\label{eq: fv dual ALE weak form p  quasi static stabilized}
\mathcal{B}_{\sdd\sbp}(\bv{u}_{\s};1) \bv{v}_{\s}
+
\tfrac{1}{2}\mathcal{B}_{\sdd\scp}(\bv{u}_{\s};0) \fv
+
\mathcal{K}_{\sdd\sdp}(\bv{u}_{\s}) p
+
\mathcal{S}_{\sdd\sbp}(\bv{u}_{\s}) \bv{v}_{\s}
&= \mathcal{F}_{\sdd}(\bv{u}_{\s}),
\end{align}
where, once the solution to the above problem is available, the field $\vf \in \mathcal{V}^{\bv{v}_{\f}}$ can be recovered using Eq.~\eqref{eq: fv dual ALE weak form vf  quasi static}.
%
%
\end{problem}
In relation to Problem~\ref{dual pb abstract weak ALE fv quasi static stabilized} we have the following result:
\begin{theorem}[Stability of the Formulation in Problem~\ref{dual pb abstract weak ALE fv quasi static stabilized}]
\label{theorem: stability of stabilized 2}
The quasi-static formulation in Problem~\ref{dual pb abstract weak ALE fv quasi static stabilized} is stable.
\end{theorem}
\begin{proof}
The proof is omitted in that it follows the same steps presented in the proof of Theorem~\ref{theorem: stability of stabilized 1}.
\end{proof}
\begin{remark}[Difference between the Fluid Velocity and Filtration Velocity Formulations]
\label{remark: one more stabilization consideration}
The formulation of Problems~\ref{dual pb abstract weak ALE fv}--\ref{dual pb abstract weak ALE fv quasi static stabilized} has some important differences relative to the formulations introduced earlier.  Specifically, we note that various operators in Problems~\ref{dual pb abstract weak ALE fv}--\ref{dual pb abstract weak ALE fv quasi static stabilized} no longer depend on the solid's referential volume fraction.  In particular, we note that the coercivity of the operator $\mathcal{D}_{\scd\scp}$ is unaffected by said volume fraction.  This indicates that one should expect in the corresponding \ac{FE} implementation a somewhat more robust behavior of the filtration velocity formulation when it comes to accuracy as a function of $\xi_{R_{\s}}$.
\end{remark}

\section{Discrete approximation}
The abstract formulations were approximated by defining a triangulation $B_{\s_{h}}$ with diameter $h$ of the domain $B_{\s}$ into closed cells $K$ (triangle or quadrilaterals in 2D, and tetrahedra or hexahedra in 3D) such that
\begin{enumerate}
\item
$\overline{B_{\s}} = \cup\{ K \in B_{\s_{h}}\}$;

\item
For any two cells $K_{i}, K_{j} \in B_{\s_{h}}$, $K_{i}\cap K_{j}$ consists only of common faces, edges, or vertices;

\item
$B_{\s_{h}}$ respects the decomposition of the boundary in its Neumann and Dirichlet subsets. 
\end{enumerate}
On $B_{\s_{h}}$ we define the finite dimensional subspaces $\mathcal{V}^{\bv{u}_{\s}}_{h} \subset \mathcal{V}^{\bv{u}_{\s}}$, $\mathcal{V}^{\bv{v}_{\s}}_{h} \subset \mathcal{V}^{\bv{v}_{\s}}$, $\mathcal{V}^{\bv{v}_{\f}}_{h} \subset \mathcal{V}^{\bv{v}_{\f}}$, and $\mathcal{V}^{p}_{h} \subset \mathcal{V}^{p}$ as
\begin{alignat}{2}
\label{eq: FESpace us}
\mathcal{V}^{\bv{u}_{\s}}_{h}
&\coloneqq 
\left\{
\bv{u}_{\s_{h}} \mid (\bv{u}_{\s_{h}})_{i}|_{K} \in \mathcal{P}_{\us}^{n_{\us}}(K), i = 1, \dots, d, K \in B_{\s_{h}}
\right\}
&~
&\equiv \Span \left( \tilde{\bv{u}}_{\s}^{i} \right)_{i = 1}^{N^{\us}_{h}}
,
\\
\label{eq: FESpace vs}
\mathcal{V}^{\bv{v}_{\s}}_{h}
&\coloneqq 
\left\{
\bv{v}_{\s_{h}} \mid (\bv{v}_{\s_{h}})_{i}|_{K} \in \mathcal{P}_{\vs}^{n_{\vs}}(K), i = 1, \dots, d, K \in B_{\s_{h}}
\right\}
&~
&\equiv \Span \left( \tilde{\bv{v}}_{\s}^{i} \right)_{i = 1}^{N^{\vs}_{h}}
,
\\
\label{eq: FESpace vf}
\mathcal{V}^{\bv{v}_{\f}}_{h}
&\coloneqq 
\left\{
\bv{v}_{\f_{h}} \mid (\bv{v}_{\f_{h}})_{i}|_{K} \in \mathcal{P}_{\vf}^{n_{\vf}}(K), i = 1, \dots, d, K \in B_{\s_{h}}
\right\}
&~
&\equiv \Span \left( \tilde{\bv{v}}_{\f}^{i} \right)_{i = 1}^{N^{\vf}_{h}}
,
\\
\label{eq: FESpace p}
\mathcal{V}^{p}_{h}
&\coloneqq 
\left\{
p_{h} \mid p_{h}|_{K} \in \mathcal{P}_{p}^{n_{p}}(K), K \in B_{\s_{h}}
\right\}
&~
&\equiv \Span \left( \tilde{p}^{\,i} \right)_{i = 1}^{N^{p}_{h}}
,
\end{alignat}
where the notation $\phi|_{K}$ indicates the restriction of the scalar field $\phi$ to the cell $K$, and where $(\bv{w})_{i}$ indicates the $i$-th scalar component of the vector field $\bv{w}$.  Furthermore, the notation $\mathcal{P}_{\phi}^{n_{\phi}}(K)$ indicates the polynomial space of degree $n_{\phi}$ on the cell $K$, $\tilde{\bv{u}}_{\s}^{i}$ is the $i$-th element of a selected basis in $\mathcal{V}^{\bv{u}_{\s}}_{h}$, the latter having dimension $N^{\us}_{h}$, and the remainder of the symbols can be interpreted in a similar manner. 

We note that the chosen finite element spaces are included in the pivot spaces for the time derivatives of the fields in the formulation.  Hence, to avoid a proliferation of symbols, we will use said spaces for both the primary fields and their time rates.

Recalling that our problem is time dependent, to represent, say, $\ush(\bv{X},t)$ we write
\begin{equation}
\label{eq: representation of time dependent fields}
\ush(\bv{X},t) = \sum_{i = 1}^{N^{\us}_{h}} u^{i}(t) \tilde{\bv{u}}_{\s}^{i}(\bv{X}),
\end{equation}
where $u^{i}(t)$ is the time dependent coefficient for the $i$-th base element of $\mathcal{V}^{\bv{u}_{\s}}_{h}$.  This is a very common strategy (see, e.g., \citealp{Hughes2000The-Finite-Element-0}) that allows us to define matrix representations for the operators in Problems~\ref{dual pb abstract weak ALE stabilized static} and~\ref{dual pb abstract weak ALE fv} in the discrete case.  For example, the operator $\mathcal{M}_{\sad\sap}$ defined in Eq.~\eqref{aeq: M11 def} can be represented as a matrix $N^{\us}_{h} \times N^{\us}_{h}$ matrix with $ij$-the element $\mathcal{M}_{\sad\sap}^{ij}$ given by
\begin{equation}
\label{eq: finite dimensional operator example}
\mathcal{M}_{\sad\sap}^{ij} \coloneqq 
\prescript{}{(\mathcal{V}^{\bv{u}_{\s}})^{*}}{\bigl\langle}
\mathcal{M}_{\sad\sap}\tilde{\bv{u}}_{\s}^{i},\tilde{\bv{u}}_{\s}^{j}
\big\rangle_{\mathcal{V}^{\bv{u}_{\s}}},
\end{equation}
where $1 \leq i,j \leq N^{\us}_{h}$.  Extending these considerations to all the operators presented in this paper, the finite dimensional version of the Problems~\ref{dual pb abstract weak ALE stabilized static} and~\ref{dual pb abstract weak ALE fv} is obtained by simply replacing the fields $\us$, $\vs$, $\vf$, $p$, and $\fv$ (along with the corresponding test functions) by their finite dimensional counterparts $\ush$, $\vsh$, $\vfh$, $ph$, and $\fvh$. With this in mind, we then have the following
\begin{theorem}[Semi-discrete strong consistency]
\label{theorem: strong consistency}
For any conforming approximation, i.e., whenever $\mathcal{V}^{\bv{u}_{\s}}_{h} \subset \mathcal{V}^{\bv{u}_{\s}}$, $\mathcal{V}^{\bv{v}_{\s}}_{h} \subset \mathcal{V}^{\bv{v}_{\s}}$, $\mathcal{V}^{\bv{v}_{\f}}_{h} \subset \mathcal{V}^{\bv{v}_{\f}}$, and $\mathcal{V}^{p}_{h} \subset \mathcal{V}^{p}$, the discrete formulations of Problems~\ref{dual pb abstract weak ALE stabilized static}--\ref{dual pb abstract weak ALE fv quasi static stabilized}, are strongly consistent.
\end{theorem}
The proof of Theorem~\ref{theorem: strong consistency} is omitted as it mimics very well-established results as can be found in \cite{Brenner2002The-Mathematical-Theory-0}.

\subsection{Semi-discrete stability estimates for the quasi-static formulations}
Under the assumption that $\tensor{F}_{\s}$ is invertible, all of the operators that appear in the  quasi-static version of the problems introduced earlier, both in the fluid velocity and in the filtration velocity formulations, have been defined so as to preserve their properties intact in the passage to the finite dimensional context.  Furthermore, so long as the field $\ush$ is Lipschitz continuous, the expression in Eq.~\eqref{eq: weak form balance laws energy estimate 6} of Lemma~\ref{lemma: theorem of power expended} remains valid in the discrete case.  Therefore we have proven the following result:
\begin{theorem}[Stability for the Quasi-Static Discrete Formulation]
The discrete counterparts of Problems~\ref{dual pb abstract weak ALE stabilized static} and~\ref{dual pb abstract weak ALE fv quasi static stabilized}  are stable.
\end{theorem}

\section{Numerical results}
The formulations presented in this paper have been implemented in \comsol \citep{COMSOLCite}, which we have used as a \ac{FEM}-specific programming environment.  The implementation has been done via the ``Weak Form PDE'' interface.

As the problem studied in this paper is nonlinear, the performance of a formulation is expected to change somewhat in relation to the specific choice of values in the data of the problem.  In particular, we expect significant changes in response due to different choices of constitutive response function for the solid's elastic response and different choices of the referential volume fraction field $\xi_{R_{\s}}$.  With this in mind, we select a single set of data for our simulations and focus on the analysis of the performance of the formulations themselves and for a range of values of $\xi_{R_{\s}}$.

\subsection{Problem setup}
\subsubsection{Domain and data specification}
We present results based a planar problem.  With reference to Fig.~\ref{fig: UnitSquare},
\begin{figure}[htb]
    \centering
    \includegraphics{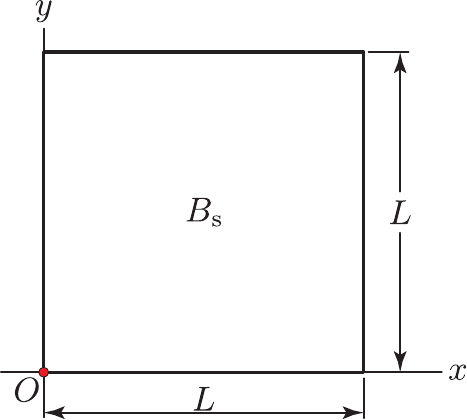}
    \caption{Domain used in the determination of convergence rates.}
    \label{fig: UnitSquare}
\end{figure}
the domain $B_{\s}$ is taken to be a square of side $L = \np[m]{1}$.  We discussed the fact that the coercivity of certain operators is strongly influenced by the volume fraction of the fluid.  For this reason, we select a relatively low value of porosity to focus on this aspect. Specifically, we set the porosity in the reference configuration of the solid to be uniform and equal to 10\%, which implies $\xi_{R_{\s}} = 0.9$.\footnote{This value is certainly a conservative lower bound of what is currently accepted for brain parenchyma (see \citealp{Sykova2008Diffusion-in-Br0} for a thorough discussion; also \citealp{Korogod2015Ultrastructural-Analysis-0}).}  We select the elastic response of the solid skeleton to be neo-Hookean (cf.\ \citealp{Ogden1997Non-Linear-Elastic-0}) with unit shear modulus $G = \np[Pa]{1}$, so that, referring to the second of Eqs.~\eqref{eq: Piola transform of Te}, we have
\begin{equation}
\label{eq: simulations 1st PK}
\tensor{P}^{e} = 0.9 \, G \, \tensor{F}_{\s}.
\end{equation}
The rest of the constitutive parameters are set to unit values, as indicated in Table~\ref{table: constitutive parameters list}.
\begin{table}[ht]
\caption{\label{table: constitutive parameters list}%
Summary of the parameters values in the calculation of convergence rates.%
}
\begin{minipage}{\textwidth}
\centering
\renewcommand{\footnoterule}{\vspace{-7pt}\rule{0pt}{0pt}}
\renewcommand{\arraystretch}{1.2}
\begin{tabular}{l l}
\toprule
Quantity & Value
\\
\midrule
$\rho_{\s}^{*}$ & \np[kg/m^{3}]{1} \\
$\rho_{\f}^{*}$ & \np[kg/m^{3}]{1} \\
$\mu_{\f}$ & \np[Pa\ucdot s]{1} \\
$k_{\s}$ & \np[m^2]{1} \\
$G$ & \np[Pa]{1} \\
\bottomrule
\end{tabular}
\end{minipage}
\end{table}
The analysis of the results is carried out via the Method of Manufactured Solutions \citep{Salari2000Code-Verification-0}.  The chosen manufactured solution is
\begin{align}
\bv{u}_{\s} &= u_{0} \sin(2 \pi t/t_{0}) \biggl[\cos\biggl(2 \pi \frac{x + y}{L}\biggr)\,\ui + \sin\biggl(2 \pi \frac{x - y}{L}\biggr)\,\uj\biggr],
\\
\vs &= \partial_{t} \bv{u}_{\s},
\\
\vf &= v_{0} \cos(2 \pi t/t_{0}) \biggl[\sin\biggl(2 \pi \frac{x^{2} + y^{2}}{L^{2}}\biggr)\,\ui +
\cos\biggl(2 \pi \frac{x^{2} - y^{2}}{L^{2}}\biggr)\,\uj\biggr],
\\
\fv &= (1 - \xi_{R_{\s}}/J_{\s})(\vf-\vs),
\\
p &= p_{0} \sin(2 \pi t/t_{0}) \sin(2 \pi (x + y)/L),
\end{align}
where $u_{0} = \np[m]{0.01}$, $t_{0} = \np[s]{1}$, $v_{0} = \np[m/s]{1}$, $p_{0} = \np[Pa]{1}$, and where $J_{\s} = \det(\tensor{I} + \Grad \bv{u}_{\s})$.

\subsubsection{Solvers and time integration}
The formulations presented earlier yield time-dependent nonlinear \ac{DAE}.  We have adopted the default approach available in \comsol for such equations.  Specifically, the time-dependent aspect is implemented via the method of lines \citep{COMSOLCite,Schiesser1991The-Numerical-Method-0}.  The specific nonlinear solver used was IDAS \citep{Hindmarsh2005SUNDIALS-Suite-0}, implementing a variable-order variable-step-size \ac{BDF}.  The solver provided by IDAS is designed to solve \ac{DAE} systems of the type $F(t,y,y',p) = 0$. The \ac{BDF} method was configured so as to allow orders $1$--$5$ and a maximum time step size of $\np[s]{0.001}$.  For the linear solver we chose PARDISO 5.0.0 \citep{pardiso1,pardiso2,pardiso3}.

\subsubsection{Finite element choice and uniform refinement setup}
\label{subsubsection: interpolation order}
The triangulation over the solution's domain $B_{\s_{h}}$ consisted of triangular cells.  For each scalar component of the problem, the approximation spaces were piecewise Lagrange polynomials. The interpolation order was fixed to $2$ (i.e., second order Lagrange Polynomials) for all fields except the multiplier $p$. For the latter field, we have considered various interpolation orders, which will be indicated on a case by case basis and will be denoted by $\mathcal{P}^{r}_{p}$, where $r$ is the polynomial order of the interpolation.

The convergence rates were computed under uniform refinement of the solution's domain.  The uniformly refined meshes were automatically generated in \comsol by specifying that the values of the minimum and maximum element diameter $h$ be the same.  We note that the convergence rates are not uniform as a function of time, that is, they change somewhat depending on the time instant at which they are computed.  With this in mind, we present results pertaining to two time instants, namely, $t = \np[s]{0.7}$ and $t = \np[s]{1.0}$.  The reason for this choice is that $t = \np[s]{1.0}$ is the end of the time interval considered for the determination of the convergence rates, and $t = \np[s]{0.7}$ is representative of the worst convergence rates we have obtained in our calculations.

\subsection{Fluid velocity quasi-static formulation (Problem~\ref{dual pb abstract weak ALE stabilized static}) results}
Here we present the results for the quasi-static stabilized formulation based on Problem~\ref{dual pb abstract weak ALE stabilized static}.  We note that we are not presenting  results obtained with the corresponding non-stabilized formulation because we encountered severe problems in determining \ac{FE} spaces of Lagrange polynomials that would produce acceptable outcomes, thus justifying the need for stabilization. 

Figure~\ref{fig: Solutions}
\begin{figure}[htb]
    \centering
    \includegraphics{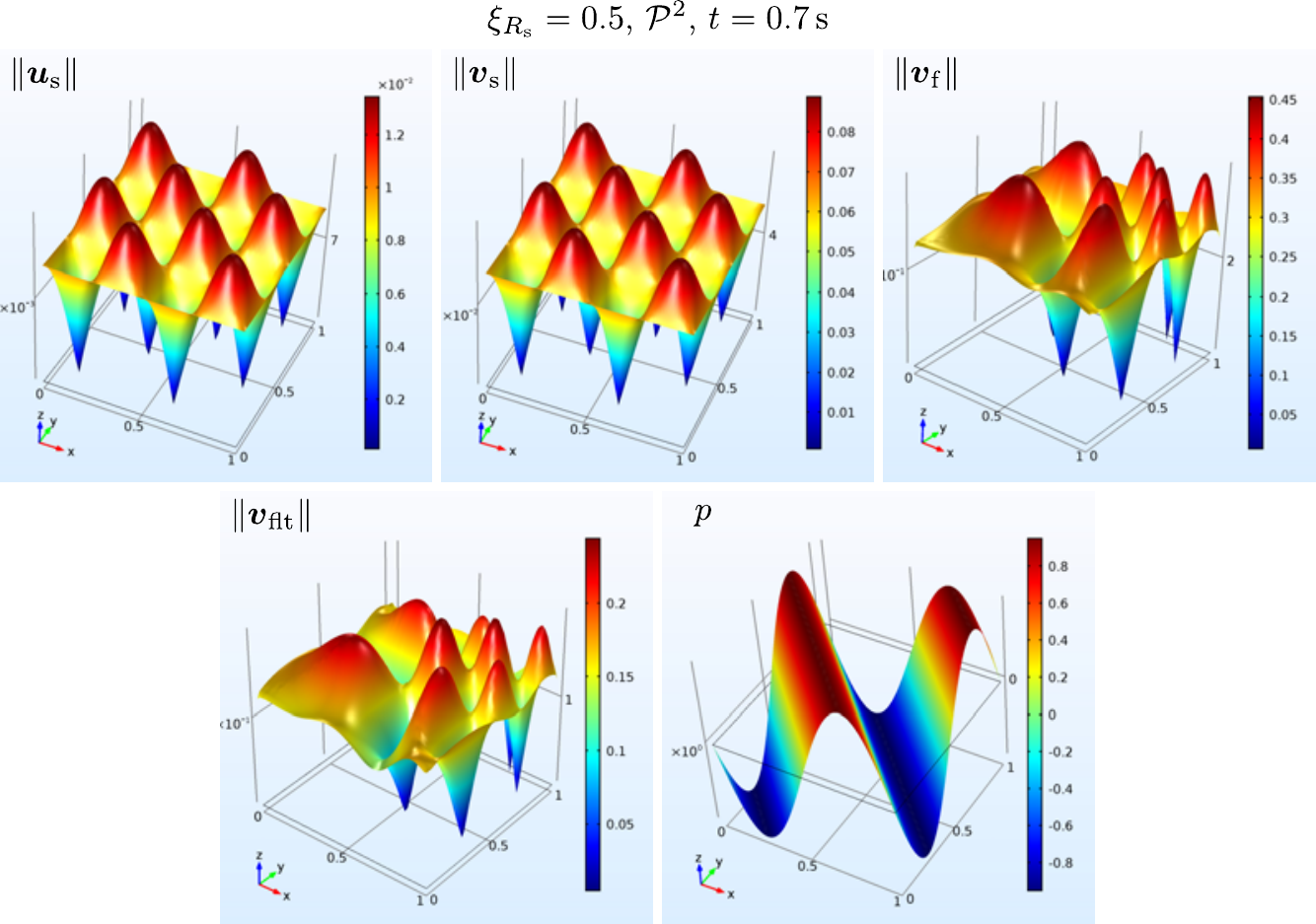}
    \caption{Fields $\|\us\|$, $\|\vs\|$, $\|\vf\|$, $\|\fv\|$, and $p$ at $t = \np[s]{0.7}$ obtained in a calculation using the fluid velocity formulation with $\xi_{R_{\s}} = 0.5$, and $h = (1/64)\,\mathrm{m} = \np[m]{0.015625}$ for a total number of degrees of freedom equal to $\np{187029}$.}
    \label{fig: Solutions}
\end{figure}
shows the magnitude of the fields $\|\us\|$, $\|\vs\|$, $\|\vf\|$, $\|\fv\|$, and $p$ at $t = \np[s]{0.7}$ obtained in a calculation using the stabilized fluid velocity formulation based on Problem~\ref{dual pb abstract weak ALE stabilized static} with $\xi_{R_{\s}} = 0.5$, and $h = (1/64)\,\mathrm{m} = \np[m]{0.015625}$ for a total number of degrees of freedom equal to $\np{187029}$. The mesh used consisted of (essentially) equal size triangles with second order Lagrange polynomials for all fields.  Because all convergence results presented in this paper are based on the same manufactured solution, the appearance of the plots in Fig.~\ref{fig: Solutions} turns out to be visually identical for all solutions regardless of formulation and order of approximation.  Hence, the above plots will not be shown again for other cases.

Figures~\ref{fig: CR xiR 05 p 2 t07} and~\ref{fig: CR xiR 05 p 2 t1}
\begin{figure}[htb]
    \centering
\includegraphics[scale=0.9]{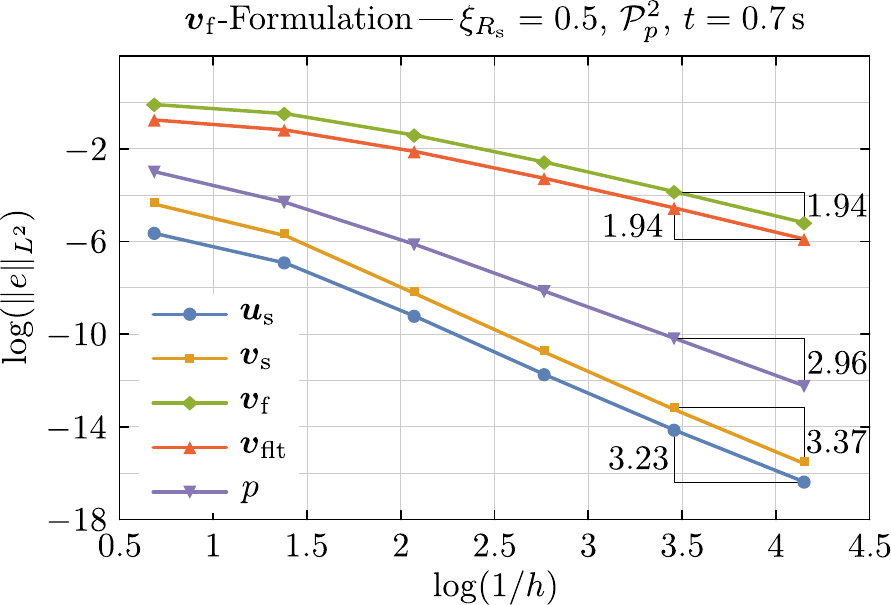}\hfill%
\includegraphics[scale=0.9]{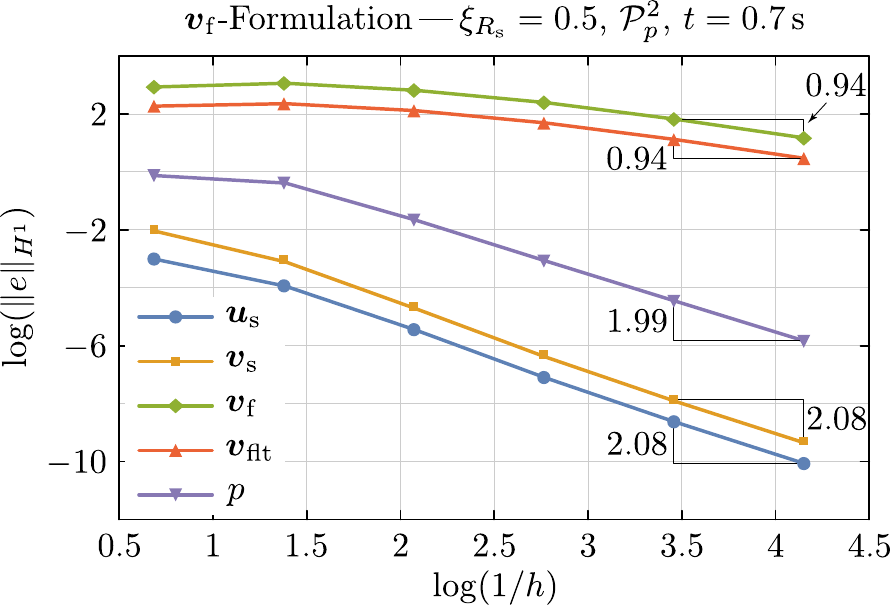}
    \caption{Convergence rates for the $L^{2}$-norm (left) and the $H^{1}$-(semi)norm (right) of the error at $t = \np[s]{0.7}$ obtained via the fluid velocity formulation with $\xi_{R_{\s}} = 0.5$ and $\mathcal{P}_{p}^{2}$.}
    \label{fig: CR xiR 05 p 2 t07}
\end{figure}
\begin{figure}[htb]
    \centering
\hfill%
\includegraphics[scale=0.9]{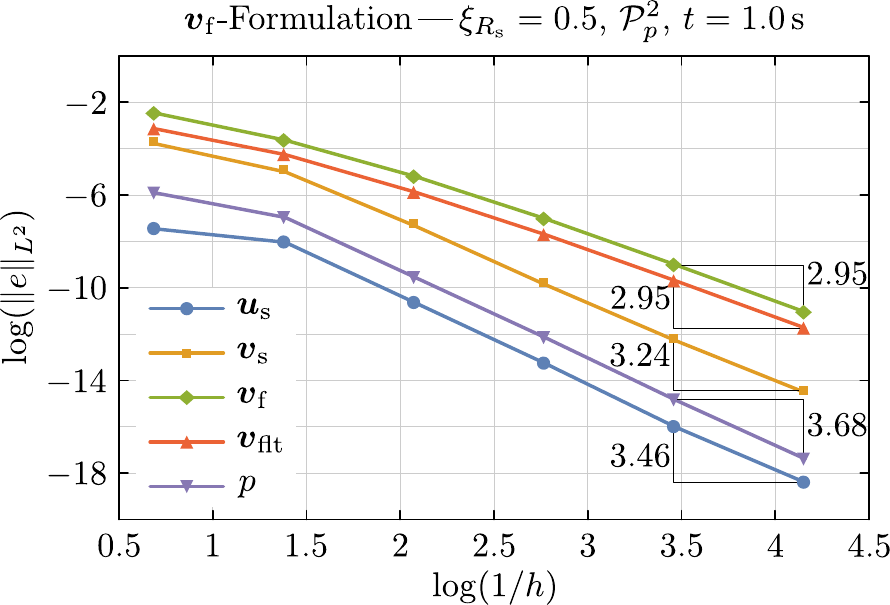}
\includegraphics[scale=0.9]{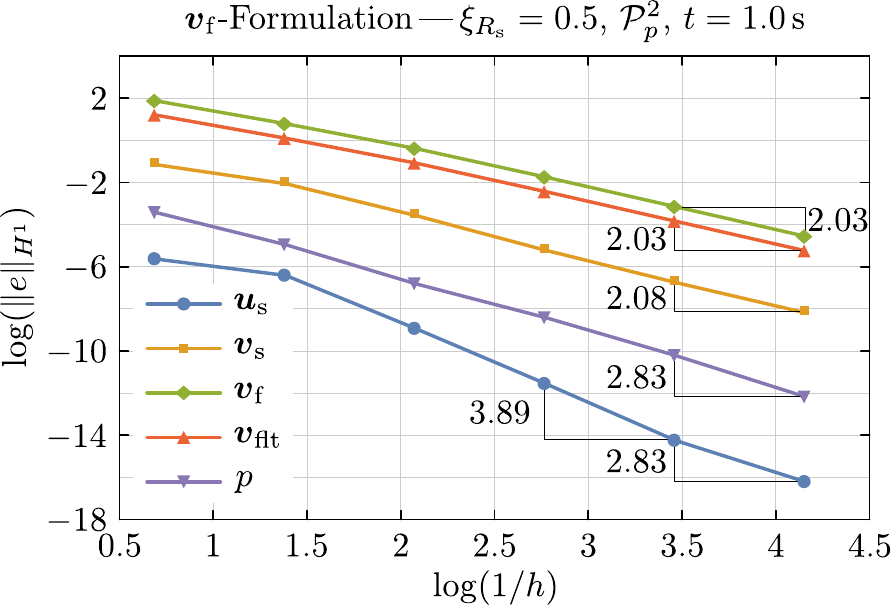}
    \caption{Convergence rates for the $L^{2}$-norm (left) and the $H^{1}$-(semi)norm (right) of the error at $t = \np[s]{1.0}$ obtained via the fluid velocity formulation with $\xi_{R_{\s}} = 0.5$ and $\mathcal{P}_{p}^{2}$.}
    \label{fig: CR xiR 05 p 2 t1}
\end{figure}
show the convergence rates for the case with $\xi_{R_{\s}} = 0.5$ and $\mathcal{P}_{p}^{2}$ (for the other fields see Section~\ref{subsubsection: interpolation order}) for $t = \np[s]{0.7}$ and $t = \np[s]{1.0}$, respectively.  We have reported these rates both in terms of the $L^{2}$ and $H^{1}$ error norms.  However, due to the formulation of the problem, there is no set expectation on the convergence rates of the $H^{1}$ norm for the fields $\vf$ and $\fv$.  As far as the values of the convergence rates are concerned, we refer to the results for the time-independent linear Darcy-flow problem through a rigid porous medium presented by \cite{Masud2002A-Stabilized-Mixed-0}.  We note that the velocity field in the cited work corresponds to the filtration velocity in the present paper.  With this in mind, \cite{Masud2002A-Stabilized-Mixed-0} found that for $k_{\s}/\mu = 1$, and for continuous pressure $6$-node triangles and continuous pressure $9$-node quadrilaterals the convergence rates in terms of the $L^{2}$-norm of $\fv$ approached the optimal value of $2$, and the $L^{2}$- and $H^{1}$-norms of $p$ approached the optimal value of $3$ and $2$, respectively (cf.\ Fig.~17 in \citealp{Masud2002A-Stabilized-Mixed-0}).  With reference to Figs.~\ref{fig: CR xiR 05 p 2 t07} and~\ref{fig: CR xiR 05 p 2 t1}, as well as Tables~\ref{table: CR xiR 05 t07} and~\ref{table: CR xiR 05 t1}, 
\begin{table}[ht]
\caption{\label{table: CR xiR 05 t07}%
Convergence rates for the stabilized fluid velocity formulation for $t = \np[s]{0.7}$ with $\xi_{R_{\s}} = 0.5$ and$\mathcal{P}_{p}^{2}$ corresponding to Fig.~\ref{fig: CR xiR 05 p 2 t07}.%
}
\begin{minipage}{\textwidth}
\centering
\renewcommand{\footnoterule}{\vspace{-7pt}\rule{0pt}{0pt}}
\renewcommand{\arraystretch}{1.2}
\begin{tabular}{c r r r r r r r r}
\toprule
$h$\,(m) & $\|\us\|_{L^{2}}$ & $\|\vs\|_{L^{2}}$ & $\|\vf\|_{L^{2}}$ & $\|\fv\|_{L^{2}}$ & $\|p\|_{L^{2}}$ & $\|\us\|_{H^{1}}$ & $\|\vs\|_{H^{1}}$ & $\|p\|_{H^{1}}$
\\
\midrule
$\tfrac{1}{2} \to \tfrac{1}{4}$ &
$-1.84$ & $-1.95$ & $-0.567$ & $-0.624$ & $-1.90$ & $-1.34$ & $-1.52$ & $-0.372$ \\
$\tfrac{1}{4} \to \tfrac{1}{8}$ &
$-3.33$ & $-3.62$ & $-1.35$ & $-1.34$ & $-2.63$ & $-2.18$ & $-2.33$ & $-1.84$ \\
$\tfrac{1}{8} \to \tfrac{1}{16}$ &
$-3.64$ & $-3.66$ & $-1.68$ & $-1.68$ & $-2.92$ & $-2.38$ & $-2.41$ & $-2.03$ \\
$\tfrac{1}{16} \to \tfrac{1}{32}$ &
$-3.46$ & $-3.57$ & $-1.84$ & $-1.84$ & $-2.93$ & $-2.21$ & $-2.21$ & $-2.00$ \\
$\tfrac{1}{32} \to \tfrac{1}{64}$ &
$-3.23$ & $-3.37$ & $-1.94$ & $-1.94$ & $-2.96$ & $-2.08$ & $-2.08$ & $-1.99$ \\
\bottomrule
\end{tabular}
\end{minipage}
\end{table}
\begin{table}[ht]
\caption{\label{table: CR xiR 05 t1}%
Convergence rates for the stabilized fluid velocity formulation for $t = \np[s]{1.0}$ with $\xi_{R_{\s}} = 0.5$ and$\mathcal{P}_{p}^{2}$ corresponding to Fig.~\ref{fig: CR xiR 05 p 2 t1}.%
}
\begin{minipage}{\textwidth}
\centering
\renewcommand{\footnoterule}{\vspace{-7pt}\rule{0pt}{0pt}}
\renewcommand{\arraystretch}{1.2}
\begin{tabular}{c r r r r r r r r}
\toprule
$h$\,(m) & $\|\us\|_{L^{2}}$ & $\|\vs\|_{L^{2}}$ & $\|\vf\|_{L^{2}}$ & $\|\fv\|_{L^{2}}$ & $\|p\|_{L^{2}}$ & $\|\us\|_{H^{1}}$ & $\|\vs\|_{H^{1}}$ & $\|p\|_{H^{1}}$
\\
\midrule
$\tfrac{1}{2} \to \tfrac{1}{4}$ &
$-0.835$ & $-1.75$ & $-1.68$ & $-1.60$ & $-1.53$ & $-1.13$ & $-1.31$ & $-2.21$ \\
$\tfrac{1}{4} \to \tfrac{1}{8}$ &
$-3.77$ & $-3.37$ & $-2.26$ & $-2.34$ & $-3.73$ & $-3.64$ & $-2.18$ & $-2.68$ \\
$\tfrac{1}{8} \to \tfrac{1}{16}$ &
$-3.78$ & $-3.65$ & $-2.63$ & $-2.64$ & $-3.74$ & $-3.78$ & $-2.38$ & $-2.32$ \\
$\tfrac{1}{16} \to \tfrac{1}{32}$ &
$-3.95$ & $-3.47$ & $-2.86$ & $-2.87$ & $-3.88$ & $-3.89$ & $-2.20$ & $-2.60$ \\
$\tfrac{1}{32} \to \tfrac{1}{64}$ &
$-3.46$ & $-3.24$ & $-2.95$ & $-2.95$ & $-3.68$ & $-2.83$ & $-2.08$ & $-2.83$ \\
\bottomrule
\end{tabular}
\end{minipage}
\end{table}
our results for the ``worst-case'' scenario ($t = \np[s]{0.7}$) behave in a way consistent with results in  \cite{Masud2002A-Stabilized-Mixed-0} and exceed expectations (with respect to the linear problem) for other time instants.  In Tables~\ref{table: CR xiR 05 t07} and~\ref{table: CR xiR 05 t1}, we have omitted $H^{1}$ error norm results for $\vf$ and $\fv$ since no formal results exist for these fields.  As far as the solid displacement and velocity fields are concerned, we have found results that consistently exceed expectations relative to typical estimates for linear elasto-statics and we have often run into somewhat puzzling super-convergent behavior as that shown in Fig.~\ref{fig: CR xiR 05 p 2 t1}.

Figures~\ref{fig: CR xiR 09 p 2 t07} and~\ref{fig: CR xiR 09 p 2 t1}
\begin{figure}[htb]
    \centering
\includegraphics[scale=0.9]{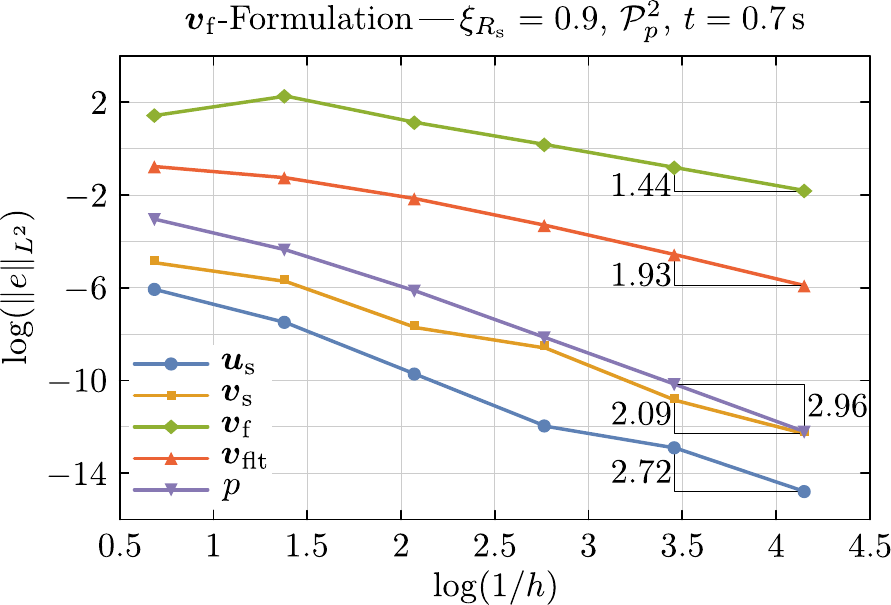}\hfill%
\includegraphics[scale=0.9]{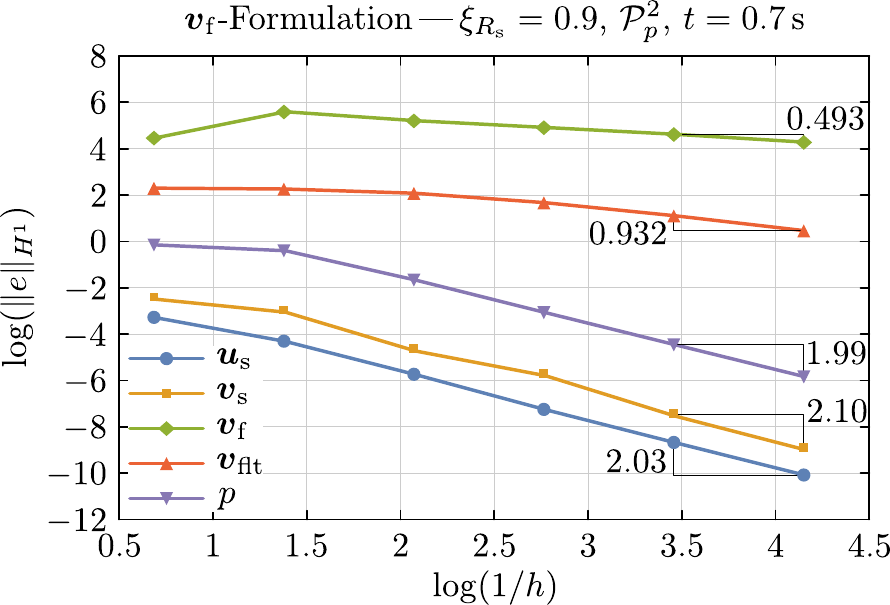}
    \caption{Convergence rates for the $L^{2}$-norm (left) and the $H^{1}$-(semi)norm (right) of the error at $t = \np[s]{0.7}$ obtained via the fluid velocity formulation with $\xi_{R_{\s}} = 0.9$ and $\mathcal{P}_{p}^{2}$.}
    \label{fig: CR xiR 09 p 2 t07}
\end{figure}
\begin{figure}[htb]
    \centering
\hfill%
\includegraphics[scale=0.9]{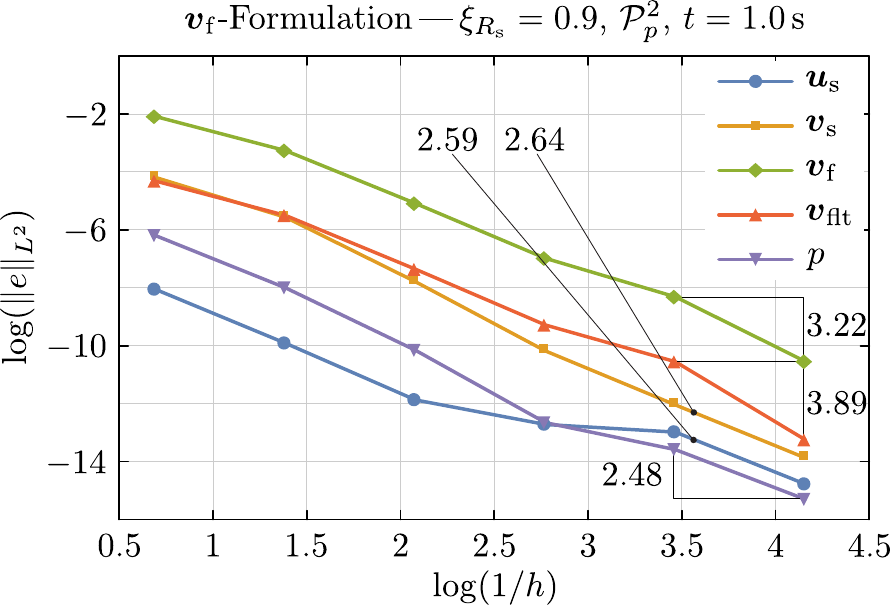}
\includegraphics[scale=0.9]{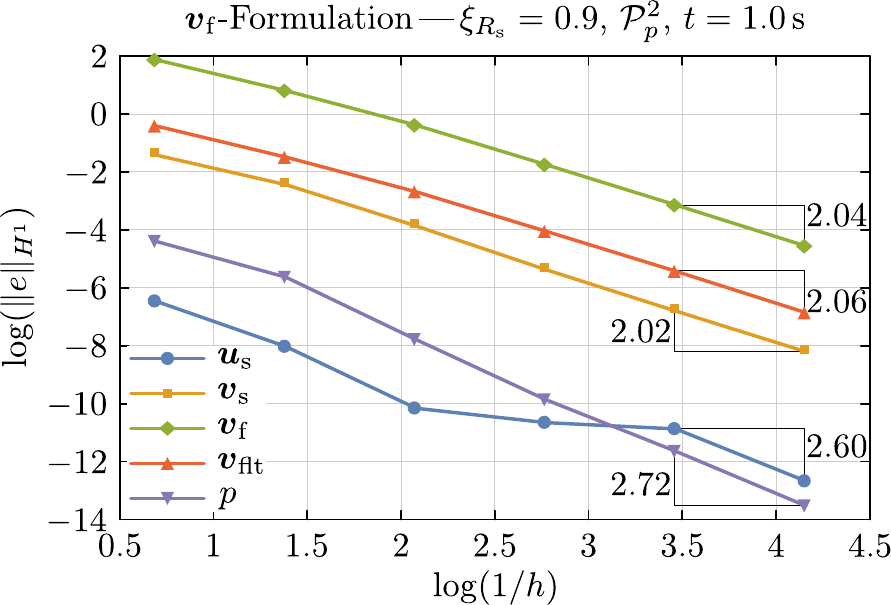}
    \caption{Convergence rates for the $L^{2}$-norm (left) and the $H^{1}$-(semi)norm (right) of the error at $t = \np[s]{1.0}$ obtained via the fluid velocity formulation with $\xi_{R_{\s}} = 0.9$ and $\mathcal{P}_{p}^{2}$.}
    \label{fig: CR xiR 09 p 2 t1}
\end{figure}
as well as Tables~\ref{table: CR xiR 09 t07} and~\ref{table: CR xiR 09 t1}
\begin{table}[ht]
\caption{\label{table: CR xiR 09 t07}%
Convergence rates for the stabilized fluid velocity formulation for $t = \np[s]{0.7}$ with $\xi_{R_{\s}} = 0.9$ and$\mathcal{P}_{p}^{2}$ corresponding to Fig.~\ref{fig: CR xiR 09 p 2 t07}.%
}
\begin{minipage}{\textwidth}
\centering
\renewcommand{\footnoterule}{\vspace{-7pt}\rule{0pt}{0pt}}
\renewcommand{\arraystretch}{1.2}
\begin{tabular}{c r r r r r r r r}
\toprule
$h$\,(m) & $\|\us\|_{L^{2}}$ & $\|\vs\|_{L^{2}}$ & $\|\vf\|_{L^{2}}$ & $\|\fv\|_{L^{2}}$ & $\|p\|_{L^{2}}$ & $\|\us\|_{H^{1}}$ & $\|\vs\|_{H^{1}}$ & $\|p\|_{H^{1}}$
\\
\midrule
$\tfrac{1}{2} \to \tfrac{1}{4}$ &
$-2.05$ & $-1.16$ & $1.20$ & $-0.693$ & $-1.89$ & $-1.48$ & $-0.808$ & $-0.363$ \\
$\tfrac{1}{4} \to \tfrac{1}{8}$ &
$-3.22$ & $-2.87$ & $-1.64$ & $-1.31$ & $-2.56$ & $-2.06$ & $-2.42$ & $-1.82$ \\
$\tfrac{1}{8} \to \tfrac{1}{16}$ &
$-3.23$ & $-1.28$ & $-1.38$ & $-1.66$ & $-2.91$ & $-2.18$ & $-1.54$ & $-2.03$ \\
$\tfrac{1}{16} \to \tfrac{1}{32}$ &
$-1.36$ & $-3.26$ & $-1.43$ & $-1.83$ & $-2.91$ & $-2.06$ & $-2.52$ & $-2.00$ \\
$\tfrac{1}{32} \to \tfrac{1}{64}$ &
$-2.72$ & $-2.09$ & $-1.44$ & $-1.93$ & $-2.96$ & $-2.03$ & $-2.10$ & $-1.99$ \\
\bottomrule
\end{tabular}
\end{minipage}
\end{table}
\begin{table}[ht]
\caption{\label{table: CR xiR 09 t1}%
Convergence rates for the stabilized fluid velocity formulation for $t = \np[s]{1.0}$ with $\xi_{R_{\s}} = 0.9$ and$\mathcal{P}_{p}^{2}$ corresponding to Fig.~\ref{fig: CR xiR 09 p 2 t1}.%
}
\begin{minipage}{\textwidth}
\centering
\renewcommand{\footnoterule}{\vspace{-7pt}\rule{0pt}{0pt}}
\renewcommand{\arraystretch}{1.2}
\begin{tabular}{c r r r r r r r r}
\toprule
$h$\,(m) & $\|\us\|_{L^{2}}$ & $\|\vs\|_{L^{2}}$ & $\|\vf\|_{L^{2}}$ & $\|\fv\|_{L^{2}}$ & $\|p\|_{L^{2}}$ & $\|\us\|_{H^{1}}$ & $\|\vs\|_{H^{1}}$ & $\|p\|_{H^{1}}$
\\
\midrule
$\tfrac{1}{2} \to \tfrac{1}{4}$ &
$-2.68$ & $-1.98$ & $-1.70$ & $-1.71$ & $-2.60$ & $-2.26$ & $-1.48$ & $-1.79$ \\
$\tfrac{1}{4} \to \tfrac{1}{8}$ &
$-2.82$ & $-3.21$ & $-2.63$ & $-2.68$ & $-3.11$ & $-3.07$ & $-2.05$ & $-3.10$ \\
$\tfrac{1}{8} \to \tfrac{1}{16}$ &
$-1.22$ & $-3.45$ & $-2.73$ & $-2.78$ & $-3.60$ & $-0.715$ & $-2.18$ & $-3.00$ \\
$\tfrac{1}{16} \to \tfrac{1}{32}$ &
$-0.381$ & $-2.69$ & $-1.92$ & $-1.83$ & $-1.34$ & $-0.314$ & $-2.06$ & $-2.57$ \\
$\tfrac{1}{32} \to \tfrac{1}{64}$ &
$-2.59$ & $-2.64$ & $-3.22$ & $-3.89$ & $-2.48$ & $-2.60$ & $-2.02$ & $-2.72$ \\
\bottomrule
\end{tabular}
\end{minipage}
\end{table}\afterpage{\clearpage}
show the convergence rates for the case with $\xi_{R_{\s}} = 0.9$, that is, for high concentration of the solid phase.  As can be seen, several of the observations made for the case with $\xi_{R_{\s}} = 0.5$ still apply.  However, it is apparent that the convergence rates are more erratic and have degraded. The degradation of the convergence rates was expected in that the volume fraction $\xi_{R_{\s}}$ appears as a coefficient affecting the coercivity of some operators and it affects limit behaviors as the values $0$ and $1$ are approached.  To gather some information about the formulation's behavior as a function of $\xi_{R_{\s}}$, we have conducted a parametric sweep with $0.25 \leq \xi_{R_{\s}} \leq 0.95$.  This calculation was carried out with $h = (1/32)\,\mathrm{m}$ and the results are shown in Figs.~\ref{fig: Four Field XiBehavior p2 t07} and~\ref{fig: Four Field XiBehavior p2 t1}.
\begin{figure}[htb]
    \centering
\includegraphics[scale=0.9]{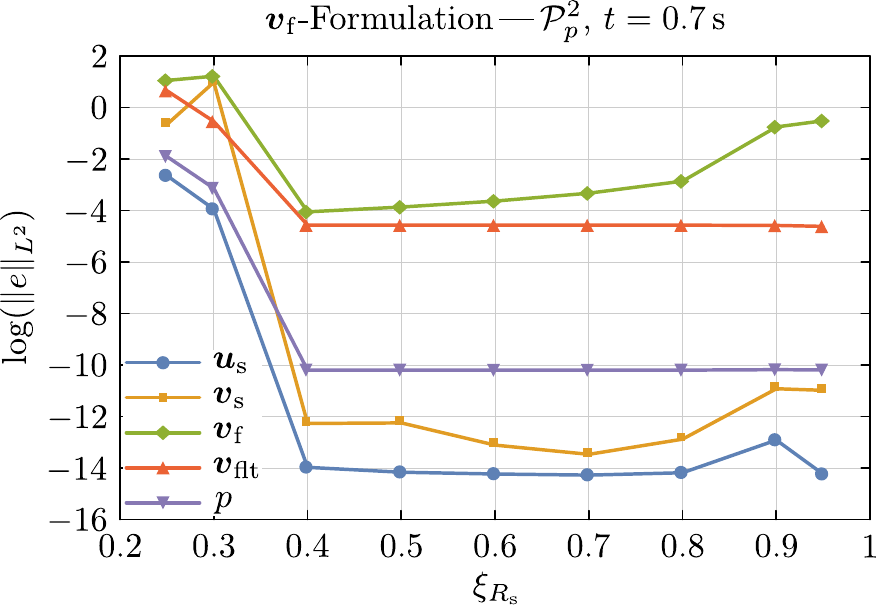}\hfill%
\includegraphics[scale=0.9]{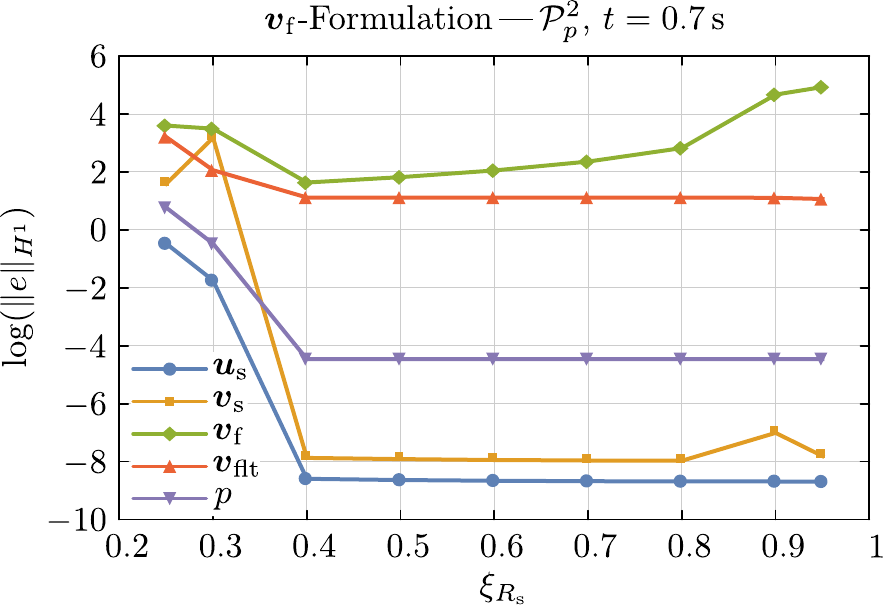}
    \caption{$L^{2}$-norm (left) and the $H^{1}$-(semi)norm (right) of the error at $t = \np[s]{0.7}$ for the fluid velocity formulation as a function of $\xi_{R_{\s}}$.  The element diameter for this simulation was $h = (1/32)\,\mathrm{m}$.}
    \label{fig: Four Field XiBehavior p2 t07}
\end{figure}
\begin{figure}[htb]
    \centering
\hfill%
\includegraphics[scale=0.9]{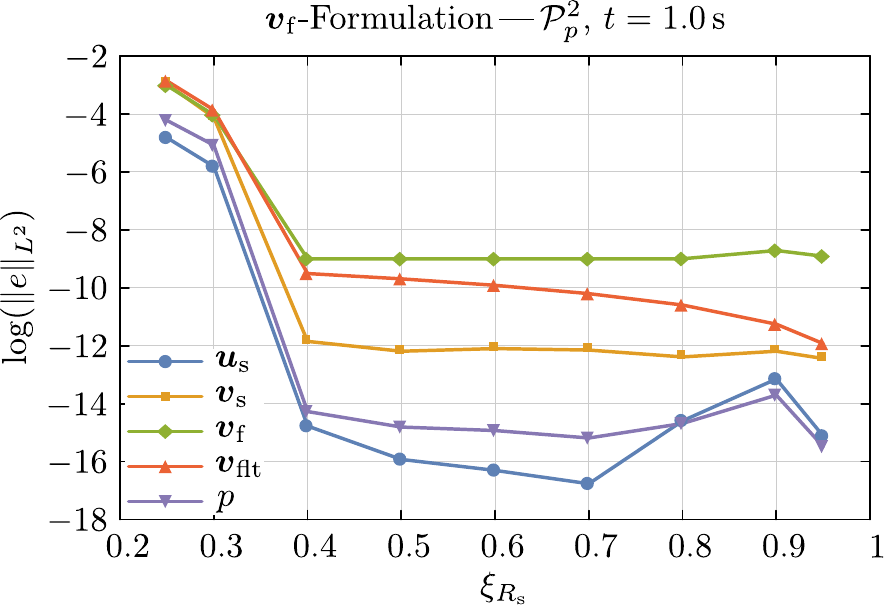}
\includegraphics[scale=0.9]{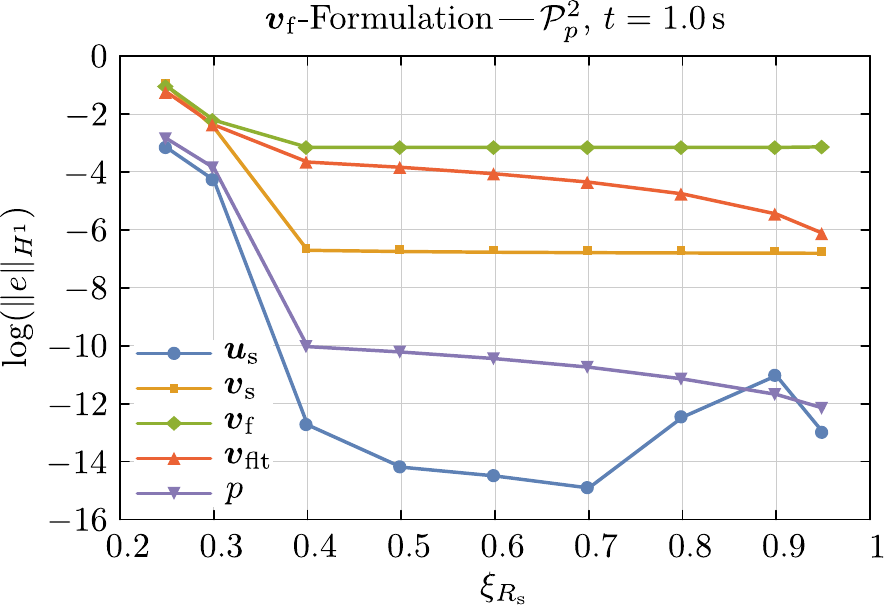}
    \caption{$L^{2}$-norm (left) and the $H^{1}$-(semi)norm (right) of the error at $t = \np[s]{1.0}$ for the fluid velocity formulation as a function of $\xi_{R_{\s}}$.  The element diameter for this simulation was $h = (1/32)\,\mathrm{m}$.}
    \label{fig: Four Field XiBehavior p2 t1}
\end{figure}
These results indicate that the accuracy is highly degraded for all fields for high volume fractions of the fluid.  From a numerical viewpoint, a different formulation would be needed for these cases.  More importantly, from a physical viewpoint, cases with high porosity, i.e., high fluid volume fractions, should be modeled as Brinkman flow problems (cf.\ \citealp{Masud2007A-Stabilized-Mixed-0}). For increasing values of $\xi_{R_{\s}}$, Fig.~\ref{fig: Four Field XiBehavior p2 t07} and~\ref{fig: Four Field XiBehavior p2 t1} show degraded accuracy for the field $\vf$, while the accuracy in terms of the other fields, especially $\fv$ and $p$, appears to be relatively unaffected by increasing values of the solid volume fraction.

For the stabilized quasi-static formulation there are no restrictions induced by the Brezzi-Babu\v{s}ka condition.  To illustrate this point, in Figs.~\ref{fig: Four Field CR p3 t07} and~\ref{fig: Four Field CR p3 t1}
\begin{figure}[htb]
    \centering
\includegraphics[scale=0.9]{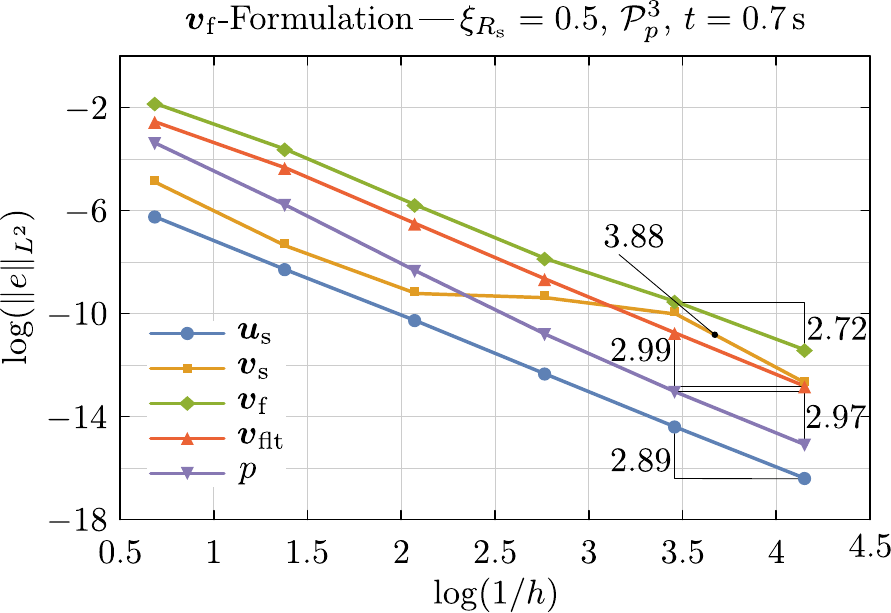}\hfill%
\includegraphics[scale=0.9]{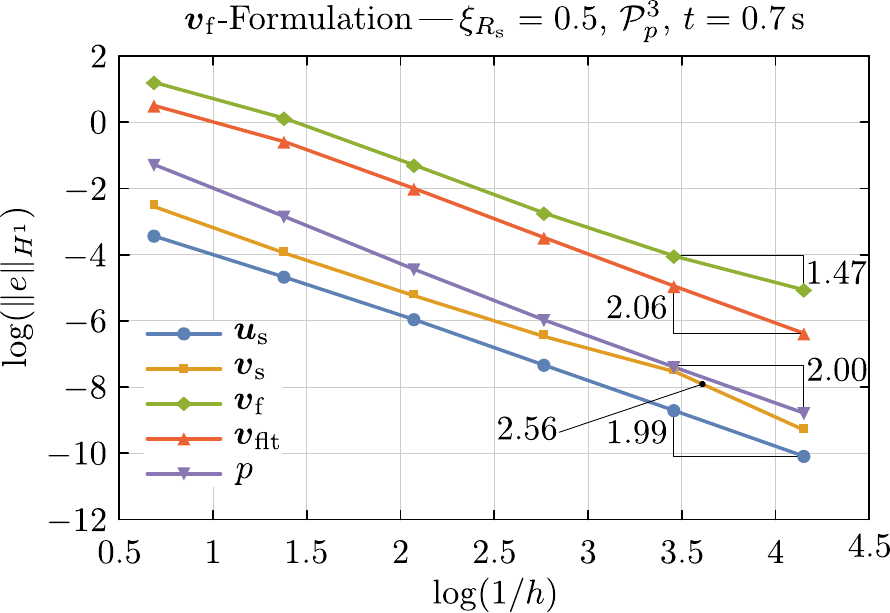}
    \caption{Convergence rates for the $L^{2}$-norm (left) and the $H^{1}$-(semi)norm (right) of the error at $t = \np[s]{0.7}$ obtained via the fluid velocity formulation with $\xi_{R_{\s}} = 0.5$ and $\mathcal{P}_{p}^{3}$.}
    \label{fig: Four Field CR p3 t07}
\end{figure}
\begin{figure}[htb]
    \centering
\hfill%
\includegraphics[scale=0.9]{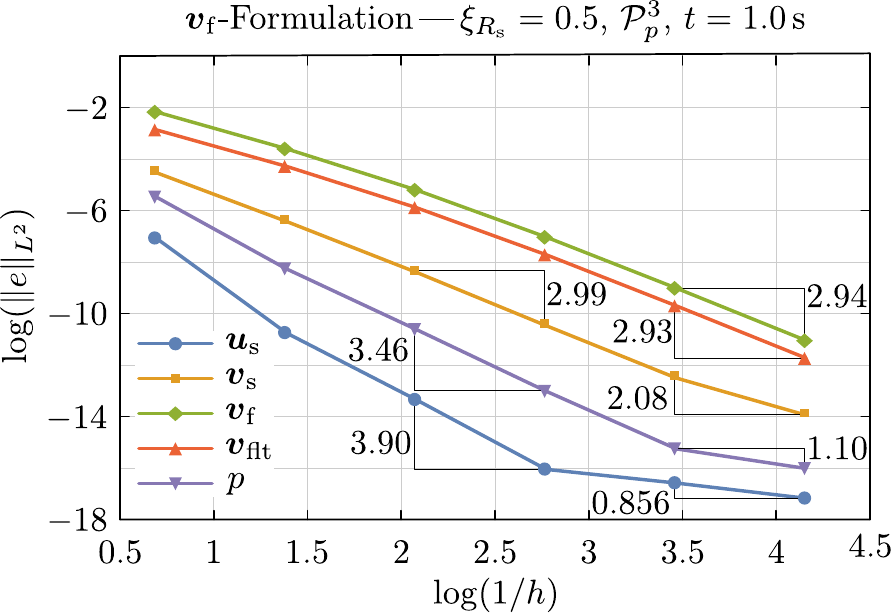}
\includegraphics[scale=0.9]{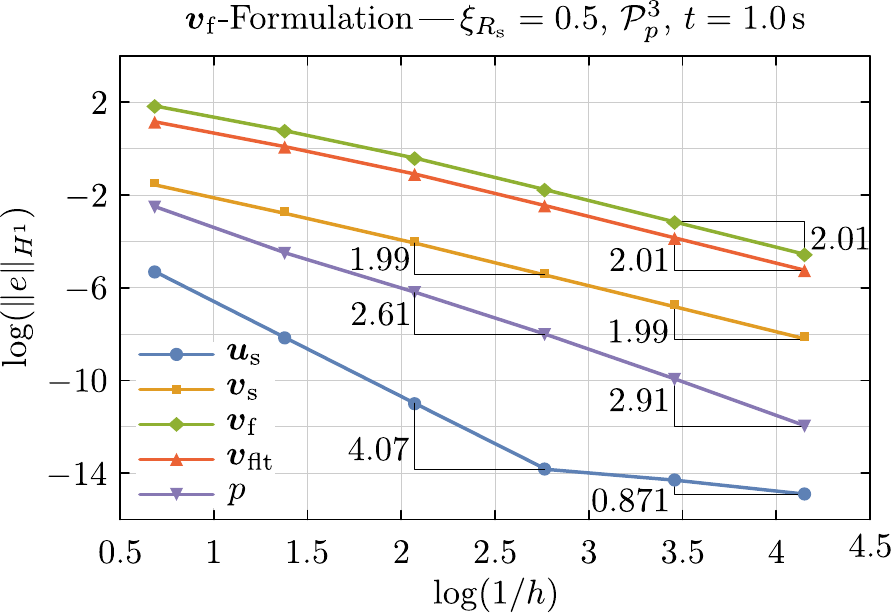}
    \caption{Convergence rates for the $L^{2}$-norm (left) and the $H^{1}$-(semi)norm (right) of the error at $t = \np[s]{1.0}$ obtained via the fluid velocity formulation with $\xi_{R_{\s}} = 0.5$ and $\mathcal{P}_{p}^{3}$.}
    \label{fig: Four Field CR p3 t1}
\end{figure}
we present results with $\xi_{R_{\s}} = 0.5$ in which all fields are interpolated via second order Lagrange polynomials except $p$, which is interpolated using cubic Lagrange polynomials.  In addition to illustrating the possibility of choosing arbitrary solution spaces, the result in question reveals an interesting effect. While the observed convergence rates are not easily explained and require a formal analysis, it appears that by increasing the order of interpolation of $p$ has a beneficial effect on the order of convergence of of the fields $\vf$ and $\fv$ without negatively affecting the behavior of the fields $\us$ and $\vs$ (in fact, somewhat positive). If confirmed, this result suggests that one might gain almost a full order of convergence in the vector fields $\vf$ and $\fv$ at the relatively moderate cost of increasing the degrees of freedom of a scalar field.

\subsection{Filtration velocity quasi-static formulation (Problem~\ref{dual pb abstract weak ALE fv quasi static stabilized}) results}
The convergence results for the quasi-static problem solved using the stabilized filtration velocity formulation in Problem~\ref{dual pb abstract weak ALE fv quasi static stabilized} for $\xi_{R_{\s}} = 0.5$ are presented in Figs.~\ref{fig: Five Field CR p2 t07} and~\ref{fig: Five Field CR p2 t1}.
\begin{figure}[htb]
    \centering
\includegraphics[scale=0.9]{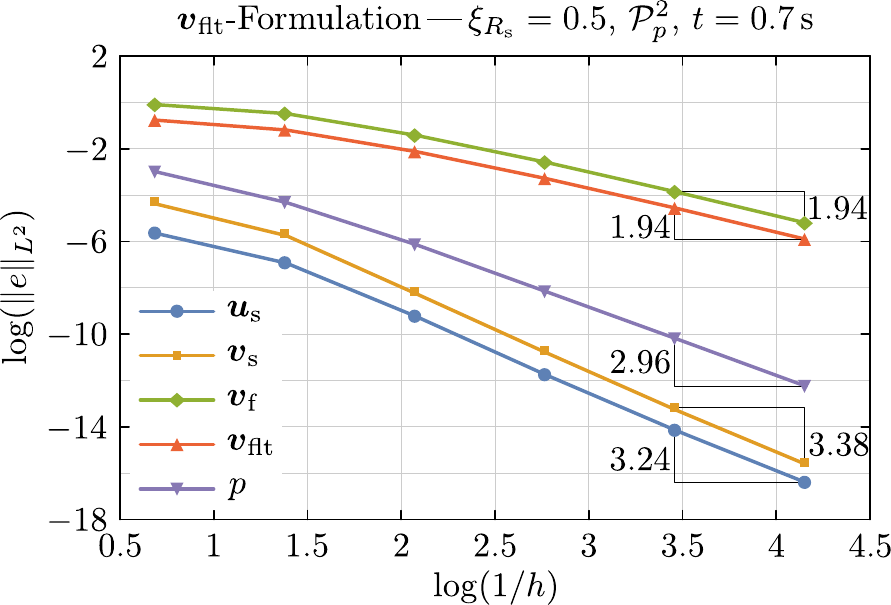}\hfill%
\includegraphics[scale=0.9]{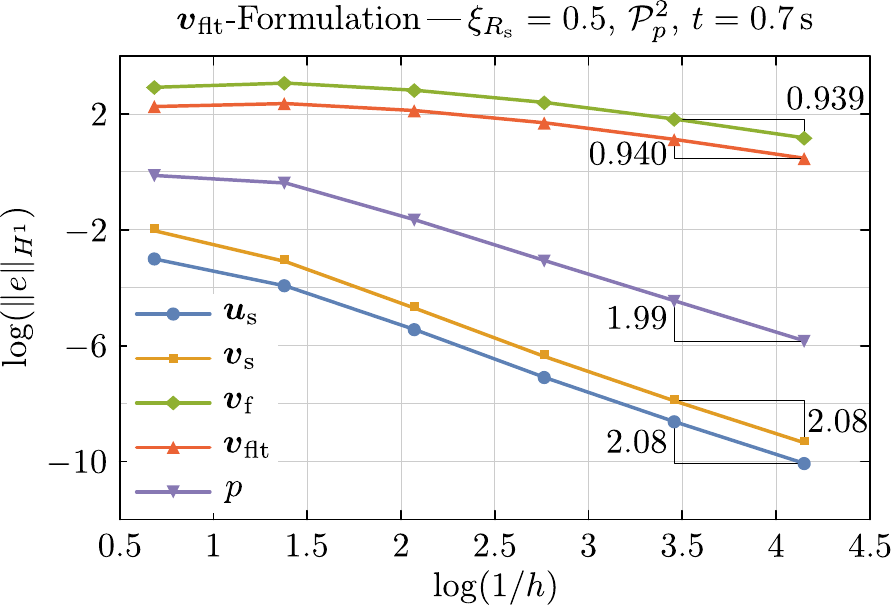}
    \caption{Convergence rates for the $L^{2}$-norm (left) and the $H^{1}$-(semi)norm (right) of the error at $t = \np[s]{0.7}$ obtained via the filtration velocity formulation with $\xi_{R_{\s}} = 0.5$ and $\mathcal{P}_{p}^{2}$.}
    \label{fig: Five Field CR p2 t07}
\end{figure}
\begin{figure}[htb]
    \centering
\hfill%
\includegraphics[scale=0.9]{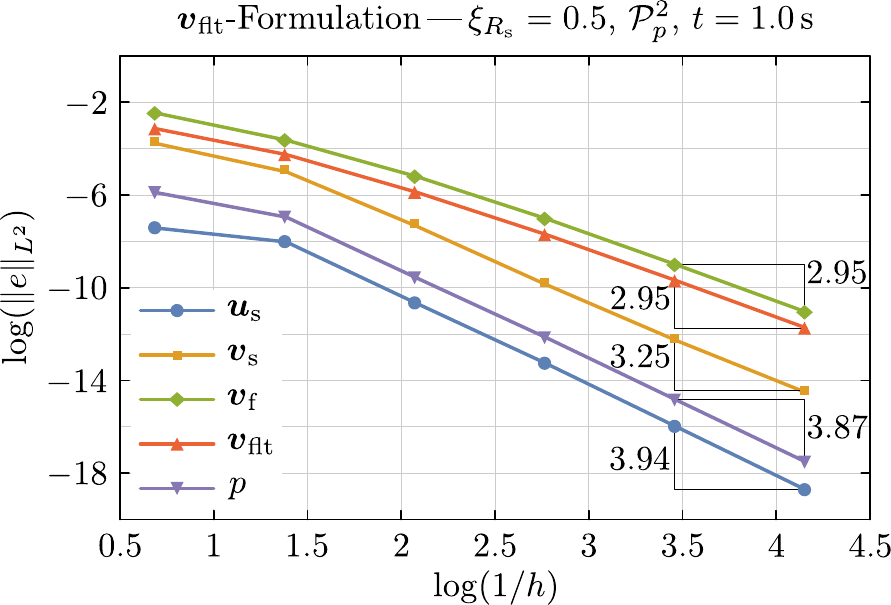}
\includegraphics[scale=0.9]{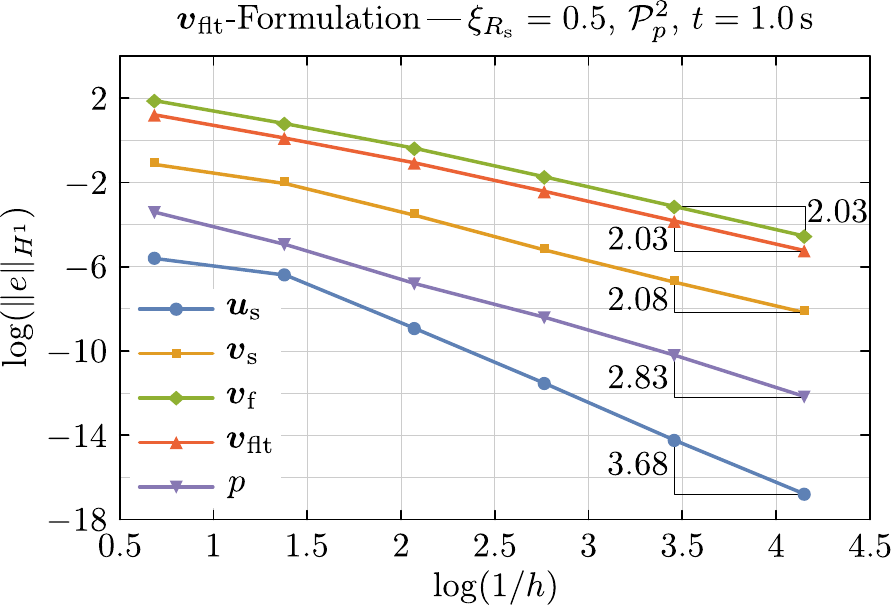}
    \caption{Convergence rates for the $L^{2}$-norm (left) and the $H^{1}$-(semi)norm (right) of the error at $t = \np[s]{1.0}$ obtained via the filtration velocity formulation with $\xi_{R_{\s}} = 0.5$ and $\mathcal{P}_{p}^{2}$.}
    \label{fig: Five Field CR p2 t1}
\end{figure}
The detailed convergence rates corresponding to Figs.~\ref{fig: Five Field CR p2 t07} and~\ref{fig: Five Field CR p2 t1} are in Tables~\ref{table: five field CR xiR 05 t07} and~\ref{table: five field CR xiR 05 t1}.
\begin{table}[ht]
\caption{\label{table: five field CR xiR 05 t07}%
Convergence rates for the stabilized filtration velocity formulation for $t = \np[s]{0.7}$ with $\xi_{R_{\s}} = 0.5$ and$\mathcal{P}_{p}^{2}$ corresponding to Fig.~\ref{fig: Five Field CR p2 t07}.%
}
\begin{minipage}{\textwidth}
\centering
\renewcommand{\footnoterule}{\vspace{-7pt}\rule{0pt}{0pt}}
\renewcommand{\arraystretch}{1.2}
\begin{tabular}{c r r r r r r r r}
\toprule
$h$\,(m) & $\|\us\|_{L^{2}}$ & $\|\vs\|_{L^{2}}$ & $\|\vf\|_{L^{2}}$ & $\|\fv\|_{L^{2}}$ & $\|p\|_{L^{2}}$ & $\|\us\|_{H^{1}}$ & $\|\vs\|_{H^{1}}$ & $\|p\|_{H^{1}}$
\\
\midrule
$\tfrac{1}{2} \to \tfrac{1}{4}$ &
$-1.85$ & $-1.96$ & $-0.559$ & $-0.613$ & $-1.90$ & $-1.33$ & $-1.52$ & $-0.375$ \\
$\tfrac{1}{4} \to \tfrac{1}{8}$ &
$-3.33$ & $-3.62$ & $-1.35$ & $-1.35$ & $-2.63$ & $-2.19$ & $-2.34$ & $-1.84$ \\
$\tfrac{1}{8} \to \tfrac{1}{16}$ &
$-3.64$ & $-3.67$ & $-1.68$ & $-1.68$ & $-2.92$ & $-2.38$ & $-2.41$ & $-2.03$ \\
$\tfrac{1}{16} \to \tfrac{1}{32}$ &
$-3.46$ & $-3.57$ & $-1.84$ & $-1.84$ & $-2.93$ & $-2.21$ & $-2.21$ & $-2.00$ \\
$\tfrac{1}{32} \to \tfrac{1}{64}$ &
$-3.24$ & $-3.38$ & $-1.94$ & $-1.94$ & $-2.96$ & $-2.08$ & $-2.08$ & $-1.99$ \\
\bottomrule
\end{tabular}
\end{minipage}
\end{table}
\begin{table}[ht]
\caption{\label{table: five field CR xiR 05 t1}%
Convergence rates for the stabilized filtration velocity formulation for $t = \np[s]{1.0}$ with $\xi_{R_{\s}} = 0.5$ and$\mathcal{P}_{p}^{2}$ corresponding to Fig.~\ref{fig: Five Field CR p2 t1}.%
}
\begin{minipage}{\textwidth}
\centering
\renewcommand{\footnoterule}{\vspace{-7pt}\rule{0pt}{0pt}}
\renewcommand{\arraystretch}{1.2}
\begin{tabular}{c r r r r r r r r}
\toprule
$h$\,(m) & $\|\us\|_{L^{2}}$ & $\|\vs\|_{L^{2}}$ & $\|\vf\|_{L^{2}}$ & $\|\fv\|_{L^{2}}$ & $\|p\|_{L^{2}}$ & $\|\us\|_{H^{1}}$ & $\|\vs\|_{H^{1}}$ & $\|p\|_{H^{1}}$
\\
\midrule
$\tfrac{1}{2} \to \tfrac{1}{4}$ &
$-0.865$ & $-1.76$ & $-1.68$ & $-1.60$ & $-1.53$ & $-1.14$ & $-1.31$ & $-2.20$ \\
$\tfrac{1}{4} \to \tfrac{1}{8}$ &
$-3.81$ & $-3.38$ & $-2.26$ & $-2.34$ & $-3.77$ & $-3.68$ & $-2.18$ & $-2.69$ \\
$\tfrac{1}{8} \to \tfrac{1}{16}$ &
$-3.76$ & $-3.65$ & $-2.63$ & $-2.64$ & $-3.73$ & $-3.76$ & $-2.38$ & $-2.31$ \\
$\tfrac{1}{16} \to \tfrac{1}{32}$ &
$-3.93$ & $-3.47$ & $-2.86$ & $-2.87$ & $-3.88$ & $-3.90$ & $-2.20$ & $-2.60$ \\
$\tfrac{1}{32} \to \tfrac{1}{64}$ &
$-3.94$ & $-3.25$ & $-2.95$ & $-2.95$ & $-3.87$ & $-3.68$ & $-2.08$ & $-2.83$ \\
\bottomrule
\end{tabular}
\end{minipage}
\end{table}
These results are virtually identical to those already presented in Figs~\ref{fig: CR xiR 05 p 2 t07} and~\ref{fig: CR xiR 05 p 2 t1} indicating that for $\xi_{R_{\s}} = 0.5$ there is no appreciable difference between the two formulations.  Hence, we can say that the convergence rates approach optimal values in the sense discussed earlier.

The convergence results for the quasi-static problem solved using the stabilized filtration velocity formulation in Problem~\ref{dual pb abstract weak ALE fv quasi static stabilized} for $\xi_{R_{\s}} = 0.9$ are presented in Figs.~\ref{fig: Five Field CR p2 xi09 t07} and~\ref{fig: Five Field CR p2 xi09 t1}.
\begin{figure}[htb]
    \centering
\includegraphics[scale=0.9]{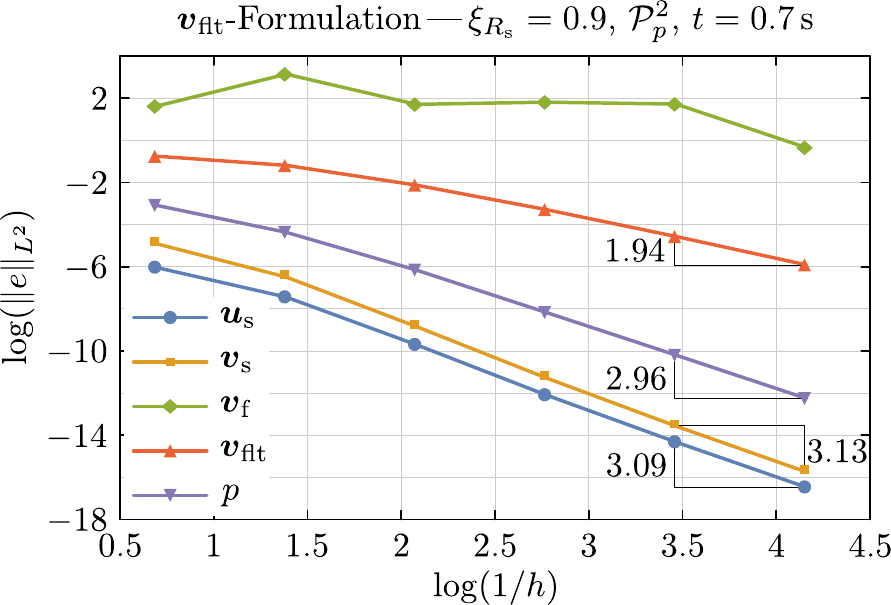}\hfill%
\includegraphics[scale=0.9]{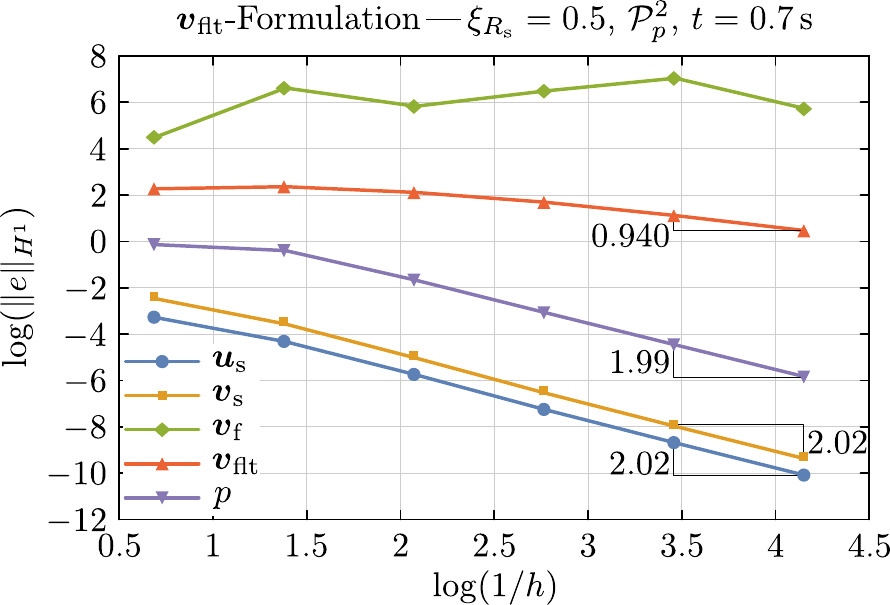}
    \caption{Convergence rates for the $L^{2}$-norm (left) and the $H^{1}$-(semi)norm (right) of the error at $t = \np[s]{0.7}$ obtained via the filtration velocity formulation with $\xi_{R_{\s}} = 0.9$ and $\mathcal{P}_{p}^{2}$. The rates for $\vf$ have been omitted.}
    \label{fig: Five Field CR p2 xi09 t07}
\end{figure}
\begin{figure}[htb]
    \centering
\hfill%
\includegraphics[scale=0.9]{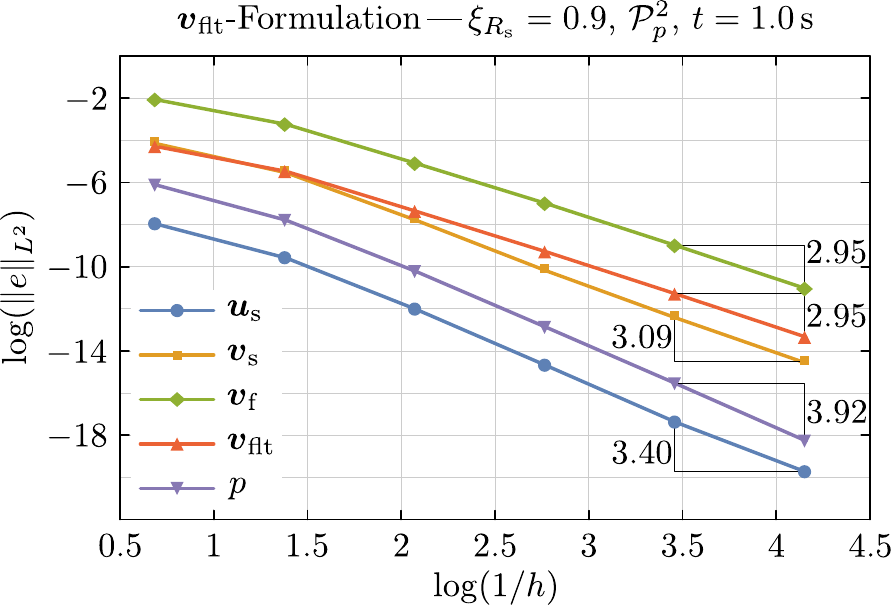}
\includegraphics[scale=0.9]{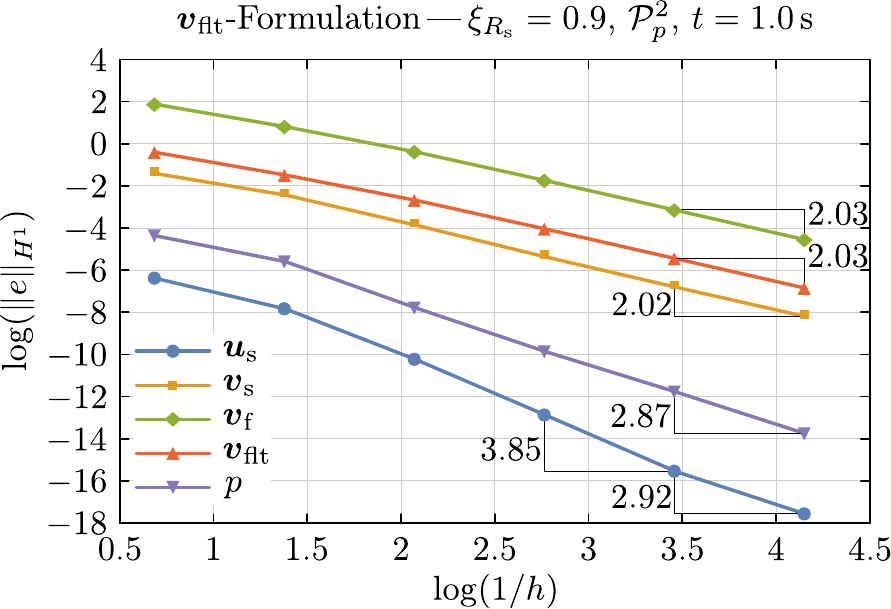}
    \caption{Convergence rates for the $L^{2}$-norm (left) and the $H^{1}$-(semi)norm (right) of the error at $t = \np[s]{1.0}$ obtained via the filtration velocity formulation with $\xi_{R_{\s}} = 0.9$ and $\mathcal{P}_{p}^{2}$. The rates for $\vf$ have been omitted.}
    \label{fig: Five Field CR p2 xi09 t1}
\end{figure}
The results in Figs.~\ref{fig: Five Field CR p2 xi09 t07} and~\ref{fig: Five Field CR p2 xi09 t1} present a slight improvement in the behavior of the fields $\us$, $\vf$, $\fv$, and $p$ relative to that of the fluid velocity formulation (Figs.~\ref{fig: CR xiR 09 p 2 t07} and~\ref{fig: CR xiR 09 p 2 t1}).  However, for the plots with $t = \np[s]{0.7}$, the behavior in the field $\vf$ is degraded to the point that non-convergent behavior can be clearly observed.  This result suggests that the filtration velocity formulation does not have clear advantages over the fluid velocity formulation in terms of accuracy and reliability, at least for the parameter set and range of $\xi_{R_{\s}}$ used in this paper.

\section{Conclusions}
In this paper we considered the quasi-static motion of a poroelastic system consisting of an hyperelastic incompressible solid skeleton saturated by an incompressible viscous fluid.  The model considered stems from mixture theory and it is intended for modular expansion to include progressively more complex physics such as the chemistry of biodegradation.  For this system, we considered a consistent stabilized \ac{FEM} predicated on the assumption that the porosity data provided as input to the model is not necessarily smooth.  For said formulation we proved stability and obtained convergence rates using the method of manufactured solutions.  We have shown that the rates in question are optimal relative to the estimates by \cite{Masud2002A-Stabilized-Mixed-0} and that in some cases, supercovergence is observed.  Future work will focus on the formulation of a stable \ac{FEM} that includes inertia effects.

\section*{Acknowledgements}
\label{sec:acknowledgements}
F.~Costanzo gratefully acknowledges support from the US National Science Foundation through Award N.~1537008 (CMMI).

Sandia National Laboratories is a multi-program laboratory managed and operated by Sandia Corporation, a wholly owned subsidiary of Lockheed Martin Corporation, for the U.S. Department of Energy's National Nuclear Security Administration under contract DE-AC04-94AL85000. 

\section*{Appendix}
\subsection*{Operators in the dual formulation}
This Appendix presents the definitions of the operators used in defining the dual formulation introduced earlier in the paper.  Following \citet{Heltai2012Variational-Imp0}, we introduce the following notation
\begin{equation}
\label{Aeq: notation}
\prescript{}{V^{*}}{\bigl\langle} \psi, \phi \big\rangle_{V},
\end{equation}
in which, given a vector space $V$ and its dual $V^{*}$, $\psi$ and $\phi$ are elements of the vector spaces $V^{*}$ and $V$, respectively, and where $\prescript{}{V^{*}}{\bigl\langle} \bullet, \bullet \big\rangle_{V}$ identifies the duality product between $V^{*}$ and  $V$.  , The operators defined in this Appendix, although not all linear, can be interpreted as though they were matrices.  For this reason we adopt the following notation to identify the domain and ranges of the operators in question and whether or not they are linear:
\begin{itemize}
\item
We will refer to the spaces $\mathcal{V}^{\bv{u}_{\s}}$, $\mathcal{V}^{\bv{v}_{\s}}$, $\mathcal{V}^{\bv{v}_{\f}}$, and $\mathcal{V}^{p}$ are identified by the indices $\sap$, $\sbp$, $\scp$, and $\sdp$, respectively.

\item
We will refer to the spaces $(\mathcal{V}^{\bv{u}_{\s}})^{*}$, $(\mathcal{V}^{\bv{v}_{\s}})^{*}$, $(\mathcal{V}^{\bv{v}_{\f}})^{*}$, and $(\mathcal{V}^{p})^{*}$ are identified by the indices $\sad$, $\sbd$, $\scd$, and $\sdd$, respectively.

\item
If an operator is nonlinear, the symbol denoting the operator will be followed by a list in parenthesis containing the fields on which the operator depends.  If the list is absent, then the operator is understood to be linear.

\item
If an operator is characterized by one one subscript, the latter denotes the range of the operator.
\end{itemize}
For example, the notation $\mathcal{M}_{\sad\sap}$ denotes a linear operator mapping an element $\mathcal{V}^{\bv{u}_{\s}}$ into an element of $(\mathcal{V}^{\bv{u}_{\s}})^{*}$.  By contrast, an operator $\mathcal{A}_{\sbd\scp}(\bv{u}_\s)$ denotes a nonlinear operator that depends on $\bv{u}_{\s}$ and that maps an element of $\mathcal{V}^{\bv{v}_{\f}}$ into an element of $(\mathcal{V}^{\bv{v}_{\s}})^{*}$. As a final example, an operator denoted by $\mathcal{F}_{\sbp}(\bv{u}_{\s})$ is an operator that depends on $\bv{u}_{\s}$ and takes values in $\mathcal{V}^{\bv{v}_{\s}}$.

\subsection*{Operators used to define Problem~\ref{dual pb abstract weak ALE}}
With the above in mind, we define the following operators:
\begin{alignat}{2}
\label{aeq: M11 def}
\mathcal{M}_{\sad\sap} &: \mathcal{H}^{\bv{u}_{\s}} \to (\mathcal{V}^{\bv{u}_{\s}})^{*},
&\quad
\prescript{}{(\mathcal{V}^{\bv{u}_{\s}})^{*}}{\bigl\langle}
\mathcal{M}_{\sad\sap}\bv{v},\bv{w}
\big\rangle_{\mathcal{V}^{\bv{u}_{\s}}} 
&\coloneqq \int_{B_{\s}} \bv{w} \cdot \bv{v}, 
\nonumber
\\
& & &\qquad\qquad
\forall \bv{v} \in \mathcal{H}^{\bv{u}_{\s}}, \forall \bv{w} \in \mathcal{V}^{\bv{u}_{\s}}_{0},
\\
\label{eq: M12 def}
\mathcal{M}_{\sad\sbp}
&:
\mathcal{V}^{\bv{v}_{\s}} \to (\mathcal{V}^{\bv{u}_{\s}})^{*},
&\quad
\prescript{}{(\mathcal{V}^{\bv{u}_{\s}})^{*}}{\bigl\langle}
\mathcal{M}_{\sad\sbp} \bv{v} , \bv{w}
\big\rangle_{\mathcal{V}^{\bv{u}_{\s}}}
&\coloneqq
\int_{B_{\s}} \bv{w} \cdot \bv{v}, 
\nonumber
\\
& & & \qquad\qquad
\forall \bv{v} \in \mathcal{V}^{\bv{v}_{\s}}, \forall \bv{w} \in \mathcal{V}^{\bv{u}_{\s}}_{0},
\\
\label{eq: B42 def}
\mathcal{B}_{\sdd\sbp}(\bv{u}_{\s}; \xi_{R_{\s}}) &:  \mathcal{V}^{\bv{v}_{\s}} \to (\mathcal{V}^{p})^{*},
&\quad
\prescript{}{(\mathcal{V}^{p})^{*}}{\bigl\langle}
\mathcal{B}_{\sdd\sbp}(\bv{u}_{\s}; \xi_{R_{\s}}) \bv{v}, q
\big\rangle_{\mathcal{V}^{p}} 
&\coloneqq
-\int_{B_{\s}} \xi_{R_{\s}} \Grad q \cdot \tensor{F}_{\s}^{-1} \bv{v}, 
\nonumber
\\
& & & \qquad\qquad
\forall \bv{v} \in \mathcal{V}^{\bv{v}_{\s}} , \forall q \in \mathcal{V}^{p},
\\
\label{eq: B42T def}
\transpose{\mathcal{B}}_{\sdd\sbp}(\bv{u}_{\s}; \xi_{R_{\s}}) &:  \mathcal{V}^{p} \to (\mathcal{V}^{\bv{v}_{\s}})^{*},
&\quad
\prescript{}{(\mathcal{V}^{\bv{v}_{\s}})^{*}}{\bigl\langle}
\transpose{\mathcal{B}}_{\sdd\sbp}(\bv{u}_{\s}; \xi_{R_{\s}}) q, \bv{v}
\big\rangle_{\mathcal{V}^{\bv{v}_{\s}}} 
&\coloneqq
-\int_{B_{\s}} \xi_{R_{\s}} \Grad q \cdot \tensor{F}_{\s}^{-1} \bv{v}, 
\nonumber
\\
& & & \qquad\qquad
\forall q \in \mathcal{V}^{p}, \forall \bv{v} \in \mathcal{V}^{\bv{v}_{\s}}_{0},
\\
\label{eq: B43 def}
\mathcal{B}_{\sdd\scp}(\bv{u}_{\s}; \xi_{R_{\s}}) &:  \mathcal{V}^{\bv{v}_{\f}} \to (\mathcal{V}^{p})^{*},
&\quad
\prescript{}{(\mathcal{V}^{p})^{*}}{\bigl\langle}
\mathcal{B}_{\sdd\scp}(\bv{u}_{\s}; \xi_{R_{\s}}) \bv{v}, q
\big\rangle_{\mathcal{V}^{p}} 
&\coloneqq
-\int_{B_{\s}} (J_{\s} - \xi_{R_{\s}}) \Grad q \cdot \tensor{F}_{\s}^{-1} \bv{v}, 
\nonumber
\\
& & & \qquad\qquad
\forall \bv{v} \in \mathcal{V}^{\bv{v}_{\f}} , \forall q \in \mathcal{V}^{p},
\\
\label{eq: B43T def}
\transpose{\mathcal{B}}_{\sdd\scp}(\bv{u}_{\s}; \xi_{R_{\s}}) &:  \mathcal{V}^{p} \to (\mathcal{V}^{\bv{v}_{\f}})^{*},
&\quad
\prescript{}{(\mathcal{V}^{\bv{v}_{\f}})^{*}}{\bigl\langle}
\transpose{\mathcal{B}}_{\sdd\scp}(\bv{u}_{\s}; \xi_{R_{\s}}) q, \bv{v}
\big\rangle_{\mathcal{V}^{\bv{v}_{\f}}} 
&\coloneqq
-\int_{B_{\s}} (J_{\s} - \xi_{R_{\s}}) \Grad q \cdot \tensor{F}_{\s}^{-1} \bv{v}, 
\nonumber
\\
& & & \qquad\qquad
\forall q \in \mathcal{V}^{p}, \forall \bv{v} \in \mathcal{V}^{\bv{v}_{\f}},
\\
\label{eq: M33 def}
\mathcal{M}_{\scd\scp}(\bv{u}_{\s}; \xi_{R_{\s}}) &:  \mathcal{H}^{\bv{v}_{\f}} \to (\mathcal{V}^{\bv{v}_{\f}})^{*},
&\quad
\prescript{}{(\mathcal{V}^{\bv{v}_{\f}})^{*}}{\bigl\langle}
\mathcal{M}_{\scd\scp}(\bv{u}_{\s}; \xi_{R_{\s}}) \bv{w}, \bv{v}
\big\rangle_{\mathcal{V}^{\bv{v}_{\f}}} 
&\coloneqq
\int_{B_{\s}} (J_{\s} - \xi_{R_{\s}}) \rho^{*}_{\f} \bv{w} \cdot \bv{v}, 
\nonumber
\\
& & & \qquad\qquad
\forall \bv{w} \in \mathcal{H}^{\bv{v}_{\f}}, \forall \bv{v} \in \mathcal{V}^{\bv{v}_{\f}},
\\
\label{eq: N33 def}
\mathcal{N}_{\scd\scp}(\bv{u}_{\s},\bv{v}_{\f}; \xi_{R_{\s}}) &: \mathcal{V}^{\bv{v}_{\f}} \to (\mathcal{V}^{\bv{v}_{\f}})^{*},
& &
\hspace{-50mm}
\begin{aligned}[t]
&\prescript{}{(\mathcal{V}^{\bv{v}_{\f}})^{*}}{\bigl\langle}
\mathcal{N}_{\scd\scp}(\bv{u}_{\s},\bv{v}_{\f}; \xi_{R_{\s}}) \bv{w}, \bv{v}
\big\rangle_{\mathcal{V}^{\bv{v}_{\f}}} 
\\
&\hspace{35mm}\coloneqq
\int_{B_{\s}} (J_{\s} - \xi_{R_{\s}}) \rho^{*}_{\f} \Grad \bv{v}_{\f}\tensor{F}_{\s}^{-1} \bv{w} \cdot \bv{v}, 
\end{aligned}
\nonumber
\\
& & & \qquad\qquad\qquad
\forall \bv{w},\bv{v} \in \mathcal{V}^{\bv{v}_{\f}},
\\
\label{eq: N32 def}
\mathcal{N}_{\scd\sbp}(\bv{u}_{\s},\bv{v}_{\f}; \xi_{R_{\s}}) &: \mathcal{V}^{\bv{v}_{\s}} \to (\mathcal{V}^{\bv{v}_{\f}})^{*},
& &
\hspace{-50mm}
\begin{aligned}[t]
&\prescript{}{(\mathcal{V}^{\bv{v}_{\f}})^{*}}{\bigl\langle}
\mathcal{N}_{\scd\sbp}(\bv{u}_{\s},\bv{v}_{\f}; \xi_{R_{\s}}) \bv{w}, \bv{v}
\big\rangle_{\mathcal{V}^{\bv{v}_{\f}}} 
\\
&\hspace{35mm}\coloneqq
\int_{B_{\s}} (J_{\s} - \xi_{R_{\s}}) \rho^{*}_{\f} \Grad \bv{v}_{\f}\tensor{F}_{\s}^{-1} \bv{w} \cdot \bv{v}, 
\end{aligned}
\nonumber
\\
& & & \qquad\qquad
\forall \bv{w} \in \mathcal{V}^{\bv{v}_{\s}}, \forall \bv{v} \in \mathcal{V}^{\bv{v}_{\f}},
\\
\label{eq: D33 def}
\mathcal{D}_{\scd\scp}(\bv{u}_{\s}; \xi_{R_{\s}}) &: \mathcal{V}^{\bv{v}_{\f}} \to (\mathcal{V}^{\bv{v}_{\f}})^{*},
&\quad
\prescript{}{(\mathcal{V}^{\bv{v}_{\f}})^{*}}{\bigl\langle}
\mathcal{D}_{\scd\scp}(\bv{u}_{\s}; \xi_{R_{\s}}) \bv{w}, \bv{v}
\big\rangle_{\mathcal{V}^{\bv{v}_{\f}}} 
&\coloneqq
\int_{B_{\s}} (J_{\s} - \xi_{R_{\s}})^{2} \frac{\mu_{\f}}{J_{\s}k_{\s}} \bv{w} \cdot \bv{v}, 
\nonumber
\\
& & & \qquad\qquad
\forall \bv{w},\bv{v} \in \mathcal{V}^{\bv{v}_{\f}},
\\
\label{eq: D32 def}
\mathcal{D}_{\scd\sbp}(\bv{u}_{\s}; \xi_{R_{\s}}) &: \mathcal{V}^{\bv{v}_{\s}} \to (\mathcal{V}^{\bv{v}_{\f}})^{*},
&\quad
\prescript{}{(\mathcal{V}^{\bv{v}_{\f}})^{*}}{\bigl\langle}
\mathcal{D}_{\scd\sbp}(\bv{u}_{\s}; \xi_{R_{\s}}) \bv{w}, \bv{v}
\big\rangle_{\mathcal{V}^{\bv{v}_{\f}}} 
&\coloneqq
\int_{B_{\s}} (J_{\s} - \xi_{R_{\s}})^{2} \frac{\mu_{\f}}{J_{\s}k_{\s}} \bv{w} \cdot \bv{v}, 
\nonumber
\\
& & & \qquad\qquad
\forall \bv{w} \in \mathcal{V}^{\bv{v}_{\s}}, \forall \bv{v} \in \mathcal{V}^{\bv{v}_{\f}},
\\
\label{eq: D32T def}
\transpose{\mathcal{D}}_{\scd\sbp}(\bv{u}_{\s}; \xi_{R_{\s}}) &: \mathcal{V}^{\bv{v}_{\f}} \to (\mathcal{V}^{\bv{v}_{\s}})^{*},
&\quad
\prescript{}{(\mathcal{V}^{\bv{v}_{\s}})^{*}}{\bigl\langle}
\transpose{\mathcal{D}}_{\scd\sbp}(\bv{u}_{\s}; \xi_{R_{\s}}) \bv{v}, \bv{w}
\big\rangle_{\mathcal{V}^{\bv{v}_{\s}}} 
&\coloneqq
\int_{B_{\s}} (J_{\s} - \xi_{R_{\s}})^{2} \frac{\mu_{\f}}{J_{\s}k_{\s}} \bv{v} \cdot \bv{w}, 
\nonumber
\\
& & & \qquad\qquad
\forall \bv{w} \in \mathcal{V}_{0}^{\bv{v}_{\s}}, \forall \bv{v} \in \mathcal{V}^{\bv{v}_{\f}},
\\
\label{eq: D22 def}
\mathcal{D}_{\sbd\sbp}(\bv{u}_{\s}; \xi_{R_{\s}}) &: \mathcal{V}^{\bv{v}_{\s}} \to (\mathcal{V}^{\bv{v}_{\s}})^{*},
&\quad
\prescript{}{(\mathcal{V}^{\bv{v}_{\s}})^{*}}{\bigl\langle}
\mathcal{D}_{\sbd\sbp}(\bv{u}_{\s}; \xi_{R_{\s}}) \bv{v}, \bv{w}
\big\rangle_{\mathcal{V}^{\bv{v}_{\s}}} 
&\coloneqq
\int_{B_{\s}} (J_{\s} - \xi_{R_{\s}})^{2} \frac{\mu_{\f}}{J_{\s}k_{\s}} \bv{v} \cdot \bv{w}, 
\nonumber
\\
& & & \qquad\qquad
\forall \bv{w} \in \mathcal{V}_{0}^{\bv{v}_{\s}}, \forall \bv{v} \in \mathcal{V}^{\bv{v}_{\s}},
\\
\label{eq: M22 def}
\mathcal{M}_{\sbd\sbp}(\xi_{R_{\s}}) &: \mathcal{H}^{\bv{v}_{\s}} \to (\mathcal{V}^{\bv{v}_{\s}})^{*},
&\quad
\prescript{}{(\mathcal{V}^{\bv{v}_{\s}})^{*}}{\bigl\langle}
\mathcal{M}_{\sbd\sbp}(\xi_{R_{\s}}) \bv{v}, \bv{w}
\big\rangle_{\mathcal{V}^{\bv{v}_{\s}}} 
&\coloneqq
\int_{B_{\s}}\xi_{R_{\s}} \rho_{\s}^{*} \bv{v} \cdot \bv{w}, 
\nonumber
\\
& & & \qquad\qquad
\forall \bv{w} \in \mathcal{V}_{0}^{\bv{v}_{\s}}, \forall \bv{v} \in \mathcal{H}^{\bv{v}_{\s}},
\\
\label{eq: A21 def}
\mathcal{A}_{\sbd}(\bv{u};\xi_{R_{\s}}) &\in (\mathcal{V}^{\bv{v}_{\s}})^{*},
&\quad
\prescript{}{(\mathcal{V}^{\bv{v}_{\s}})^{*}}{\bigl\langle}
\mathcal{A}_{\sbd}(\bv{u};\xi_{R_{\s}}), \bv{w}
\big\rangle_{\mathcal{V}^{\bv{v}_{\s}}} 
&\coloneqq
\int_{B_{\s}} \tensor{P}^{e}[\bv{u}] \colondot \Grad \bv{w}, 
\nonumber
\\
& & & \qquad\qquad
\forall \bv{w} \in \mathcal{V}_{0}^{\bv{v}_{\s}}, \forall \bv{u} \in \mathcal{V}^{\bv{u}_{\s}},
\\
\label{eq: S42 def}
\mathcal{S}_{\sdd\sbp}(\bv{u}_{\s}) &:  \mathcal{V}^{\bv{v}_{\s}} \to (\mathcal{V}^{p})^{*},
&\quad
\prescript{}{(\mathcal{V}^{p})^{*}}{\bigl\langle}
\mathcal{S}_{\sdd\sbp}(\bv{u}_{\s}) \bv{v}, q
\big\rangle_{\mathcal{V}^{p}} 
&\coloneqq
\int_{\Gamma^{N}_{\s}} J_{\s} \tensor{F}_{\s}^{-1} \bv{v} \cdot \bv{n}_{\s} q, 
\nonumber
\\
& & & \qquad\qquad
\forall \bv{v} \in \mathcal{V}^{\bv{v}_{\s}}, \forall q \in \mathcal{V}^{p},
\\
\label{eq: S42T def}
\transpose{\mathcal{S}}_{\sdd\sbp}(\bv{u}_{\s}) &:  \mathcal{V}^{p} \to (\mathcal{V}^{\bv{v}_{\s}})^{*},
&\quad
\prescript{}{(\mathcal{V}^{\bv{v}_{\s}})^{*}}{\bigl\langle}
\transpose{\mathcal{S}}_{\sdd\sbp}(\bv{u}_{\s}) q, \bv{v}
\big\rangle_{\mathcal{V}^{\bv{v}_{\s}}} 
&\coloneqq
\int_{\Gamma^{N}_{\s}} J_{\s} \tensor{F}_{\s}^{-1} \bv{v} \cdot \bv{n}_{\s} q, 
\nonumber
\\
& & & \qquad\qquad
\forall \bv{v} \in \mathcal{V}_{0}^{\bv{v}_{\s}}, \forall q \in \mathcal{V}^{p},
\\
\label{eq: F4 def}
\mathcal{F}_{\sdd}(\bv{u}_{\s}) &\in (\mathcal{V}^{p})^{*},
&\quad
\prescript{}{(\mathcal{V}^{p})^{*}}{\bigl\langle}
\mathcal{F}_{\sdd}(\bv{u}_{\s}), q
\big\rangle_{\mathcal{V}^{p}} 
&\coloneqq
-\int_{\Gamma^{D}_{\s}} J_{\s} \tensor{F}_{\s}^{-1} \bar{\bv{v}}_{\s} \cdot \bv{n}_{\s} q, 
\nonumber
\\
& & & \qquad\qquad
\forall q \in \mathcal{V}^{p}
\\
\label{eq: F3 def}
\mathcal{F}_{\scd}(\bv{u}_{\s}; \xi_{R_{\s}}) &\in (\mathcal{V}^{\bv{v}_{\f}})^{*},
&\quad
\prescript{}{(\mathcal{V}^{\bv{v}_{\f}})^{*}}{\bigl\langle}
\mathcal{F}_{\scd}(\bv{u}_{\s}; \xi_{R_{\s}}), \bv{v}
\big\rangle_{\mathcal{V}^{\bv{v}_{\f}}} 
&\coloneqq
\int_{B_{\s}} (J_{\s} - \xi_{R_{\s}}) \rho^{*}_{\f} \bv{b}_{\f} \cdot \bv{v} , 
\nonumber
\\
& & & \qquad
\forall \bv{b}_{\f} \in H^{-1}(B_{\s}), \forall \bv{v} \in \mathcal{V}^{\bv{v}_{\f}},
\\
\label{eq: F2 def}
\mathcal{F}_{\sbd}(\xi_{R_{\s}}) &\in (\mathcal{V}^{\bv{v}_{\s}})^{*},
&\quad
\prescript{}{(\mathcal{V}^{\bv{v}_{\s}})^{*}}{\bigl\langle}
\mathcal{F}_{\sbd}(\xi_{R_{\s}}), \bv{v}
\big\rangle_{\mathcal{V}^{\bv{v}_{\s}}} 
&\coloneqq
\begin{aligned}[t]
&\int_{B_{\s}} \xi_{R_{\s}} \rho^{*}_{\s} \bv{b}_{\s} \cdot \bv{v} , 
+\int_{\Gamma_{\s}^{N}} \bar{\bv{s}} \cdot \bv{v}
\end{aligned}
\nonumber
\\
& & &
\hspace{-20mm}
\forall \bv{b}_{\s} \in H^{-1}(B_{\s}), \forall \bar{\bv{s}} \in H^{-\frac{1}{2}}\bigl(\Gamma_{\s}^{N}\bigr), \forall \bv{v} \in \mathcal{V}^{\bv{v}_{\s}}_{0},
\\
\label{eq: M33 check def}
\check{\mathcal{M}}_{\scd\scp}(\us;\xi_{R_{\s}}) &:  \mathcal{V}^{\vf} \to (\mathcal{V}^{\vf})^{*},
&\quad
\prescript{}{(\mathcal{V}^{\vf})^{*}}{\bigl\langle}
\check{\mathcal{M}}_{\scd\scp}(\us;\xi_{R_{\s}}) \bv{v}, \bv{w}
\big\rangle_{\mathcal{V}^{\vf}} 
&\coloneqq
\begin{aligned}[t]
&\int_{B_{\s}} (J_{\s} - \xi_{R_{\s}}) \bv{v} \cdot \bv{w}, 
\end{aligned}
\nonumber
\\
& & &
\forall \bv{v}, \bv{w} \in \mathcal{V}^{\vf},
\\
\label{eq: M32 check def}
\check{\mathcal{M}}_{\scd\sbp}(\us;\xi_{R_{\s}}) &:  \mathcal{V}^{\vs} \to (\mathcal{V}^{\vf})^{*},
&\quad
\prescript{}{(\mathcal{V}^{\vf})^{*}}{\bigl\langle}
\check{\mathcal{M}}_{\scd\sbp}(\us;\xi_{R_{\s}}) \bv{v}, \bv{w}
\big\rangle_{\mathcal{V}^{\vf}} 
&\coloneqq
\begin{aligned}[t]
&\int_{B_{\s}} (J_{\s} - \xi_{R_{\s}}) \bv{v} \cdot \bv{w}, 
\end{aligned}
\nonumber
\\
& & &
\forall \bv{v} \in \mathcal{V}^{\vs}, \forall \bv{w} \in \mathcal{V}^{\vf}.
\end{alignat}

\subsubsection*{Additional operators used to define Problem~\ref{dual pb abstract weak ALE stabilized static}}
In setting up Problem~\ref{dual pb abstract weak ALE stabilized static} the following additional operators are defined:
\begin{gather}
\label{eq: F2 tilde def}
\begin{multlined}
\tilde{\mathcal{F}}_{\sbd}(\bv{u}_{\s}; \xi_{R_{\s}}) \in (\mathcal{V}^{\bv{v}_{\s}})^{*},
\quad
\prescript{}{(\mathcal{V}^{\bv{v}_{\s}})^{*}}{\bigl\langle}
\tilde{\mathcal{F}}_{\sbd}(\bv{u}_{\s}; \xi_{R_{\s}}), \bv{v}
\big\rangle_{\mathcal{V}^{\bv{v}_{\s}}} 
\coloneqq
-\int_{B_{\s}} (J_{\s} - \xi_{R_{\s}}) \rho^{*}_{\f} \bv{w} \cdot \bv{v}, 
\\
\forall \bv{w} \in H^{-1}(B_{\s}), \forall \bv{v} \in \mathcal{V}_{0}^{\bv{v}_{\s}},
\end{multlined}
\\
\label{eq: K44 def}
\begin{multlined}
\mathcal{K}_{\sdd\sdp}(\bv{u}_{\s}) : \mathcal{V}^{p} \to (\mathcal{V}^{p})^{*},
\quad
\prescript{}{(\mathcal{V}^{p})^{*}}{\bigl\langle}
\mathcal{K}_{\sdd\sdp}(\bv{u}_{\s}) p, q
\big\rangle_{\mathcal{V}^{p}} 
\coloneqq
\int_{B_{\s}} \frac{J_{\s}k_{\s}}{\mu_{\f}}  \tensor{C}_{\s}^{-1} \Grad p  \cdot \Grad q,
\\
\forall p,q \in \mathcal{V}^{p}.
\end{multlined}
\\
\label{eq: F4 tilde def}
\begin{multlined}
\tilde{\mathcal{F}}_{\sdd}(\bv{u}_{\s}) \in (\mathcal{V}^{p})^{*},
\quad
\prescript{}{(\mathcal{V}^{p})^{*}}{\bigl\langle}
\tilde{\mathcal{F}}_{\sdd}(\bv{u}_{\s}) - \mathcal{F}_{\sdd}(\bv{u}_{\s}), q
\big\rangle_{\mathcal{V}^{p}} 
\coloneqq
\int_{B_{\s}} \frac{J_{\s}k_{\s} \rho^{*}_{\f}}{\mu_{\f}}  \tensor{F}_{\s}^{-1} \bv{v}  \cdot \Grad q, 
\\
\forall \bv{v} \in H^{-1}(B_{\s}), \forall q \in \mathcal{V}^{p},
\end{multlined}
\end{gather}

\subsubsection*{Additional operators used to define Problem~\ref{dual pb abstract weak ALE fv}}
\begin{gather}
\label{eq: M23 def}
\begin{multlined}
\mathcal{M}_{\sbd\scp}(\bv{u}_{\s}; \xi_{R_{\s}}) : \mathcal{H}^{\vf} \to (\mathcal{V}^{\vs})^{*},
\quad
\prescript{}{(\mathcal{V}^{\vs})^{*}}{\bigl\langle}
\mathcal{M}_{\sbd\scp}(\bv{u}_{\s}; \xi_{R_{\s}}) \bv{v}, \bv{w}
\big\rangle_{\mathcal{V}^{\vs}} 
\coloneqq
\int_{B_{\s}} (J_{\s} - \xi_{R_{\s}}) \rho_{\f}^{*} \bv{v} \cdot \bv{w},
\\
\forall \bv{w} \in \mathcal{V}^{\vs}_{0}, \forall \bv{v} \in \mathcal{H}^{\vf}.
\end{multlined}
\\
\label{eq: N23 def}
\begin{multlined}
\mathcal{N}_{\sbd\scp}(\bv{u}_{\s}, \vf; \xi_{R_{\s}}) : \mathcal{V}^{\vf} \to (\mathcal{V}^{\vs})^{*},
\quad
\prescript{}{(\mathcal{V}^{\vs})^{*}}{\bigl\langle}
\mathcal{N}_{\sbd\scp}(\bv{u}_{\s}, \vf; \xi_{R_{\s}}) \bv{v}, \bv{w}
\big\rangle_{\mathcal{V}^{\vs}}
\\
\hspace{90mm}\coloneqq
\int_{B_{\s}} (J_{\s} - \xi_{R_{\s}}) \rho_{\f}^{*} \Grad \vf \tensor{F}_{\s}^{-1} \bv{v} \cdot \bv{w},
\\
\forall \bv{w} \in \mathcal{V}^{\vs}_{0}, \forall \bv{v} \in \mathcal{V}^{\vf}.
\end{multlined}
\\
\label{eq: N22 def}
\begin{multlined}
\mathcal{N}_{\sbd\sbp}(\bv{u}_{\s}, \vf; \xi_{R_{\s}}) : \mathcal{V}^{\vs} \to (\mathcal{V}^{\vs})^{*},
\quad
\prescript{}{(\mathcal{V}^{\vs})^{*}}{\bigl\langle}
\mathcal{N}_{\sbd\sbp}(\bv{u}_{\s}, \vf; \xi_{R_{\s}}) \bv{v}, \bv{w}
\big\rangle_{\mathcal{V}^{\vs}}
\\
\hspace{90mm}\coloneqq
\int_{B_{\s}} (J_{\s} - \xi_{R_{\s}}) \rho_{\f}^{*} \Grad \vf \tensor{F}_{\s}^{-1} \bv{v} \cdot \bv{w},
\\
\forall \bv{w} \in \mathcal{V}^{\vs}_{0}, \forall \bv{v} \in \mathcal{V}^{\vs}.
\end{multlined}
\end{gather}

}

\bibliographystyle{chicago}
\bibliography{CMPoroelastic}

\end{document}